\documentclass[11pt]{amsart}
\usepackage{amssymb,amsmath,amsthm}
\usepackage{amsfonts}
\usepackage{enumerate}
\usepackage{dsfont}
\usepackage{cite}
\usepackage{color} 
\usepackage[colorlinks,citecolor=red]{hyperref}
\usepackage{cleveref}
\usepackage{mathrsfs}
\usepackage{epsfig}
\usepackage{lscape}
\usepackage{subfigure}
\usepackage{epstopdf}
\usepackage{caption}
\usepackage{algorithm}
\usepackage{algpseudocode}
\usepackage{multirow}
\usepackage{geometry}
\usepackage{paralist}
\usepackage{enumerate}
\usepackage{url}
\usepackage{graphicx,cite}
\usepackage{booktabs}   
\usepackage{makecell}   
\usepackage{array}      
\usepackage{makecell}
\usepackage{longtable}
\usepackage{tabularx} 
\usepackage{tikz}
\usetikzlibrary{arrows.meta, positioning, decorations.pathreplacing, calc,angles,quotes}

\usepackage{bbm}

\newcommand{\argmin}[1]{\mathop{\rm argmin}\limits_{#1}}

\newtheorem{definition}{Definition}[section]

\newtheorem{prop}[definition]{Proposition}
\newtheorem{theorem}[definition]{Theorem}
\newtheorem{lemma}[definition]{Lemma}

\newtheorem{remark}[definition]{Remark}

\newtheorem{assumption}{Assumption}[section]
\date{}

\begin{document}
\newpage
\baselineskip 18pt
\title
[Subspace-constrained preconditioning for RIM]
{On subspace-constrained preconditioning for randomized iterative methods}
\author{Yonghan Sun}
\address{School of Mathematical Sciences, Beihang University, Beijing, 100191, China.}
\email{sunyonghan@buaa.edu.cn}

\author{Hou-Duo Qi}
\address{Department of Data Science and Artificial Intelligence and Department of Applied Mathematics, The Hong Kong Polytechnic University, Hung Hom, Kowloon, Hong Kong, China.}
\email{houduo.qi@polyu.edu.hk}

\author{Deren Han}
\address{LMIB of the Ministry of Education, School of Mathematical Sciences, Beihang University, Beijing, 100191, China.}
\email{handr@buaa.edu.cn}

\author{Jiaxin Xie}
\address{LMIB of the Ministry of Education, School of Mathematical Sciences, Beihang University, Beijing, 100191, China.}
\email{xiejx@buaa.edu.cn}

\begin{abstract}
In this paper, we further investigate and refine the subspace-constrained preconditioning technique to enhance the theoretical and numerical convergence properties of randomized iterative methods for solving linear systems. In particular, we design a QR-like factorization that transforms the original linear system into an equivalent block-orthogonal form, thus avoiding the full-rank assumptions adopted in existing work. Moreover, this reformulation reduces the problem to solving a smaller linear system with a favorable singular value distribution, provided an appropriate initial point is employed. The proposed framework can be implemented implicitly within the iteration and does not require explicitly constructing either a preconditioner matrix or a preconditioned linear system, which eliminates the prohibitive cost of forming a fully preconditioned system. Furthermore, we construct orthogonalized search directions from stochastic gradients and develop accelerated variants of the framework. We prove that the proposed algorithmic framework converges linearly in expectation. Numerical experiments demonstrate the benefits of the proposed preconditioning strategy.
\end{abstract}
\maketitle

\let\thefootnote\relax\footnotetext{Key words: linear systems, subspace-constrained preconditioner, randomized iterative methods, iterative sketching, Kaczmarz method, Krylov subspace methods}
\let\thefootnote\relax\footnotetext{Mathematics subject classification (2020): 65F10, 65F08, 65F20, 90C25, 15A06, 68W20}

\section{Introduction}

Randomized iterative methods \cite{Needell2016Stochastic,necoara2019faster,strohmer2009randomized,Zeng2024adaptive,Gower2015Randomized,sun2025connecting} form a powerful class of algorithms for solving the large-scale linear system
\begin{equation}\label{main-prob}
	Ax = b,\ A \in \mathbb{R}^{m \times n},\ b \in \mathbb{R}^m.
\end{equation}
These methods can be regarded as iterative sketching approaches in the randomized numerical linear algebra paradigm \cite{Derezinski2024Recent,Martinsson2020Randomized}, and also serve as extensions of stochastic approximation methods from optimization, including stochastic gradient descent (SGD) \cite{Needell2016Stochastic,Robbins1951stochastic,Ma2019Stochastic}.  However, such methods may converge slowly for ill-conditioned coefficient matrices, and therefore a preconditioner is needed \cite{avron2010blendenpik,meier2024sketch,meng2014lsrn,benzi2002preconditioning,epperly2026fast,lok2025subspace}.  

The randomized Kaczmarz (RK) method \cite{strohmer2009randomized} is one of the most widely used iterative methods for solving large-scale linear systems in recent years due to its simplicity and efficiency. Recently, Lok and Rebrova \cite{lok2024subspace} incorporated subspace constraints into the RK method, developing the subspace-constrained RK (SCRK) method and demonstrating that it achieves superior numerical efficiency over the standard RK method.
They further extended this technique to the randomized block Kaczmarz (RBK) method \cite{needell2014paved}, however, under the assumption that the sampled submatrices possess full row rank \cite[Remark 3.9]{lok2024subspace}.
 In this paper, we extend this subspace-constrained preconditioning technique to a broader class of randomized iterative methods. Our analysis removes the full row rank restriction on sampled submatrices and generalizes the subspace-constrained preconditioner to arbitrary linear systems. 

 We next briefly illustrate the generalization from the standard RK method to the general randomized iterative method adopted in this work. For any \(i \in [m]:=\{1,\ldots,m\}\), let \(e_i\) denote the \(i\)-th unit vector, \(\top\) denotes the transpose, \(a^\top_{i}\) denote the \(i\)-th row of \(A\), \(b_i\) denote the \(i\)-th entry of \(b\), and let \(\|\cdot\|_2\) and \(\|\cdot\|_F\) denote the Euclidean and Frobenius norms, respectively. 
 The RK method \cite{strohmer2009randomized}  employs the following update formula
 \begin{equation}
 	\label{RK-iteration}
 	x^{k+1} = x^k - \frac{\langle a_{i_k}, x^k\rangle - b_{i_k}}{\|a_{i_k}\|_2^2} a_{i_k}=x^k-\alpha_k A^\top e_{i_k} e_{i_k}^\top (Ax^k-b),
 \end{equation}
 where  the \(i_k\)-th row is selected with probability \(\|a_{i}\|_2^2 / \|A\|_F^2\),  and \(\alpha_k = 1 / \|a_{i_k}\|_2^2\) is the stepsize.
 Noting that $e_{i_k}$ in \eqref{RK-iteration} acts as the role that extracts partial information of $A$, we extend $e_{i_k}$ to a general randomized sketching matrix $S_k\in \mathbb{R}^{m\times q}$ and apply the following iteration scheme
 \begin{equation}
 	\label{RIM-iteration}
 	x^{k+1} =x^k-\alpha_k A^\top S_{k} S_{k}^\top (Ax^k-b).
 \end{equation}
 The  matrix $S_k$ is sampled from a user-defined probability space ${(\Omega,\mathcal{F},\mathbf{P})}$. This is precisely the randomized iterative method proposed in \cite{Zeng2024adaptive,xie2025randomized} for solving linear systems.  We note that by choosing the probability space $(\Omega, \mathcal{F}, \mathbf{P})$ in \eqref{RIM-iteration} appropriately, one can recover existing methods or develop new ones \cite[Section 5]{xie2025randomized}, such as the RBK method \cite{needell2014paved} and the randomized average block Kaczmarz (RABK) method \cite{necoara2019faster}.

\subsection{Our contributions}
For any index subset $\mathcal{I}\subseteq[m]$, we use $A_{\mathcal{I}}$ and $b_{\mathcal{I}}$ to denote the row submatrix of $A$ and the subvector of $b$ indexed by $\mathcal{I}$, respectively. Let $\mathcal{I}_p \subset [m] $ be a fixed index set and  $\mathcal{I}_r = [m] \setminus \mathcal{I}_p$. The key idea of the subspace-constrained preconditioner is that we enforce the subsystem $A_{\mathcal{I}_p}x=b_{\mathcal{I}_p}$ to hold exactly throughout the iteration. That is, every iterate $x^k$ is required to satisfy $A_{\mathcal{I}_p}x^k=b_{\mathcal{I}_p}$. This strategy of fixing the row block $A_{\mathcal{I}_p}$ is particularly beneficial when such a block possesses favorable structural or spectral properties \cite{xie2021subset}.

 The main contributions of this work are as follows.
 
\begin{itemize}
	\item [1. ] We introduce a QR-like factorization of the coefficient matrix \(A^\top\) based on the fixed submatrix \(A_{\mathcal{I}_p}\). This factorization allows us not only to transform the original linear system into an equivalent preconditioned system, but also to avoid the full-rank assumptions adopted in \cite[Remark 3.9]{lok2024subspace}. Thus, it further extends  the applicability of the subspace-constrained preconditioner to any type of linear system, including underdetermined, overdetermined, full-rank, and rank-deficient matrices.
	\item [2. ] We introduce a projected SGD  to solve the preconditioned system, where the projection enforces the subsystem \(A_{\mathcal{I}_p}x = b_{\mathcal{I}_p}\) to hold exactly throughout the iteration. We show that, when the initial point is chosen appropriately, the proposed method is equivalent to applying the randomized iterative method \eqref{RIM-iteration} with preconditioner \(P := I - A_{\mathcal{I}_p}^\dagger A_{\mathcal{I}_p}\) to solve the reduced linear system
	$
	A_{\mathcal{I}_r}  x = b_{\mathcal{I}_r},
	$
	where  \(A_{\mathcal{I}_p}^\dagger\) denotes the pseudoinverse of \(A_{\mathcal{I}_p}\). This equivalence enables parallel implementation and greatly reduces computational cost, while the smaller matrix \(A_{\mathcal{I}_r} P\) enjoys a more favorable singular value distribution than the original coefficient matrix \(A\).
	
	\item [3. ] We prove that the proposed subspace-constrained randomized iterative method (SCRIM) converges linearly in expectation, with a convergence factor that depends on both the choice of the probability space and the reduced matrix \(A_{\mathcal{I}_r} P\). We prove that the subspace-constrained preconditioning yields a tighter convergence factor than its unconstrained counterpart. In addition, we establish a decoupled upper bound on the convergence factor and, based on this bound, provide practical strategies for selecting the constraint subspace.
	\item [4. ] We then replace the stochastic gradient direction used by SCRIM with orthogonalized search directions, motivated by the success of the orthogonal direction method~\cite{shewchuk1994introduction}. Specifically, each descent direction is constructed by applying a truncated Gram-Schmidt orthogonalization process to the stochastic gradients. We demonstrate that the resulting framework constitutes an iterative-sketching-based Krylov subspace (IS-Krylov) method with subspace-constrained preconditioning, and we prove that the proposed SC-IS-Krylov method achieves a tighter  upper bound for convergence. Numerical experiments validate the theoretical findings and demonstrate the efficiency of the proposed method.
\end{itemize}

\subsection{Related work}
\subsubsection{The Kaczmarz method}

The literature on solving linear systems via iterative methods is vast and has a long history \cite{Saad2003Iterative,kelley1995iterative,scott2025sparse}. Row-action methods, for instance, have been proposed as practical approaches for solving large-scale linear systems. These methods access and process one or a few rows of \(A\) at a time. A prominent example of the row-action method is the Kaczmarz method \cite{Kac37} and its randomized variant \cite{strohmer2009randomized}. In recent years, a large amount of work has been devoted to the development of Kaczmarz-type methods, including block Kaczmarz methods \cite{necoara2019faster,needell2014paved,xie2025randomized,Gower2015Randomized,Xiang2025Randomized}, accelerated RK methods \cite{sun2025connecting,Zeng2024adaptive,Rieger2023Generalized,Liu2016accelerated,Loizou2020Momentum,han2026pseudoinverse,wang2026linear,su2024greedy}, randomized Douglas-Rachford methods \cite{han2024randomized,guo2025enhanced}, and others.

Very recently, Lok and Rebrova \cite{lok2024subspace} introduced the subspace-constrained randomized Kaczmarz (SCRK) method and its block extensions, as briefly reviewed in Remarks  \ref{remark-SCRBK} and \ref{remark-SCRK1}. The core idea of SCRK is to enforce the subsystems $A_{\mathcal{I}_p}x = b_{\mathcal{I}_p}$ to hold exactly at all iterations, such that every generated iterate $x^k$ strictly satisfies $A_{\mathcal{I}_p}x^k = b_{\mathcal{I}_p}$.  In this paper, we extend this line of research by proposing a subspace-constrained preconditioner for a broad range of randomized iterative methods \eqref{RIM-iteration}. Furthermore, based on a QR-like factorization, our method does not require the full-row-rank assumption, which significantly broadens and enriches the results presented in \cite{lok2024subspace}.

\subsubsection{Preconditioned SGD}

SGD \cite{Robbins1951stochastic,nemirovski2009robust,gower2019sgd,garrigos2023handbook} has recently gained popularity for solving large-scale optimization problems of the generic form
\(
\min_x \, F(x) := \mathbb{E}[f(x;\xi)],
\)
where $\xi$ is sampled from $(\Omega, \mathcal{F}, \mathbf{P})$. 
However, the convergence of SGD can be slow when the problem is ill-conditioned. To mitigate this issue, one can incorporate a preconditioner that scales the gradient direction, leading to preconditioned SGD
\begin{equation}\label{iteration-pre-SGD}
	x^{k+1} = x^k - \alpha_k M \nabla f(x^k; \xi_k),
\end{equation}
where $M$ is a symmetric positive definite preconditioner, $\alpha_k > 0$ denotes the stepsize, and $\nabla f(x^k;\xi_k)$ stands for the stochastic gradient evaluated at $x^k$ based on an i.i.d. sample $\xi_k$. The standard SGD is recovered when $M = I$.
A substantial body of work has explored preconditioned SGD \cite{ye2024preconditioning,zhang2020gradient,kovalev2025sgd,scott2025designing,li2017preconditioned}, with the design of the preconditioner $M$ broadly falling into two paradigms. The first employs diagonal preconditioners (e.g., Adam \cite{kingma2014adam}, AMSGrad \cite{reddi2019convergence}), which adaptively rescale coordinates using historical second-moment estimates. 
The second paradigm constructs curvature-aware preconditioners by approximating second-order structure \cite{schraudolph2002fast,fletcher2013practical,garg2025second,martens2015optimizing}, such as the Hessian or Fisher information matrix, to improve conditioning.

In many applications, one may require $x$ to belong to a certain constraint set $\mathcal{X}$. This motivates the projected variant of SGD 
\begin{equation}\label{iteration-pro-SGD}
	x^{k+1} = \Pi_{\mathcal{X}} \left( x^k - \alpha_k \nabla f(x^k; \xi_k) \right),
\end{equation}
where $\Pi_{\mathcal{X}}$ denotes the projection onto the constraint set $\mathcal{X}$. We note that the randomized iterative method \eqref{RIM-iteration} can be viewed as SGD for solving the stochastic optimization problem
\(
\min_x \frac{1}{2} \mathbb{E}\!\left[ \|S^\top (A x - b)\|_2^2 \right]
\)
reformulated from the linear system.  Indeed, our proposed subspace-constrained iterative framework admits a dual interpretation as both a preconditioned SGD scheme and a projected SGD scheme; see Section \ref{sec-3}. This unified dual perspective provides a novel paradigm for designing and analyzing preconditioned randomized iterative methods for large-scale linear systems.


\subsubsection{Minimal error methods}

Minimal error methods \cite{weiss1995theoretical,gaul2013framework,weiss1994error} are iterative methods that compute the next iterate $x^k$ by minimizing $\|x - x^*\|_2^2$ over a prescribed search space $\mathcal{L}_k$, i.e.,
\begin{equation}\label{eq-MEM}
	x^{k+1} = \argmin{x \in \mathcal{L}_k} \|x - x^*\|_2^2,
\end{equation}
where $x^*$ denotes a certain solution of the linear system $Ax = b$. This approach directly targets proximity to the true solution, in contrast to residual-minimizing methods such as GMRES \cite{saad1986gmres}, which minimize $\|b - Ax\|_2^2$ instead.
When $A \in \mathbb{R}^{n \times n}$ is nonsingular and the search space is chosen as 
\(\mathcal{L}_k = x^0 + \mathcal{K}_k(A^\top, A^\top r^0)\) with an initial point $x^0$ and initial residual $r^0 = b - A x^0$, the minimal error method \eqref{eq-MEM} reduces to the generalized minimal error (GMERR) method \cite{weiss1994error,ehrig1997gmerr}. Here $\mathcal{K}_k(A^\top, A^\top r^0)$ denotes the Krylov subspace of order $k$, whose definition can be found in \eqref{def-krylov}.
When $A \in \mathbb{R}^{m \times n}$ is possibly rectangular, the minimal error method \eqref{eq-MEM} with search space $\mathcal{L}_k = x^0 + \mathcal{K}_k(A^\top A, A^\top r^0)$ coincides with the conjugate gradient normal equation (CGNE) method \cite{craig1955n}.

Note that the error-minimizing scheme amounts to projecting the true solution $x^*$ onto the subspace $\mathcal{L}_k$ at iteration $k$. It faces two fundamental challenges: (1) $x^*$ is unknown in practice and (2) the projection onto $\mathcal{L}_k$ may be computationally expensive. Classical minimal error methods circumvent the dependence on $x^*$ by constructing $\mathcal{L}_k$ as an affine shift of a Krylov subspace, so that the projection can be reformulated entirely in terms of accessible quantities such as residuals and matrix-vector products. To facilitate efficient computation of the projection, these methods build an orthogonal basis for $\mathcal{L}_k$. A canonical realization of this strategy is the orthogonal direction method \cite{weiss1995theoretical,weiss1994error}, which generates a sequence of mutually orthogonal search directions $p^0, p^1, \dots, p^k$ and updates the iterate by minimizing $\|x - x^*\|_2^2$ along each new direction; see Section \ref{sec-4}.

Inspired by the intrinsic error-minimization principle and orthogonal-direction acceleration mechanism of classical minimal error methods, we further enhance the convergence performance of our subspace-constrained iterative framework. This algorithmic improvement ultimately gives rise to the proposed subspace-constrained iterative-sketching-based Krylov subspace (SC-IS-Krylov) method.

\subsection{Organization}
The remainder of this paper is organized as follows. In Section \ref{sec-2}, we introduce the notation and preliminary results used throughout the paper. 
Section~\ref{sec-3} develops
the subspace-constrained randomized iterative method, and establishes its linear convergence in expectation. 
Section \ref{sec-4} proposes subspace-constrained iterative-sketching-based Krylov methods and analyzes their convergence properties.
Numerical experiments demonstrating the effectiveness of the proposed methods are presented in Section \ref{sec-5}. 
Section~\ref{sec-6} concludes the paper. Proofs of the main results involving lengthy derivations  are provided in the appendices.

\section{Notation and preliminaries}\label{sec-2}	
\subsection{Notations}

	For vector $x\in\mathbb{R}^n$, we use $x_i,x^\top$, and $\|x\|_2$ to denote the $i$-th entry, the transpose, and the Euclidean norm of $x$, respectively. 
	 We denote the linear span of vectors $x^1,\ldots, x^k \in \mathbb{R}^n$ as \( \operatorname{span}\{ x^1,\ldots, x^k \} \).
	 For any matrix $A \in \mathbb{R}^{m \times n}$, we use $a^\top_i$,  $A^\dagger$, $A_{\mathcal{J}}$, $A_{\mathcal{J}}^\dagger$, $A^\top$, $\|A\|_2$, $\|A\|_F$, $\operatorname{rank}(A)$, and $\operatorname{Range}(A)$ to denote the $i$-th row, the Moore-Penrose pseudoinverse,  the row submatrix indexed by $\mathcal{J}$, the pseudoinverse of the row submatrix $A_{\mathcal{J}}$, the transpose, the spectral norm, the Frobenius norm, the rank, and the column space, respectively. We use $\lambda_{\min}(A)$  to denote the smallest eigenvalue of a  symmetric matrix $A$, and $\sigma_{\min}(A)$ to denote the smallest nonzero singular value of a general matrix $A$.  
	 The cardinality of the set $\mathcal{J}$ is denoted by $ \vert \mathcal{J}\vert$.
	For any random variables $\xi$, we use $\mathbb{E}[\xi]$ to denote the expectation of $\xi$. 

\subsection{A QR-like factorization}
\label{sec-3.1}

In this subsection, we introduce a QR-like factorization for $A^\top$. 
We note that this factorization is introduced solely for the purpose of theoretical analysis and algorithm development, it is not explicitly computed in the actual implementation of our method.

Let $\mathcal{I}_p \subset [m]$ be a fixed index set with cardinality $m_p = |\mathcal{I}_p|$, and define $\mathcal{I}_r = [m] \setminus \mathcal{I}_p$ with $m_r= |\mathcal{I}_r|$. Without loss of generality, we assume $\mathcal{I}_p =[m_p]$.
Let $P := I - A_{\mathcal{I}_p}^\dagger A_{\mathcal{I}_p}$. Using the identity 
\(
A_{\mathcal{I}_p}^\top(A_{\mathcal{I}_p}^\dagger)^\top A_{\mathcal{I}_r}^\top+PA_{\mathcal{I}_r}^\top=A_{\mathcal{I}_r}^\top,
\)
we obtain
\begin{equation}\label{eq-Atop-QR-like}
	A^\top=\begin{pmatrix}
		A_{\mathcal{I}_p}^\top & A_{\mathcal{I}_r}^\top
	\end{pmatrix}
	=\underbrace{\begin{pmatrix}
			A_{\mathcal{I}_p}^\top & PA_{\mathcal{I}_r}^\top
	\end{pmatrix}}_{\hat{A}^\top}
	\underbrace{\begin{pmatrix}
			I & (A_{\mathcal{I}_p}^\dagger)^\top A_{\mathcal{I}_r}^\top\\
			0 & I
	\end{pmatrix}}_{L^\top}
	=\hat{A}^\top L^\top.
\end{equation}
Substituting $P=I-A_{\mathcal{I}_p}^\dagger A_{\mathcal{I}_p}$ and employing the property $A_{\mathcal{I}_p}A_{\mathcal{I}_p}^\dagger A_{\mathcal{I}_p}=A_{\mathcal{I}_p}$, we derive
\[
A_{\mathcal{I}_p}(A_{\mathcal{I}_r}P)^\top
= A_{\mathcal{I}_p} P A_{\mathcal{I}_r}^\top
= \bigl(A_{\mathcal{I}_p}-A_{\mathcal{I}_p}A_{\mathcal{I}_p}^\dagger A_{\mathcal{I}_p}\bigr)A_{\mathcal{I}_r}^\top
=0,
\]
which verifies that $\hat{A}^\top$ has orthogonal column blocks. Hence, the factorization $A^\top = \hat{A}^\top L^\top$ in \eqref{eq-Atop-QR-like} decomposes $A^\top$ into the product of a matrix $\hat{A}^\top$ with orthogonal column blocks and a unit upper block triangular matrix $L^\top$. This yields a QR-like factorization \cite[Section 5.2]{golub2013matrix} for $A^\top$, where only the block columns of $\hat{A}^\top$ satisfy the orthogonality condition.

Now, the linear system \eqref{main-prob} can be rewritten as the following preconditioned system
\begin{equation}\label{main-pre-prob}
	\hat{A} x = \hat{b} \ \ \text{with} \ \ \hat{b} := L^{-1} b = 
	\begin{pmatrix}
		b_{\mathcal{I}_p} \\
		b_{\mathcal{I}_r} - A_{\mathcal{I}_r} A_{\mathcal{I}_p}^\dagger b_{\mathcal{I}_p}
	\end{pmatrix}.
\end{equation} 
To solve \eqref{main-pre-prob}, we note that the proposed method in Section \ref{sec-3} (see \eqref{eq-bas-xk+1-2}) is equivalent to applying a randomized iterative method \eqref{RIM-iteration}  to the reduced linear system
\[
 A_{\mathcal{I}_r} P x = \hat{b}_{\mathcal{I}_r}.
\]
We next present a result characterizing the singular value distribution of $A_{\mathcal{I}_r} P$ in comparison with the original coefficient matrix $A$.

\begin{lemma}\label{lem-interlacing}
	Let \( r = \operatorname{rank}(A) \) and \( r_p = \operatorname{rank}(A_{\mathcal{I}_p}) \).
	Then
	\(
	\operatorname{rank}(A_{\mathcal{I}_r} P) = r - r_p,
	\)
	and the following interlacing inequalities hold
	\[
		\sigma_{i+r_p}(A) \leq \sigma_i(A_{\mathcal{I}_r}P) \leq \sigma_i(A), \quad i=1,\ldots,n-r_p,
	\]
	where $\sigma_i(\cdot)$ denotes the $i$-th largest singular value. 
\end{lemma}

\begin{remark}
	\label{remark-sv-distribution}
 Lemma \ref{lem-interlacing}  implies that the  singular values of the reduced matrix $A_{\mathcal I_r}P$ are more concentrated than those of $A$, leading to a more favorable singular value distribution. The narrowed range of singular values also reduces the condition number $\|A_{\mathcal{I}_r}P\|_2/\sigma_{\min}(A_{\mathcal{I}_r}P)$ of $A_{\mathcal{I}_r}P$.
\end{remark}

In general, $A^\dagger\neq \begin{pmatrix}
	A_{\mathcal{I}_p}^\dagger & A_{\mathcal{I}_r}^\dagger
\end{pmatrix} $. However, if $A_{\mathcal{I}_p} A_{\mathcal{I}_r}^\top=0$, the following lemma shows that $A^\dagger= \begin{pmatrix}
A_{\mathcal{I}_p}^\dagger & A_{\mathcal{I}_r}^\dagger
\end{pmatrix} $. Lemma~\ref{prop-ortho} is useful in our argument and we believe it is of independent interest.
\begin{lemma}\label{prop-ortho}
	Let $B\in\mathbb{R}^{m_p\times n}$ and $C\in\mathbb{R}^{m_r\times n}$ be matrices satisfying $BC^\top=0$. Then,
	\[
	\begin{aligned}
		\begin{pmatrix}
			B\\
			C
		\end{pmatrix}^\dagger=\begin{pmatrix}
			B^\dagger&C^\dagger
		\end{pmatrix}.
	\end{aligned}
	\]   
\end{lemma}
\begin{proof} 
Recall that for any matrix $M$, the Moore–Penrose pseudoinverse $M^\dagger$ is defined as the unique matrix satisfying the following four conditions \cite[Section 5.5.2]{golub2013matrix}:
(1) $M M^\dagger M = M$;
(2) $M^\dagger M M^\dagger = M^\dagger$;  
(3) $(M M^\dagger)^\top = M M^\dagger$;  
(4) $(M^\dagger M)^\top = M^\dagger M$.
Substituting $M = \begin{pmatrix} B^\top & C^\top \end{pmatrix}^\top$ and directly verifying the above four identities confirms the desired result.
\end{proof}

The following result shows that the original linear system $Ax=b$ and the preconditioned system  $\hat{A} x = \hat{b}$ share the same unique least-squares solution. 
\begin{lemma}\label{lemma-same-solution-set}
	Let $\hat{A}$ and $\hat{b}$ be defined as in \eqref{eq-Atop-QR-like} and \eqref{main-pre-prob}, respectively. We have  
	\[
	A^\dagger b = \hat{A}^\dagger \hat{b}= A_{\mathcal{I}_p}^\dagger b_{\mathcal{I}_p} + (A_{\mathcal{I}_r}P)^\dagger \hat{b}_{\mathcal{I}_r} \
	\text{ and }\
	A^\dagger A = \hat{A}^\dagger \hat{A}=A_{\mathcal{I}_p}^\dagger A_{\mathcal{I}_p}+(A_{\mathcal{I}_r}P)^\dagger A_{\mathcal{I}_r}P.
	\]
\end{lemma}
\begin{proof} 
	By substituting $B=A_{\mathcal{I}_p}$ and $C=A_{\mathcal{I}_r}P$ in Lemma \ref{prop-ortho}, we obtain $\hat{A}^\dagger=\begin{pmatrix}
		A_{\mathcal{I}_p}^\dagger&(A_{\mathcal{I}_r}P)^\dagger
	\end{pmatrix}$ and $\hat A^\dagger \hat A
	= A_{\mathcal{I}_p}^\dagger A_{\mathcal{I}_p}+(A_{\mathcal{I}_r}P)^\dagger A_{\mathcal{I}_r}P$.
	Since the linear systems $Ax=b$ and $\hat A x=\hat b$ share the same solution set, their unique minimum-norm solutions are identical. Therefore, we derive
	\[
	A^\dagger b=\hat A^\dagger \hat b
	= A_{\mathcal{I}_p}^\dagger b_{\mathcal I_p}+(A_{\mathcal{I}_r}P)^\dagger \hat b_{\mathcal I_r}.
	\]
	Since $L$ is invertible, it holds that $\operatorname{Range}(A^\top)=\operatorname{Range}(\hat{A}^\top L^\top)=\operatorname{Range}(\hat{A}^\top)$. By the uniqueness of orthogonal projections, the projection operators onto $\operatorname{Range}(A^\top)$ and $\operatorname{Range}(\hat{A}^\top)$ coincide. Consequently,
	\[
	A^\dagger A=\hat{A}^\dagger \hat{A}= A_{\mathcal{I}_p}^\dagger A_{\mathcal{I}_p}+(A_{\mathcal{I}_r}P)^\dagger A_{\mathcal{I}_r}P.
	\]
	This completes the proof of this lemma.
\end{proof}

\begin{remark}\label{remark-SCRBK}
	The subspace constrained randomized block Kaczmarz (SCRBK) method in~\cite{lok2024subspace} adopts the following iteration scheme
	\begin{equation}\label{iter-SCRBK}
		\begin{aligned}
			x^{k+1}=x^k-A_{\mathcal{I}_p\cup\mathcal{J}_k}^\dagger\big(A_{\mathcal{I}_p\cup\mathcal{J}_k}x^k-b_{\mathcal{I}_p\cup\mathcal{J}_k}\big),
		\end{aligned}
	\end{equation}
	where $\mathcal{J}_k\subseteq\mathcal{I}_r$ and $|\mathcal{J}_k|\geq1$. From Lemma~\ref{lemma-same-solution-set}, we have
	\(
	A_{\mathcal{I}_p \cup \mathcal{J}_k}^\dagger A_{\mathcal{I}_p \cup \mathcal{J}_k} 
	= \hat{A}_{\mathcal{I}_p \cup \mathcal{J}_k}^\dagger \hat{A}_{\mathcal{I}_p \cup \mathcal{J}_k}\) and
	\(A_{\mathcal{I}_p \cup \mathcal{J}_k}^\dagger b_{\mathcal{I}_p \cup \mathcal{J}_k} 
	= \hat{A}_{\mathcal{I}_p \cup \mathcal{J}_k}^\dagger \hat{b}_{\mathcal{I}_p \cup \mathcal{J}_k}.
	\)
	Accordingly, the iteration \eqref{iter-SCRBK} can be equivalently reformulated as
	\[
	\begin{aligned}
		x^{k+1}
		&=x^k-A_{\mathcal{I}_p\cup\mathcal{J}_k}^\dagger\big(A_{\mathcal{I}_p\cup\mathcal{J}_k}x^k-b_{\mathcal{I}_p\cup\mathcal{J}_k}\big) \\
		&=x^k-\hat{A}_{\mathcal{I}_p\cup\mathcal{J}_k}^\dagger\big(\hat{A}_{\mathcal{I}_p\cup\mathcal{J}_k}x^k-\hat{b}_{\mathcal{I}_p\cup\mathcal{J}_k}\big) \\
		&=x^k-A_{\mathcal{I}_p}^\dagger\big(A_{\mathcal{I}_p}x^k-b_{\mathcal{I}_p}\big)-\big(A_{\mathcal{J}_k}P\big)^\dagger \big(A_{\mathcal{J}_k}Px^k-\hat{b}_{\mathcal{J}_k}\big),
	\end{aligned}
	\]
	where the last equality follows from Lemma \ref{prop-ortho} that $\hat{A}_{\mathcal{I}_p\cup\mathcal{J}_k}^\dagger=\big(A_{\mathcal{I}_p}^\dagger \ \ (A_{\mathcal{J}_k}P)^\dagger\big)$. This iteration scheme is consistent with that presented in \cite[Remark 3.9]{lok2024subspace}, which is originally derived under the full row rank assumption on $A_{\mathcal{I}_p\cup\mathcal{J}_k}$. In contrast, based on the QR-like factorization established in this subsection, we demonstrate that such a full row rank assumption is actually unnecessary.
\end{remark}

\subsection{Preliminaries on the probability space}

First, we state a basic assumption on the probability space $(\Omega, \mathcal{F}, \mathbf{P})$ used throughout this paper.
 
\begin{assumption}
	\label{assumption1}
	Let $(\Omega, \mathcal{F}, \mathbf{P})$ be a probability space where the sample space $\Omega$ is a set of sketching matrices, and let $S$ denote the random matrix drawn from it. We assume $\mathbb{E}[SS^\top]$ is positive definite.
\end{assumption} 
The following lemmas are crucial for our convergence analysis.
\begin{lemma}[\cite{Zeng2024adaptive}, Lemma 2.3]
	\label{lemma-meanwhile}
	Assume that the linear system $A_{\mathcal{I}_r}Px=\hat{b}_{\mathcal{I}_r}$ is consistent. Then for any matrix $S \in \mathbb{R}^{m_r\times q}$ and any vector $\tilde{x} \in \mathbb{R}^{n}$, it holds that $PA_{\mathcal{I}_r}^\top SS^\top(A_{\mathcal{I}_r}P\tilde{x}-\hat{b}_{\mathcal{I}_r}) \neq 0$ if and only if $S^\top(A_{\mathcal{I}_r}P\tilde{x}-\hat{b}_{\mathcal{I}_r}) \neq 0$. 
\end{lemma}
\begin{lemma}[\cite{Zeng2024adaptive}, Lemma 2.4]
	\label{xie-empty}
	Assume that Assumption \ref{assumption1} holds. Then for any $\tilde{x} \in \mathbb{R}^n$,
	$
	S^\top(A_{\mathcal{I}_r}P\tilde{x}-\hat{b}_{\mathcal{I}_r}) = 0 \text{ for all } S \in \Omega
	$ if and only if $
	A_{\mathcal{I}_r}P\tilde{x} = \hat{b}_{\mathcal{I}_r}.
	$
\end{lemma}

We summarize Lemma 2.3 in \cite{Lorenz2025Minimal} and Lemma 2.5 in \cite{Zeng2024adaptive} in the following lemma.
\begin{lemma}\label{lem-pd} 
	Suppose that Assumption~\ref{assumption1} holds and that $A_{\mathcal{I}_r} P \neq 0$.
	Then the matrices
	\begin{equation}
		\label{eq-H}
		\begin{aligned}
			\bar{H}:=\mathbb{E}\left[\frac{SS^\top}{\|S\|^2_2}\right] \ \ \text{and}\ \ H:=\mathbb{E}\left[\frac{SS^\top}{\|S^\top A_{\mathcal{I}_r}P\|^2_2}\right]
		\end{aligned}
	\end{equation}
	are well-defined and positive definite, here we define $\frac{0}{0}=0$.
\end{lemma}
 


\section{Randomized iterative methods with subspace-constrained preconditioning}
\label{sec-3}

In this section, we develop subspace-constrained randomized iterative methods (SCRIM) for solving the linear system \eqref{main-prob}. 
The SCRIM method can be viewed as a projected SGD method \eqref{iteration-pro-SGD} for solving the following constrained stochastic optimization problem reformulated from the preconditioned linear system \eqref{main-pre-prob}
\begin{equation}\label{main-sto-prob}
	\begin{aligned}
		\min_{x\in\mathbb{R}^n}&\quad f(x)=\mathbb{E}[f_{S}(x)]\\
		\text{subject to}&\quad x\in \mathcal{X}_p:=\{x\in\mathbb{R}^n \mid A_{\mathcal{I}_p}x=b_{\mathcal{I}_p}\},
	\end{aligned}
\end{equation}
where $f_{S}(x):=\frac{1}{2}\lVert S^\top(A_{\mathcal{I}_r}Px-\hat{b}_{\mathcal{I}_r})\rVert_2^2$ with $S$ being a random sketching matrix drawn from $(\Omega,\mathcal{F},\mathbf{P})$, and $\hat{b}_{\mathcal{I}_r}=b_{\mathcal{I}_r} - A_{\mathcal{I}_r} A_{\mathcal{I}_p}^\dagger b_{\mathcal{I}_p}$. The resulting projected SGD iteration scheme reads
\begin{equation}\label{eq-PSGD}
	\left\{
	\begin{array}{ll}
		y^{k+1}=x^k-\alpha_kPA_{\mathcal{I}_r}^\top S_{k}S_{k}^\top(A_{\mathcal{I}_r}Px^k-\hat{b}_{\mathcal{I}_r}),\\
		x^{k+1}=\Pi_{\mathcal{X}_p}(y^{k+1})=y^{k+1}-A_{\mathcal{I}_p}^\dagger \big(A_{\mathcal{I}_p}y^{k+1}-b_{\mathcal{I}_p}\big),
	\end{array}
	\right.
\end{equation}
where $\alpha_k$ denotes the stepsize and $\Pi_{\mathcal{X}_p}(\cdot)$ represents the orthogonal projection onto the feasible set $\mathcal{X}_p$.
 Substituting the intermediate update $y^{k+1}$ into the projection step and note that $A_{\mathcal{I}_p}P = 0$, yields the compact iteration
\begin{equation}\label{eq-bas-xk+1}
	\begin{aligned}
		x^{k+1}=x^k-A_{\mathcal{I}_p}^\dagger\big(A_{\mathcal{I}_p}x^k-b_{\mathcal{I}_p}\big)-\alpha_k(A_{\mathcal{I}_r}P)^\top S_{k}S_{k}^\top\big(A_{\mathcal{I}_r}Px^k-\hat{b}_{\mathcal{I}_r}\big).
	\end{aligned}
\end{equation}
If we choose the initial point $x^0 = A_{\mathcal{I}_p}^\dagger b_{\mathcal{I}_p} \in \mathcal{X}_p$, then all subsequent iterates satisfy $x^k \in \mathcal{X}_p$. Under this valid initialization, the update \eqref{eq-bas-xk+1} simplifies to
\begin{equation}\label{eq-bas-xk+1-2}
	x^{k+1}=x^k-\alpha_kPA_{\mathcal{I}_r}^\top S_{k}S_{k}^\top\left(A_{\mathcal{I}_r}Px^k-\hat{b}_{\mathcal{I}_r}\right),
\end{equation}
which can be viewed as the randomized iterative method \eqref{RIM-iteration} for solving $A_{\mathcal{I}_r}Px=\hat{b}_{\mathcal{I}_r}$. 

We further show that the auxiliary reformulation involving $\hat{b}_{\mathcal{I}_r}$ in \eqref{eq-bas-xk+1-2} is unnecessary.  For any iterate $x^k$, we have $Px^k=x^k-A_{\mathcal{I}_p}^\dagger A_{\mathcal{I}_p}x^k=x^k-A_{\mathcal{I}_p}^\dagger b_{\mathcal{I}_p}$, which yields  $A_{\mathcal{I}_r}Px^k-\hat{b}_{\mathcal{I}_r}=A_{\mathcal{I}_r}x^k-A_{\mathcal{I}_r}A_{\mathcal{I}_p}^\dagger b_{\mathcal{I}_p}-\hat{b}_{\mathcal{I}_r}=A_{\mathcal{I}_r}x^k-b_{\mathcal{I}_r}$. 
Based on this equality,   \eqref{eq-bas-xk+1-2} can be further simplified as
\begin{equation}\label{iter-sim}
	x^{k+1}=x^k-\alpha_kPA_{\mathcal{I}_r}^\top S_{k}S_{k}^\top(A_{\mathcal{I}_r}x^k-b_{\mathcal{I}_r}),\ k\geq0.
\end{equation}
This iteration scheme can be viewed as a preconditioned SGD \eqref{iteration-pre-SGD} (or preconditioned randomized iterative method) with the preconditioner matrix $P$ for the linear system $A_{\mathcal{I}_r}x=b_{\mathcal{I}_r}$. 
 
Note that when $S_{k}^\top(A_{\mathcal{I}_r}x^k - b_{\mathcal{I}_r}) = 0$, the iteration scheme \eqref{iter-sim} reduces to $x^{k+1}=x^k$, which corresponds to a \emph{null step}. To avoid such null steps, we construct the randomized sketching matrix $S_k$ to satisfy $S_{k}^\top(A_{\mathcal{I}_r}x^k - b_{\mathcal{I}_r}) \neq 0$. 
In addition, it follows from Lemma \ref{lemma-meanwhile} that $ PA_{\mathcal{I}_r}^\top S_{k}S_{k}^\top(A_{\mathcal{I}_r}Px^k-\hat{b}_{\mathcal{I}_r})\neq0$ holds if and only if $ S_{k}^\top(A_{\mathcal{I}_r}Px^k-\hat{b}_{\mathcal{I}_r})\neq0$.
Combining this result with the identity $A_{\mathcal{I}_r}Px^k-\hat{b}_{\mathcal{I}_r}=A_{\mathcal{I}_r}x^k-b_{\mathcal{I}_r}$, we conclude that  $ PA_{\mathcal{I}_r}^\top S_{k}S_{k}^\top(A_{\mathcal{I}_r}x^k-b_{\mathcal{I}_r})\neq0$ if and only if $ S_{k}^\top(A_{\mathcal{I}_r}x^k-b_{\mathcal{I}_r})\neq0$. 
Hence,  
\begin{equation}\label{eq-Ladap}
	\begin{aligned}
		L_{\text{adap}}^k=\frac{\lVert S_{k}^\top(A_{\mathcal{I}_r}x^k-b_{\mathcal{I}_r})\rVert_2^2}{\lVert PA_{\mathcal{I}_r}^\top S_{k}S_{k}^\top(A_{\mathcal{I}_r}x^k-b_{\mathcal{I}_r})\rVert_2^2}
	\end{aligned}
\end{equation}
is well-defined throughout the iteration process.
Now we are ready to state the  proposed SCRIM, which is formally described in Algorithm \ref{Algo-1}.

\begin{algorithm}[htpb]
	\caption{Subspace-constrained randomized iterative method (SCRIM)}
	\label{Algo-1}
	\begin{algorithmic}
		\Require $A\in\mathbb{R}^{m\times n},b\in\mathbb{R}^m$, fixed row indices $\mathcal{I}_p\subseteq[m],\mathcal{I}_r=[m]\backslash \mathcal{I}_p$, probability space ${(\Omega,\mathcal{F},\mathbf{P})},\zeta\in(0,2)$, and $k=0$. 
		\begin{enumerate}
			\item[1:] Compute $A_{\mathcal{I}_p}^\dagger$ and set $x^0=A_{\mathcal{I}_p}^\dagger b_{\mathcal{I}_p}$.
			\item [2:] Randomly select a sampling matrix $S_{k}\in\Omega$ until $ S_{k}^\top(A_{\mathcal{I}_r}x^k-b_{\mathcal{I}_r})\neq0$. 
			\item [3:] Compute
			\begin{equation}
				\left\{
				\begin{array}{l}
					g^k=-A_{\mathcal{I}_r}^\top S_{k}S_{k}^\top(A_{\mathcal{I}_r}x^{k}-b_{\mathcal{I}_r}),\\
					\bar{g}^k=A_{\mathcal{I}_p}g^k,\\
					d^k=g^k-A_{\mathcal{I}_p}^\dagger\bar{g}^k.
				\end{array}\right.
			\end{equation}
			\item [4:] Compute the parameter $L_{\text{adap}}^k $ in \eqref{eq-Ladap} and set $ \alpha_k=(2-\zeta)L_{\text{adap}}^k$.
			\item [5:] Update $x^{k+1}=x^k+\alpha_kd^k$.
			\item [6:] If the stopping rule is satisfied stop and go to output. Otherwise, set $k=k+1$ and return Step 2.
		\end{enumerate}
		\Ensure
		The approximate solution.
	\end{algorithmic}
\end{algorithm}


\begin{remark}\label{remark-SCRK1}
	When the sampling space satisfies $\Omega = \{e_i\}_{i=1}^{m_r}$, with $e_{i_k}$ being sampled with probability 
	\(
	\frac{\|Pa_i\|_2^2}{\|A_{\mathcal{I}_r} P\|_F^2},  i \in \mathcal{I}_r.
	\)
	In this case, Algorithm~\ref{Algo-1} with $\zeta = 1$ reduces to
	\[
	x^{k+1} = x^k - \frac{\langle a_{i_k}, x^k\rangle - b_{i_k}}{\|P a_{i_k} \|_2^2} P a_{i_k},
	\]
	which exactly coincides with the SCRK method proposed in \cite{lok2024subspace}. When $\mathcal{I}_p = \emptyset$, Algorithm~\ref{Algo-1} reduces to 
	\[
	x^{k+1} = x^k - \alpha_k A^\top S_k S_k^\top (A x^k - b),
	\]
	which exactly recovers the randomized iterative method in \eqref{RIM-iteration}.
\end{remark}

\subsection{Convergence analysis}\label{Section 3.1}

For convenience, we define
\begin{equation}\label{eq-rho}
	\rho:=1-
	\zeta(2-\zeta)\sigma^2_{\min}(H^{\frac{1}{2}}A_{\mathcal{I}_r}P),
\end{equation}
where $H$ is given by \eqref{eq-H}. We have the following convergence result for Algorithm \ref{Algo-1}.

%
\begin{theorem}\label{Thm-Convergence-SCRIM}
	Suppose that the linear system $Ax=b$ is consistent and that the probability space $(\Omega,\mathcal{F},\mathbf{P})$ satisfies Assumption {\rm\ref{assumption1}}. Let $\{x^k\}_{k\geq0}$ be the iteration sequence generated by Algorithm {\rm\ref{Algo-1}}. If $Ax^{k}=b$, then $x^{k}=A^\dagger b$. Otherwise, 
	\begin{equation}\label{Convergence-SCRIM}
		\begin{aligned}
			\mathbb{E}[\lVert x^{k+1}-A^\dagger b\rVert_2^2]
			&\leq\rho^{k+1}\lVert x^0-A^\dagger b\rVert_2^2,
		\end{aligned}
	\end{equation}
	where $\rho$ is given by \eqref{eq-rho}.
\end{theorem}

\begin{remark}\label{remark-SCRK}
Suppose that we use the same probability space as in Remark~\ref{remark-SCRK1}. Then  we have $H = \frac{1}{\|A_{\mathcal{I}_r} P\|_F^2} I_{m_r}$ and  Theorem~\ref{Thm-Convergence-SCRIM} implies
\[
\mathbb{E}\bigl[\|x^{k} - A^\dagger b\|_2^2\bigr] \leq \left( 1 - \frac{\sigma^2_{\min}(A_{\mathcal{I}_r} P)}{\|A_{\mathcal{I}_r} P\|^2_F}\right)^{k} \|x^0 - A^\dagger b\|_2^2.
\]
This result recovers the convergence rate derived in \cite[Theorem 1.1]{lok2024subspace} for the SCRK method. We note that the analysis in \cite[Theorem 1.1]{lok2024subspace} requires the matrix $A$ to have full column rank. However, our convergence analysis generalizes this result to arbitrary matrices, including rank-deficient cases, and guarantees convergence to the least-norm solution $A^\dagger b$. This generalization is enabled by the proposed QR-like factorization for constructing the preconditioned problem \eqref{main-pre-prob}.
\end{remark}
 
Next, we show that the convergence factor in Theorem~\ref{Thm-Convergence-SCRIM} is at least as tight as that of the randomized iterative method  in \eqref{RIM-iteration} (i.e., $\mathcal{I}_p = \emptyset$). To facilitate a fair comparison, we impose a unified formulation for the randomized sketching matrix. Specifically, let $\tilde{S} \in \mathbb{R}^{m \times q}$ be drawn from the probability space $(\tilde{\Omega}, \tilde{\mathcal{F}}, \tilde{\mathbf{P}})$. For convenience, we assume $\mathcal{I}_p = \{1, \ldots, m_p\}$, such that $\tilde{S}$ admits the following block row partition
\begin{equation}\label{eq-S}
	\tilde{S} = \begin{pmatrix} \tilde{S}_{\mathcal{I}_p} \\ S \end{pmatrix}, \quad \tilde{S}_{\mathcal{I}_p} \in \mathbb{R}^{m_p \times q},\quad  S \in \mathbb{R}^{m_r \times q}, \quad m = m_p + m_r.
\end{equation}
In this setting, $S$ corresponds to the randomized sketching matrix employed in Algorithm~\ref{Algo-1}.
 Now the randomized iterative method in \eqref{RIM-iteration} exhibits the following convergence property:
 \[
 \mathbb{E}\left[\lVert x^{k} - A^\dagger b \rVert_2^2\right] \leq \tilde{\rho}^{k} \lVert x^0 - A^\dagger b \rVert_2^2,
 \]
 where
 \begin{equation}\label{eq-tH-trho}
 	\tilde{H} := \mathbb{E} \left[ \frac{\tilde{S} \tilde{S}^\top}{\lVert \tilde{S}^\top A \rVert_2^2} \right] \quad \text{and} \quad
 	\tilde{\rho} = 1 - \zeta(2-\zeta) \sigma_{\min}^2(\tilde{H}^{1/2} A).
 \end{equation}
  The following result verifies that Algorithm \ref{Algo-1} achieves a tighter convergence factor than the randomized iterative method.  
  \begin{prop}\label{thm-con-rate-factor}
 	Let \(\tilde S\) and \(S\) denote the randomized sketching matrices defined in \eqref{eq-S}. Suppose that
 	\((\tilde{\Omega},\tilde{\mathcal F},\tilde{\mathbf P})\) satisfies
 	Assumption~\ref{assumption1}. 
 	Define
 	\(T:\tilde\Omega\to\Omega\) by \(T(\tilde S)=S\), and assume that 
 	\(
 	\tilde{\mathbf P}(T^{-1}(B))=\mathbf P(B)\) for all \(B\in\mathcal F\), 
 	where \(\mathbf P\) is the probability measure of \(S\) used in Algorithm~\ref{Algo-1}.
 	Then
 	\[
 	\rho\leq\tilde\rho,
 	\]
 	where $\rho$ and $\tilde{\rho}$ refer to the convergence factors defined in \eqref{eq-rho} and \eqref{eq-tH-trho}, respectively.
 \end{prop}
  

 \subsection{Selection strategies for \(\mathcal I_p\)}\label{sec-3.2}
Ideally, the index set $\mathcal{I}_p$ is chosen to minimize the convergence factor $\rho$, which is equivalent to maximizing $\sigma_{\min}(H^{\frac{1}{2}}A_{\mathcal{I}_r}P)$. However, since both $P$ and $H$ depend on $\mathcal{I}_p$, maximizing $\sigma_{\min}(H^{\frac{1}{2}}A_{\mathcal{I}_r}P)$ with respect to $\mathcal{I}_p$ admits no closed-form solution and remains computationally intractable in practice. The following result establishes a tractable lower bound for $\sigma_{\min}(H^{\frac{1}{2}}A_{\mathcal{I}_r}P)$, which provides a feasible surrogate criterion for selecting the index set $\mathcal{I}_p$.

\begin{lemma}\label{col-surrogate}
Under Assumption \ref{assumption1} and the condition $A_{\mathcal{I}_r}P\neq0$, the inequality
\begin{equation}\label{ineq-corollary}
		\sigma_{\min}(H^{\frac{1}{2}}A_{\mathcal{I}_r}P)\geq\lambda_{\min}(\bar{H})\frac{\sigma_{\min}(A_{\mathcal{I}_r}P)}{\lVert A_{\mathcal{I}_r}P\rVert_F}
\end{equation}
holds, where $H$ and $\bar{H}$ denote symmetric positive definite matrices defined in \eqref{eq-H}.
\end{lemma}

For any matrix $B$,	we define $\kappa(B):=\frac{\lVert B\rVert_F}{\sigma_{\min}(B)}$ as the scaled condition number \cite{demmel1988probability,ke2026robust} of $B$. 
	Since $\bar{H}$ is independent of $\mathcal{I}_p$, the only term in the lower bound \eqref{ineq-corollary} that depends on $\mathcal{I}_p$ is $\kappa(A_{\mathcal{I}_r}P)$. 
	A natural strategy for selecting the index set $\mathcal{I}_p$ is to minimize the scaled condition number $\kappa(A_{\mathcal{I}_r}P)$, yet this optimization task remains computationally challenging \cite{ipsen2025many}.
We therefore adopt the minimization of $\|A_{\mathcal{I}_r}P\|_F$ as a computationally feasible surrogate criterion for choosing $\mathcal{I}_p$. Recalling that $A_{\mathcal{I}_p}P=0$, we obtain
\(
AP=
\begin{pmatrix}
	0\\
	A_{\mathcal{I}_r}P
\end{pmatrix},
\)
which implies $\lVert A_{\mathcal{I}_r}P\rVert_F^2=\lVert AP\rVert_F^2=
 \| A - A A_{\mathcal{I}_p}^\dagger A_{\mathcal{I}_p}  \|^2_F$. 
For any given $m_p$, we thus consider the problem
\[
\min_{\substack{\mathcal{I}_p \subset [m],|\mathcal{I}_p| = m_p}}  
\left\| A - A A_{\mathcal{I}_p}^\dagger A_{\mathcal{I}_p} \right\|^2_F,
\]
which exactly corresponds to the row interpolative decomposition (ID) problem \cite{dong2023simpler,dong2025robust,pearce2025adaptive,cortinovis2026adaptive}. In the following, we introduce several standard ID-based row selection strategies for constructing the index set $\mathcal{I}_p$.

\textbf{Column-pivoted QR:}
Column-pivoted QR (CPQR) \cite[Sec. 5.4.2]{golub2013matrix} is a deterministic greedy method for skeleton selection. 
Let \(X^{(0)}=A\), and let \(X_i^{(t)}\) denote the \(i\)-th row of \(X^{(t)}\).
At the \((t+1)\)-th step, CPQR performs 
\begin{equation}\label{eq-CPQR}
	s_{t+1}
	\in
	\arg\max_{i\in[m]}
	\left\|X^{(t)}_i\right\|_2^2,\quad
	X^{(t+1)}
	=
	X^{(t)}
	-
	X^{(t)}
	\frac{
		\left(X^{(t)}_{s_{t+1}}\right)^\top X^{(t)}_{s_{t+1}}
	}{
		\left\|X^{(t)}_{s_{t+1}}\right\|_2^2
	}.
\end{equation}
After \(m_p\) steps, the selected pivots are collected as
\(
\mathcal I_p=\{s_1,\ldots,s_{m_p}\}.
\)

\textbf{SVD-based selection:} The SVD-based selection \cite[Section 5.5.7]{golub2013matrix} is motivated by the ideal case in which the selected rows span the
leading right singular subspace. Let \(
A=U\Sigma V^\top
\)
be the singular value decomposition of \(A\), and let \(V_{m_p}\) denote the leading \(m_p\) right singular
vectors of \(A\). If
\(
\operatorname{Range}(A_{\mathcal{I}_p}^{\top})
=
\operatorname{Range}(V_{m_p}),
\)
then \(A_{\mathcal{I}_p}^{\dagger}A_{\mathcal{I}_p}=V_{m_p}V_{m_p}^{\top}\), and hence
\[
\left\|A-AA_{\mathcal{I}_p}^{\dagger}A_{\mathcal{I}_p}\right\|_F^2
=
\left\|A-AV_{m_p}V_{m_p}^{\top}\right\|_F^2
=
\sum_{i=m_p+1}^{\min(m,n)}\sigma_i^2(A).
\]
In this case, the ID approximation error
\(
\left\|A-AA_{\mathcal{I}_p}^{\dagger}A_{\mathcal{I}_p}\right\|_F
\)
coincides with the error of the best rank-\(m_p\) approximation of \(A\) in the Frobenius norm \cite{eckart1936approximation}. 
To approximate such an index set, we form \(AV_{m_p}\), whose rows contain the coordinates
of the rows of \(A\) along the leading right singular directions. We then apply the
row-selection CPQR procedure to \(AV_{m_p}\) and take the first \(m_p\) row pivots as
\(\mathcal I_p\). 
This SVD-based selection procedure serves as an idealized benchmark, but is computationally expensive because it requires computing the leading right singular vectors.

\textbf{Squared-norm sampling:}
Squared-norm sampling (SqNorm) \cite{frieze2004fast,deshpande2006matrix} selects rows randomly according to their relative importance,
measured here by the squared row norms
\(
 \frac{\|a_i\|_2^2}{\|A\|_F^2}\), for 
\(i\in[m]\). 
We sample \(m_p\) rows without replacement from this distribution. This method is simple
to implement and has low row-selection cost.

\textbf{Sketched CPQR:}
Sketched CPQR (SkCPQR) \cite{dong2023simpler,duersch2020randomized} is a randomized variant of CPQR that performs pivot selection on a
compressed version of the matrix. Specifically, we draw a Gaussian matrix
\(G\in\mathbb R^{n\times s}\) with \(s< n\), and form
\(
Y=AG\in\mathbb R^{m\times s}
\). 
The rows of \(Y\) provide a lower-dimensional representation of the rows of \(A\). We then
apply the row-selection CPQR procedure to \(Y\), and use the resulting row pivots as
\(\mathcal I_p\). This strategy performs pivot selection in a lower-dimensional sketched
space, thereby reducing the cost of CPQR. Related randomized sketching methods for ID can
be found in
\cite{pearce2025randomized,epperly2025adaptive,cortinovis2026adaptive,pearce2025adaptive}.

\textbf{Robust blockwise random pivoting:} Random pivoting \cite{duersch2017randomized,epperly2025adaptive,chen2025randomly} is an adaptive alternative to deterministic greedy pivoting. Instead of
selecting the largest residual-norm pivot at each step in \eqref{eq-CPQR}, robust blockwise random pivoting (RBRP) \cite{dong2025robust} samples multiple pivots according to the
squared row norms of the current residual matrix. 
This preserves the adaptive nature of
greedy pivoting while introducing randomness into the pivot selection.

\section{Subspace-constrained iterative-sketching-based Krylov subspace methods}\label{sec-4}

In this section, we replace the stochastic gradient direction used in Algorithm~\ref{Algo-1} with search directions obtained by orthogonalizing the negative stochastic gradients via a truncated Gram-Schmidt process, motivated by the success of the orthogonal direction method \cite[Section 7]{shewchuk1994introduction}. This leads to the development of subspace-constrained iterative-sketching-based Krylov subspace  (SC-IS-Krylov) method.

 We first revisit the core framework of the orthogonal direction method~\cite{shewchuk1994introduction}. Assume that we have a set of orthogonal search directions $p^0, p^1, \ldots, p^k$.
The method updates the iterate at the $k$-th iteration via
\begin{equation}\label{iter-xk+1-cg}
x^{k+1} = \argmin{x \in x^k + \operatorname{span}\{p^k\}} \frac{1}{2}\|x - A^\dagger b\|_2^2.
\end{equation}
 When $p^k \neq 0$, this leads to the update $x^{k+1} = x^k + \delta_k p^k$ with
 \[
\delta_k=\argmin{\delta\in\mathbb{R}}\frac{1}{2}\lVert x^k+\delta p^k-A^\dagger b\rVert_2^2=-\frac{\langle p^k, x^k-A^\dagger b\rangle}{\lVert p^k \rVert_2^2}.
 \]
 However, computing $\delta_k$ requires knowledge of the unknown solution $A^\dagger b$; therefore, the method is not implementable in general.  To eliminate the dependence on $A^\dagger b$, the author  \cite{shewchuk1994introduction} required that $p^k$ be $A^\top A$-conjugate to $x^{k+1} - A^\dagger b$.
 
Nevertheless, we address this issue by constructing feasible search directions $p^k$ based on the negative stochastic gradient $d^{k}:=-(A_{\mathcal{I}_r}P)^\top S_{k}S_{k}^\top(A_{\mathcal{I}_r}Px^k-\hat{b}_{\mathcal{I}_r})$ to obtain the  search directions $p^k$.
In particular, we obtain $p^k$ via the following Gram-Schmidt orthogonalization process
\begin{equation}\label{iter-pk}
	\begin{aligned}
		p^{k}=d^k-\sum_{i=j_{k,\ell}}^{k-1}\frac{\langle d^k,p^i\rangle}{\lVert p^i\rVert_2^2}p^i,
	\end{aligned}
\end{equation}
where $j_{k,\ell} := \max\{k - \ell + 1, 0\}$ tracks the most recent $\ell$ iterates and $\ell$ is a positive integer. We refer to $\ell$ as the truncation parameter.  For instance, $\ell = +\infty$ corresponds to using all previous vectors $p^i$.  
The following lemma shows that $\langle p^k,x^k-A^\dagger b\rangle=-\lVert S_k^\top(A_{\mathcal{I}_r}x^k-b_{\mathcal{I}_r})\rVert_2^2$ under a  mild condition.

\begin{lemma}\label{lem-pkneq0}
Let the initial point be $x^0=A_{\mathcal I_p}^\dagger b_{\mathcal I_p}$.
Set $p^0=d^0=-PA_{\mathcal I_r}^\top S_0S_0^\top
(A_{\mathcal I_r}x^0-b_{\mathcal I_r})$. 
Suppose that the sequences $\{x^k\}_{k\geq 0}$ and $\{p^k\}_{k\geq 1}$  are generated by \eqref{iter-xk+1-cg} and \eqref{iter-pk}, respectively. For any \(k\geq0\), we have $x^k\in\mathcal{X}_p$. 
If $S_i^\top(A_{\mathcal{I}_r}x^i-b_{\mathcal{I}_r})\neq0$ for $i=0,\ldots,k$, 
then $p^k\neq 0$ and
\(
\langle p^k,x^k-A^\dagger b\rangle=-\lVert S_k^\top(A_{\mathcal{I}_r}x^k-b_{\mathcal{I}_r})\rVert_2^2.
\) 
Moreover,   if \(\ell\ge2\) and $k\geq1$, then
\(
\langle p^\mu,p^\nu\rangle=0\) for all \(j_{k,\ell}\le \mu<\nu\le k .
\)
\end{lemma}

Hence, if $S_i^\top(A_{\mathcal{I}_r}x^i-b_{\mathcal{I}_r})\neq0,i=0,\ldots,k$, we have
\begin{equation}\label{def-delta} 
		\delta_k
		=\frac{\lVert S_{k}^\top(A_{\mathcal{I}_r}x^k-b_{\mathcal{I}_r})\rVert_2^2}{\lVert p^k\rVert_2^2},
\end{equation}
where the last equation follows from Lemma \ref{lem-pkneq0}. The proposed method is summarized in Algorithm~\ref{Algo-2}. The justification for the name SC-IS-Krylov will be provided in Section~\ref{section-4-2}.
 
\begin{algorithm}[htpb]
	\caption{The SC-IS-Krylov method}
	\label{Algo-2}
	\begin{algorithmic}
		\Require $A\in\mathbb{R}^{m\times n},b\in\mathbb{R}^m$, fixed row indices $\mathcal{I}_p\subseteq[m],\mathcal{I}_r=[m]\backslash \mathcal{I}_p$, probability space ${(\Omega,\mathcal{F},\mathbf{P})}$, positive integer $\ell$, and $k=0$.
		\begin{enumerate}
            \item[1:] Compute $A_{\mathcal{I}_p}^\dagger$ and set $x^0=A_{\mathcal{I}_p}^\dagger b_{\mathcal{I}_p}$.
			\item [2:] Randomly select a matrix $S_0 \in \Omega$ until $S_{0}^\top(A_{\mathcal{I}_r}x^{0}-b_{\mathcal{I}_r})\neq 0$. 
			\item [3:] Compute
			\begin{equation*}
				\left\{
				\begin{array}{l}
					g^0=-A_{\mathcal{I}_r}^\top S_{0}S_{0}^\top(A_{\mathcal{I}_r}x^{0}-b_{\mathcal{I}_r}),\\
					\bar{g}^0=A_{\mathcal{I}_p}g^0,\\
					p^0=g^0-A_{\mathcal{I}_p}^\dagger\bar{g}^0.
				\end{array}\right.
			\end{equation*}
			\item[4:] Set $\delta_k =\lVert S_{k}^\top(A_{\mathcal{I}_r}x^k-b_{\mathcal{I}_r})\rVert_2^2 / \|p^k\|_2^2$.
			\item[5:] Update $x^{k+1}=x^k+\delta_k p^k$.
			\item[6:] Randomly select a matrix $S_{k+1} \in \Omega$ until $ S_{k+1}^\top(A_{\mathcal{I}_r}x^{k+1}-b_{\mathcal{I}_r})\neq 0$.
			\item [7:] Update  $j_{k+1,\ell}=\max\{k-\ell+2,0\}$ and compute
			\begin{equation}\label{iter-krylov}
				\left\{
				\begin{array}{l}
					g^{k+1}=-A_{\mathcal{I}_r}^\top S_{k+1}S_{k+1}^\top(A_{\mathcal{I}_r}x^{k+1}-b_{\mathcal{I}_r}),\\
					\bar{g}^{k+1}=A_{\mathcal{I}_p}g^{k+1},\\
					d^{k+1}=g^{k+1}-A_{\mathcal{I}_p}^\dagger \bar{g}^{k+1},\\
					\eta_{k+1}^i=\langle d^{k+1},p^i\rangle/\lVert p^i\rVert_2^2, \ i=j_{k+1,\ell},\ldots,k,\\
					p^{k+1}=d^{k+1}-\sum_{i=j_{k+1,\ell}}^{k}\eta_{k+1}^ip^i.
				\end{array}\right.
			\end{equation}
			\item [8:] If the stopping rule is satisfied, stop and go to output. Otherwise, set $k=k+1$ and return to Step 4.
		\end{enumerate}
		\Ensure
		The approximate solution $x^k$.
	\end{algorithmic}
\end{algorithm}

\subsection{Convergence analysis}
We begin by introducing some auxiliary variables. 
Define $\tilde{P}_0$ as the empty matrix, and for $k \geq 1$, define
\begin{equation}\label{eq-tilde-Pk}
	\tilde{P}_k := \left( p^{j_{k,\ell}},\, p^{j_{k,\ell}+1},\, \ldots,\, p^{k-1} \right) \in \mathbb{R}^{n \times (k - j_{k,\ell})},
\end{equation}
where the vectors $p^i$ for $i = j_{k,\ell}, \dots, k-1$ are defined in Algorithm~\ref{Algo-2}. 
Since the columns of \(\tilde{P}_k\) are mutually orthogonal and nonzero, it follows from Lemma \ref{lem-pkneq0} that
\[
\tilde{P}_k^\dagger = \begin{pmatrix}
	\frac{p^{j_{k,\ell}}}{\lVert p^{j_{k,\ell}} \rVert_2^2}, \frac{p^{j_{k,\ell}+1}}{\lVert p^{j_{k,\ell}+1} \rVert_2^2}, \ldots, \frac{p^{k-1}}{\lVert p^{k-1} \rVert_2^2}
\end{pmatrix}^\top.
\]
Then, for $k\geq1$, we can rewrite $p^k$ in \eqref{iter-pk} as
\begin{equation}\label{iter-pk-2}
	p^k=d^k-\tilde{P}_k\tilde{P}_k^\dagger d^k=(I-\tilde{P}_k\tilde{P}_k^\dagger) d^k.
\end{equation}
In addition, we define 
\begin{equation}\label{eq-Q_k}    
	\mathcal{Q}_k := \left\{ S \in \Omega \mid S^\top (A_{\mathcal{I}_r} x^k - b_{\mathcal{I}_r}) \neq 0 \right\}.
\end{equation}
and 
\begin{equation} \label{eq-q_k}
	q_k := \inf_{S_k \in \mathcal{Q}_k} \left\{ \left(1 - \frac{\lVert \tilde{P}_k \tilde{P}_k^\dagger d^k \rVert_2^2}{\lVert d^k \rVert_2^2} \right)^{-1} \right\}.
\end{equation}

At the \(k\)-th iteration, we consider the product probability space
\(
\left( \prod_{i=0}^k \Omega, \otimes_{i=0}^k \mathcal{F}, \tilde{\mathbf{P}} \right),
\)
where \(\otimes\) denotes the product of \(\sigma\)-algebras and \(\tilde{\mathbf{P}}\) is the corresponding product measure, as defined in \cite[Section 5]{athreya2006measure}.
Let \(\mathcal{B}_k:= \sigma(S_0, S_1, \ldots, S_{k-1})\) be the $\sigma$-algebra generated by the random variables $ (S_0, S_1, \ldots, S_{k-1})$, 
where $\mathcal{B}_0$ is the trivial $\sigma$-algebra (i.e., $\mathcal{B}_0 = \{\emptyset, \Omega\}$). 
The conditional expectation given \(\mathcal{B}_k\) is denoted by
\(
\mathbb{E}_k[\cdot] := \mathbb{E}[\cdot \mid \mathcal{B}_k].
\)
The convergence result for Algorithm \ref{Algo-2} is as follows.

\begin{theorem}\label{Thm-Convergence-IS-Krylov-SC}
	Suppose that the linear system $Ax=b$ is consistent and that the probability space $(\Omega,\mathcal{F},\mathbf{P})$ satisfies Assumption \ref{assumption1}. Let   $\{x^k\}_{k\geq0}$ be the iteration sequence generated by Algorithm {\rm\ref{Algo-2}}. If $Ax^{k+1}=b$, then $x^{k+1}=A^\dagger b$. Otherwise, 
	\[
			\mathbb{E}_k\left[\lVert x^{k+1}-A^\dagger b\rVert_2^2 \right]\leq\left(1-q_k \sigma^2_{\min}(H^{\frac{1}{2}}A_{\mathcal{I}_r}P) \right)\lVert x^k-A^\dagger b\rVert_2^2,
	\]
	where $q_k\geq1$ and $H$ are defined in \eqref{eq-q_k} and \eqref{eq-H}, respectively.
\end{theorem}

Upon comparing Theorems~\ref{Thm-Convergence-SCRIM} and~\ref{Thm-Convergence-IS-Krylov-SC}, we observe that the SC-IS-Krylov method exhibits a convergence bound that is at least as tight as that of the SCRIM method. In particular, for certain probability spaces, the parameter $q_k$ in Theorem~\ref{Thm-Convergence-IS-Krylov-SC} can be strictly greater than $1$.
Indeed, consider the case where the sample space is $\Omega = \{S\}$ with $S$ being a fixed matrix such that $S^\top S$ is positive definite. We next show that when $A_{\mathcal{I}_r} x^k - b_{\mathcal{I}_r} \neq 0$, it holds that $q_k > 1$.
From the definition of $q_k$ in \eqref{eq-q_k} and noting that $\tilde{P}_k^\top = \tilde{P}_k^\top \tilde{P}_k \tilde{P}_k^\dagger$, to ensure $\tilde{P}_k \tilde{P}_k^\dagger d^k \neq 0$, it suffices to prove that $\tilde{P}_k^\top d^k \neq 0$. In particular, we show that the last component of $\tilde{P}_k^\top d^k$ is nonzero, i.e.,
\[
\langle p^{k-1}, d^k \rangle = \frac{\langle x^k - x^{k-1}, d^k \rangle}{\delta_{k-1}} = \frac{\langle x^k - A^\dagger b, d^k \rangle}{\delta_{k-1}} - \frac{\langle x^{k-1} - A^\dagger b, d^k \rangle}{\delta_{k-1}} \neq 0.
\]
	To show this, we first evaluate the term $\langle x^{k-1} - A^\dagger b,d^k\rangle$. 
	Since $d^k=-PA_{\mathcal{I}_r}^\top SS^\top(A_{\mathcal{I}_r}x^k-b_{\mathcal{I}_r})$ and $A_{\mathcal{I}_r}x^k-b_{\mathcal{I}_r}=A_{\mathcal{I}_r}(x^k-x^{k-1})+A_{\mathcal{I}_r}x^{k-1}-b_{\mathcal{I}_r}$, it follows that 
	\[
		\begin{aligned}
			&\langle x^{k-1} - A^\dagger b, d^k \rangle\\
			&= -\langle x^{k-1} - A^\dagger b, PA_{\mathcal{I}_r}^\top SS^\top A_{\mathcal{I}_r} (x^k - x^{k-1}) \rangle-\langle x^{k-1} - A^\dagger b, PA_{\mathcal{I}_r}^\top SS^\top (A_{\mathcal{I}_r} x^{k-1} - b_{\mathcal{I}_r}) \rangle\\
			&=-\langle PA_{\mathcal{I}_r}^\top SS^\top A_{\mathcal{I}_r} (x^{k-1} - A^\dagger b), x^k - x^{k-1} \rangle-\| S^\top (A_{\mathcal{I}_r}  x^{k-1} - b_{\mathcal{I}_r}) \|_2^2\\
			&= \langle d_{k-1}, \delta_{k-1} p^{k-1} \rangle-\| S^\top (A_{\mathcal{I}_r}  x^{k-1} - b_{\mathcal{I}_r}) \|_2^2.
		\end{aligned}
	\]
	From the definition of $p^{k-1}$ in \eqref{iter-pk}, we have 
	$$\langle d_{k-1}, \delta_{k-1} p^{k-1} \rangle=\delta_{k-1} \langle p^{k-1} + \sum_{i = j_{k-1}}^{k-2} \eta_{k-1}^i p^i, p^{k-1} \rangle= \delta_{k-1} \| p^{k-1} \|_2^2=\| S^\top (A_{\mathcal{I}_r}  x^{k-1} - b_{\mathcal{I}_r}) \|_2^2,$$
	where the last equality follows from the definition of \( \delta_{k-1} \) in \eqref{def-delta}. Combining these results, we obtain \( \langle x^{k-1} - A^\dagger b, d^k \rangle = 0 \).
	Thus, we have
	\[
		\langle p^{k-1},d^k\rangle=\frac{\langle x^k-x^{k-1},d^k\rangle}{\delta_{k-1}}=\frac{\langle x^k - A^\dagger b,d^k\rangle}{\delta_{k-1}}=\frac{\| S^\top (A_{\mathcal{I}_r} x^k - b_{\mathcal{I}_r}) \|_2^2}{\delta_{k-1}}\neq0.
	\]
	This implies that \( \tilde{P}_k^\top d^k \neq 0 \) as long as \( A_{\mathcal{I}_r}  x^k \neq b_{\mathcal{I}_r} \).  
	Consequently, we have $q_k>1$. 

\subsection{Connection to the Krylov subspace method}
\label{section-4-2}

We first introduce an equivalent form of the iterate $x^{k+1}$ obtained by \eqref{iter-xk+1-cg} and \eqref{iter-pk}.
\begin{lemma}\label{lemma-0423}
Let $\{x^k\}_{k\geq0}$ and $\{p^k\}_{k\geq0}$ be the sequences generated by Algorithm~\ref{Algo-2}. Then, for all $k \geq 0$,  the iterate $x^{k+1}$ is the unique minimizer of
\begin{equation}\label{iter-xk+1-cg2}
	\min_{x \in x^{j_{k,\ell}} + \operatorname{span}\{p^{j_{k,\ell}}, \ldots, p^k\}} \|x - A^\dagger b\|_2^2.
\end{equation}
\end{lemma}

Recall that, given a matrix $B \in \mathbb{R}^{n \times n}$ and a vector $r \in \mathbb{R}^n$, the Krylov subspace of order $k$ is defined as~\cite[Section 10.1.1]{golub2013matrix}
\begin{equation}
	\label{def-krylov}
	\mathcal{K}_k(B, r) := \operatorname{span}\{r, Br, B^2 r, \ldots, B^{k-1} r\}.
\end{equation}
The following theorem demonstrates that $x^0 + \mathcal{K}_k\bigl(\hat{A}^\top \hat{A},\; \hat{A}^\top (\hat{A} x^0 - \hat{b})\bigr)$ is a specific instance of the subspace in Lemma~\ref{lemma-0423}.
\begin{theorem}
	\label{xie-equ-krylov}
Suppose that $\{x^k\}_{k\geq0}$ is the iteration sequence generated by Algorithm~\ref{Algo-2} with $\ell = \infty$ and $\Omega = \{I\}$. Let $\hat{r}^0 = \hat{A} x^0 - \hat{b}$ and $\hat{r}_{\mathcal{I}_r}^0 = A_{\mathcal{I}_r} P x^0 - \hat{b}_{\mathcal{I}_r}$.
	Then the affine subspaces  $x^0+\operatorname{span}\{p^0,p^1,\ldots,p^k\}$, $x^0+\mathcal{K}_{k+1}\left(PA_{\mathcal{I}_r}^\top A_{\mathcal{I}_r}P,\;PA_{\mathcal{I}_r}^\top \hat{r}_{\mathcal{I}_r}^0\right)$, and $x^0+\mathcal{K}_{k+1}\left(\hat A^\top\hat A,\;\hat A^\top\hat{r}^0\right)$ are identical.
\end{theorem}

By Lemma~\ref{lemma-0423} and Theorem~\ref{xie-equ-krylov}, when $\ell = \infty$ and $\Omega = \{I\}$, the proposed orthogonalized search direction scheme \eqref{iter-xk+1-cg}--\eqref{iter-pk} reduces to a standard Krylov subspace method.
Specifically, the update rule \eqref{iter-xk+1-cg} takes the form
\[
\begin{aligned}
	x^{k+1} = \mathop{\arg\min}_{x \in \mathbb{R}^n} &\; \lVert x - A^\dagger b \rVert_2^2 \\
	\text{subject to} \quad &x \in x^0 + \mathcal{K}_{k+1}\big(\hat A^\top \hat A, \hat A^\top \hat{r}^0\big).
\end{aligned}
\]
This indicates that for general truncation parameters $\ell$ and probability spaces $(\Omega,\mathcal{F},\mathbf{P})$, our orthogonalized search direction framework \eqref{iter-xk+1-cg}--\eqref{iter-pk} inspires the construction of the novel subspace-constrained iterative-sketching-based Krylov (SC-IS-Krylov) subspace method.

\begin{remark}
	When $\ell\geq \tau:=\operatorname{rank}(A_{\mathcal{I}_r}P)$, Algorithm \ref{Algo-2} can exhibit finite-time convergence. 
	From the proof of Lemma \ref{lem-pkneq0}, we have $p^k\in\operatorname{Range}(PA_{\mathcal{I}_r}^\top)$ for all $k\geq0$. 
	Hence,  $\operatorname{span}\{p^{0},\ldots,p^{\tau-1}\}\subseteq\operatorname{Range}(PA_{\mathcal{I}_r}^\top)$. 
	Since the vectors $p^0, \ldots, p^{\tau-1}$ are mutually orthogonal and nonzero, the subspace $\operatorname{span}\{p^{0},\ldots,p^{\tau-1}\}$ has dimension $\tau$. 
	Therefore,
	$$\operatorname{span}\{p^{0},\ldots,p^{\tau-1}\}=\operatorname{Range}(PA_{\mathcal{I}_r}^\top).$$
	Note that $(A_{\mathcal{I}_r}P)^\dagger \hat{b}_{\mathcal{I}_r} \in \operatorname{Range}(PA_{\mathcal{I}_r}^\top)$, we have
	$$A^\dagger b=A_{\mathcal{I}_p}^\dagger b_{\mathcal{I}_p}+(A_{\mathcal{I}_r}P)^\dagger \hat{b}_{\mathcal{I}_r}\in A_{\mathcal{I}_p}^\dagger b_{\mathcal{I}_p}+\operatorname{Range}(PA_{\mathcal{I}_r}^\top)=x^0+\operatorname{span}\{p^{0},\ldots,p^{\tau-1}\}.$$
	By \eqref{iter-xk+1-cg2}, we can get $x^\tau = A^\dagger b$. 
	Thus, Algorithm \ref{Algo-2} terminates after exactly $\tau$ iterations in exact arithmetic. 
\end{remark}

\begin{remark}
	When $\mathcal{I}_p=\emptyset$, Algorithm \ref{Algo-2} simplifies to 
	\begin{equation*}
		\left\{
		\begin{array}{l}
			\delta_k=\lVert S_{k}^\top (Ax^{k}-b)\rVert_2^2/\lVert p^k\rVert_2^2,\\
			x^{k+1}=x^k+\delta_kp^k,\\
			d^{k+1}=-A^\top S_{k+1}S_{k+1}^\top (Ax^{k+1}-b),\\
			\eta_{k+1}^i=\langle d^{k+1},p^i\rangle/\lVert p^i\rVert_2^2,i=j_{k+1,\ell},\ldots,k,\\
			p^{k+1}=d^{k+1}-\sum_{i=j_{k+1,\ell}}^k\eta_{k+1}^ip^i.
		\end{array}\right.
	\end{equation*}
	This iteration scheme aligns with the iterative-sketching-based Krylov subspace (IS-Krylov) method proposed in \cite{sun2025connecting}.
\end{remark}

\section{Numerical experiments}\label{sec-5} 

In this section, we implement SC-IS-Krylov (Algorithm~\ref{Algo-2}). 
The case \(\ell=1\) corresponds to SCRIM (Algorithm~1) with \(\zeta=1\).  
We examine the five strategies discussed in Section~\ref{sec-3.2} for selecting the constraint set \(\mathcal{I}_p\). 
We also compare the SC-IS-Krylov method with the IS-Krylov method proposed in \cite{sun2025connecting}. 
All experiments are implemented in MATLAB R2024a under Windows 11 on a laptop computer equipped with an Intel Core Ultra 7 155H CPU and 32 GB memory. The code to reproduce our results can be found at \href{https://github.com/xiejx-math/SCRIM}{https://github.com/xiejx-math/SCRIM}.
\subsection{Numerical setup}

We consider two types of coefficient matrices in our experiments. 
The first type consists of synthetic Gaussian matrices with outlying singular values. 
Given parameters \(m,n\) and the target rank \(r\), we construct
\(
A=UDV^\top,
\)
where \(U\in\mathbb{R}^{m\times r}\) and \(V\in\mathbb{R}^{n\times r}\) have orthonormal columns, and \(D=\operatorname{diag}(\sigma_1,\ldots,\sigma_r)\) contains the prescribed nonzero singular values.
In MATLAB notation, the matrices \(U\) and \(V\) are generated by
\(
\texttt{[U,}\sim\texttt{]=qr(randn(m,r),0)}\), and \(
\texttt{[V,}\sim\texttt{]=qr(randn(n,r),0)}.
\)
The singular values are generated from three separated clusters. 
Let \(n_L\) and \(n_S\) denote the numbers of large and small outlying singular values, respectively, with \(n_L+n_S<r\). 
The remaining \(r-n_L-n_S\) singular values form the middle cluster. 
Specifically, the small, middle, and large singular values are sampled from three separated intervals
\[
R_S=[\beta_S,\gamma_S],\qquad
R_M=[\beta_M,\gamma_M],\qquad
R_L=[\beta_L,\gamma_L],
\]
respectively, where \(0<\beta_S\), \(\gamma_S<\beta_M\), and \(\gamma_M<\beta_L\). 
The sampled singular values are then sorted in nonincreasing order and used as the diagonal entries of \(D\). 
Given a parameter \(\kappa_M>1\), the middle interval is chosen such that
\(
\gamma_M/\beta_M\leq \kappa_M,
\)
so that the condition number of the middle singular-value cluster is bounded by \(\kappa_M\). 
We refer to such a matrix as an \((m,n,r,n_L,n_S,\kappa_M)\) matrix. 
When \(m=n=r\), this
construction coincides with the \((n,n_L,n_S,\kappa_M)\)
matrix model \cite[Definition~2.4]{amsel2026linear}. 
The second type consists of real-world matrices from the SuiteSparse Matrix Collection \cite{kolodziej2019suitesparse} and LIBSVM \cite{chang2011libsvm}. 
For these datasets, we use only the coefficient matrix \(A\).

We construct a consistent linear system by generating the ground-truth solution as
\(x^\ast=\texttt{randn(n,1)}\) and setting \(b=Ax^\ast\). 
We initialize SC-IS-Krylov with 
\(x^0=A_{\mathcal{I}_p}^{\dagger}b_{\mathcal{I}_p}\), 
and initialize the IS-Krylov method with \(x^0=0\). 
The iterative process is terminated when the squared relative solution error
\(
\operatorname{RSE}
=
\frac{\|x^k-A^\dagger b\|_2^2}{\|A^\dagger b\|_2^2}
\)
falls below a prescribed tolerance. 
In the following experiments, we employ the partition sampling \cite{Zeng2024adaptive,xie2025randomized,sun2025connecting}. 
Given the block size \(q\), we consider a random partition of \([m_r]\). Let
\(\varpi\) be a uniformly random permutation of \([m_r]\), and define
\begin{equation*}
	\begin{aligned}
		&\mathcal I_i
		=
		\{\varpi(k): k=(i-1)q+1,(i-1)q+2,\ldots,iq\},
		\qquad i=1,2,\ldots,t-1,\\
		&\mathcal I_t
		=
		\{\varpi(k): k=(t-1)q+1,(t-1)q+2,\ldots,m_r\},
		\qquad |\mathcal I_t|\leq q.
	\end{aligned}
\end{equation*}
We select the sampling matrix \(S^\top = I_{\mathcal I_i}\)
with probability
\(
\|(A_{\mathcal{I}_r})_{\mathcal I_i}\|_F^2/\|A_{\mathcal{I}_r}\|_F^2. 
\)
We use SC-IS-Krylov-PS and IS-Krylov-PS to denote the SC-IS-Krylov and IS-Krylov methods with partition sampling, respectively.
For a more comprehensive discussion of additional sketching strategies, we refer the reader to, e.g., \cite[Sections 8 and 9]{Martinsson2020Randomized}. 
In practice, the quantity \( \|S_k^\top (A_{\mathcal{I}_r} x^k - b_{\mathcal{I}_r})\|_2 \) is considered zero if it is smaller than \texttt{eps}. Each experiment is repeated over \(20\) independent trials to ensure statistical reliability. The lightly shaded region indicates the range from the minimum to
the maximum values, while the darker shaded region corresponds to the interval between the 25-th and 75-th quantiles.
\subsection{Comparison of selection strategies for \(\mathcal I_p\)}
Recent works \cite{amsel2026linear,Derezinski2023SolvingDL,musco2017spectrum} indicate that the convergence behavior of iterative methods for linear systems is closely related to the distribution of singular values, rather than to the classical condition number alone.
 Motivated by this observation, we first examine
the singular value distribution of \(A_{\mathcal I_r}P\) under five selection strategies for
constructing \(\mathcal I_p\), namely CPQR, SVD, SqNorm, SkCPQR and RBRP.  

For the synthetic matrices $A=UDV^\top$, \(U\) is obtained from the thin QR factorization
\(G=UR\) of a Gaussian random matrix \(G\in\mathbb R^{m\times r}\). 
For any index set \(\mathcal I_p\) with
\(|\mathcal I_p|=m_p\le r\), the submatrix \(G_{\mathcal I_p}\in\mathbb R^{m_p\times r}\) has full row rank
with probability $1$ \cite[proof of Theorem~10.5]{halko2011finding}. 
Since \(R\) is nonsingular with probability $1$,
it follows from \(G_{\mathcal I_p}=U_{\mathcal I_p}R\) that
\(
\operatorname{rank}(U_{\mathcal I_p})
=
\operatorname{rank}(G_{\mathcal I_p})
=
m_p . 
\)
Moreover, since \(D\) is nonsingular and \(V^\top\) has full row rank,
\(DV^\top\) has full row rank. 
Therefore, from $A_{\mathcal I_p}=U_{\mathcal I_p}DV^\top$, we obtain
\(
\operatorname{rank}(A_{\mathcal I_p})=\operatorname{rank}(U_{\mathcal I_p})
=m_p
\)
with probability $1$. 
In our experiments, all selected matrices \(A_{\mathcal I_p}\) are observed to have
full row rank, which implies that $r_p:=\operatorname{rank}(A_{\mathcal{I}_p})=m_p$. 
By Lemma~\ref{lem-interlacing}, the nonzero singular values of \(A_{\mathcal I_r}P\) satisfy
\begin{equation*}
	\sigma_{i+m_p}(A)
	=
	\sigma_{i+r_p}(A)
	\leq
	\sigma_i(A_{\mathcal{I}_r}P)
	\leq
	\sigma_i(A),\ i=1,\ldots,r-m_p.
\end{equation*}
Figure~\ref{ex_fig_1} illustrates these interlacing bounds for different values of \(m_p\). 
The nonzero singular values of \(A_{\mathcal I_r}P\) lie within the shaded region bounded above by \(\sigma_i(A)\) and below by  \(\sigma_{i+m_p}(A)\). 
Compared with the original matrix \(A\), the reduced matrix \(A_{\mathcal I_r}P\) has a more concentrated singular-value distribution.

We next compare the five SC-IS-Krylov-PS variants obtained from the selection strategies introduced in Section~\ref{sec-3.2}, namely SVD-IS-Krylov-PS, SqNorm-IS-Krylov-PS, CPQR-IS-Krylov-PS, SkCPQR-IS-Krylov-PS, and RBRP-IS-Krylov-PS.
Figure~\ref{ex_fig_2} shows the RSE with respect to the number of iterations and CPU time under $m_p=150,300$ and $450$. 
It can be seen from Figure~\ref{ex_fig_2} that the five strategies require comparable numbers of iterations. 
However, SVD-IS-Krylov-PS, CPQR-IS-Krylov-PS and RBRP-IS-Krylov-PS incur higher CPU time compared to the other methods, mainly due to the auxiliary matrix constructions, factorizations, and adaptive
updates required in their row-selection procedures. 
SkCPQR-IS-Krylov-PS reduces the CPU time compared
with CPQR-IS-Krylov-PS by performing pivot selection in a lower-dimensional sketched space. 
Indeed, SqNorm-IS-Krylov-PS is still faster across our experiments, since it only requires squared-row-norm sampling and avoids QR-type factorizations. 
Therefore, we focus our subsequent tests on the SqNorm-IS-Krylov-PS method.

\begin{figure}
	\centering
	\includegraphics[width=0.32\linewidth]{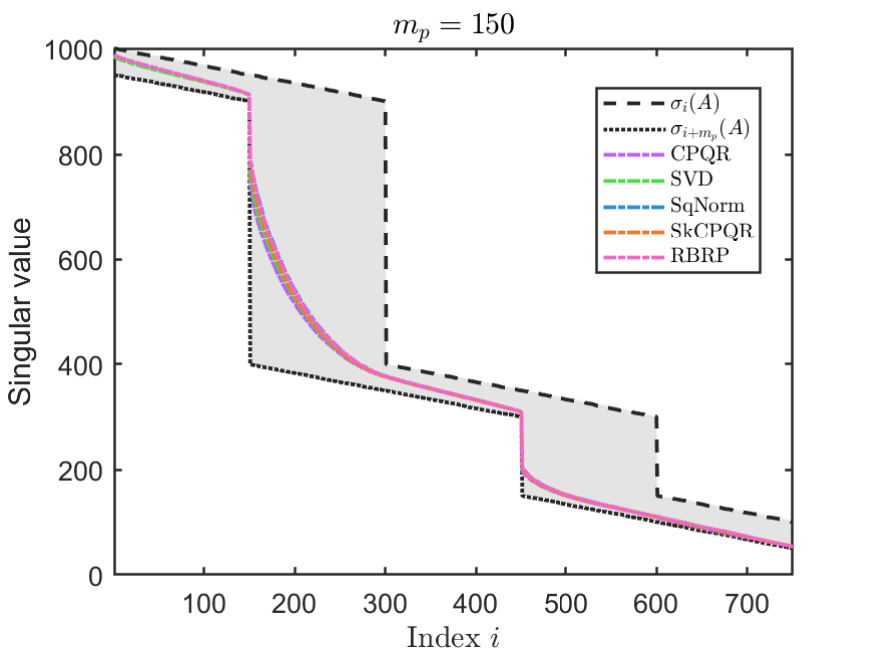}
	\includegraphics[width=0.32\linewidth]{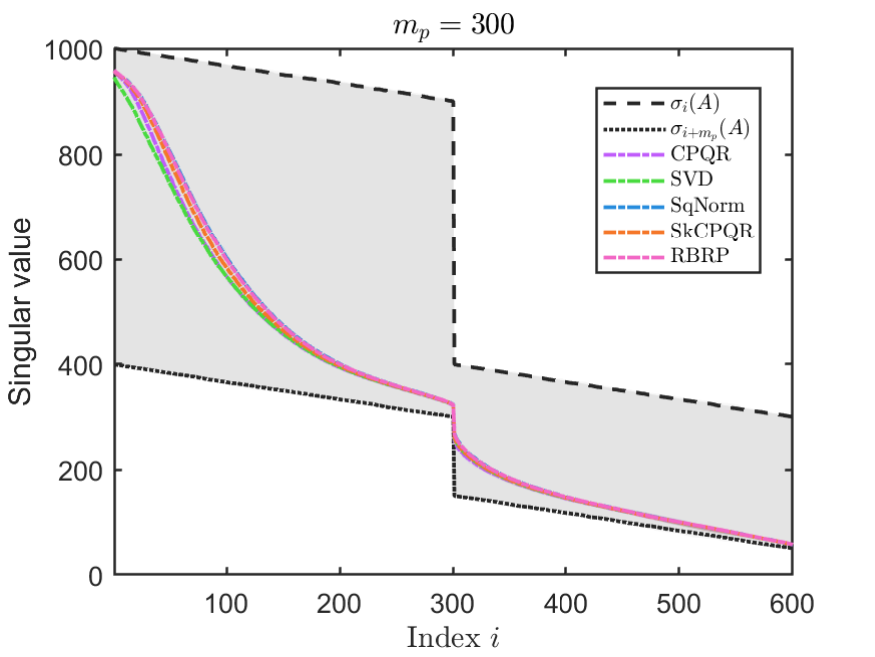}	\includegraphics[width=0.32\linewidth]{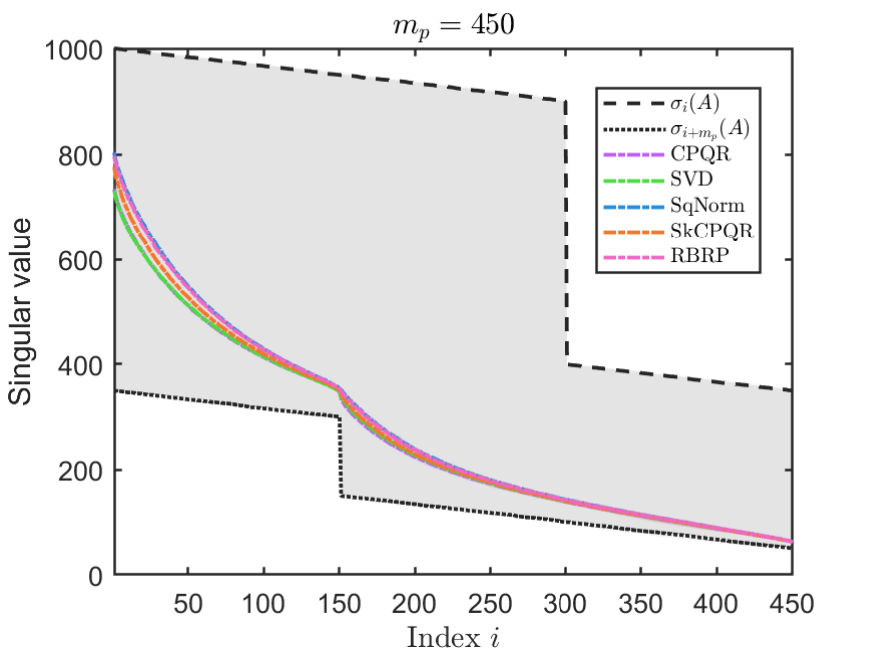}
	\caption{Figures illustrate the singular value distribution of $A_{\mathcal{I}_r}P$ and $A$. The title of each subplot indicates the corresponding values of $m_p$.  The coefficient matrices are generated as \((5000,1000,900,300,300,2)\) matrices with
		\(R_L=[900,1000]\), \(R_M=[300,400]\), and \(R_S=[50,150]\).}
	\label{ex_fig_1}
\end{figure}

\begin{figure}
	\centering
	\includegraphics[width=0.32\linewidth]{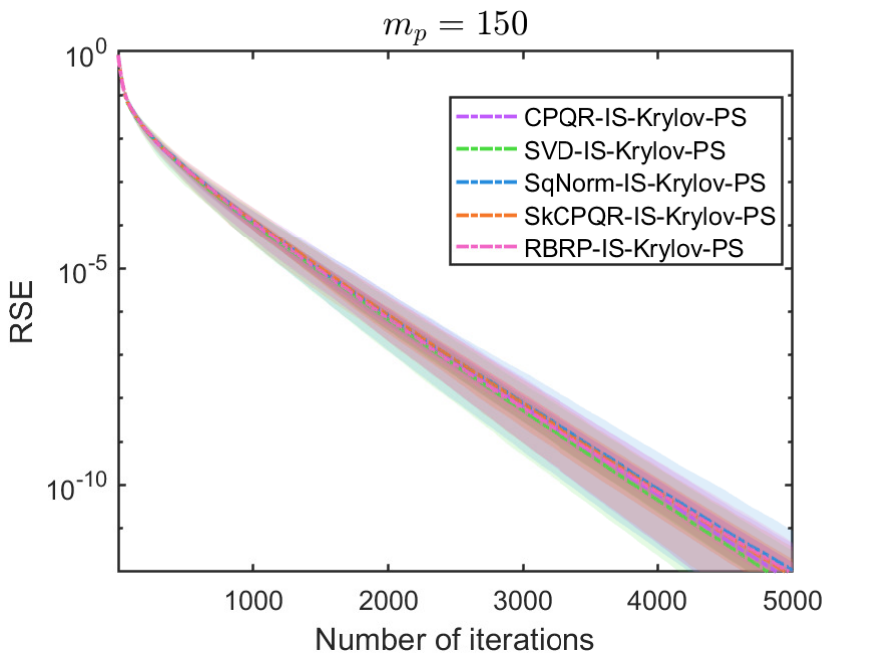}
	\includegraphics[width=0.32\linewidth]{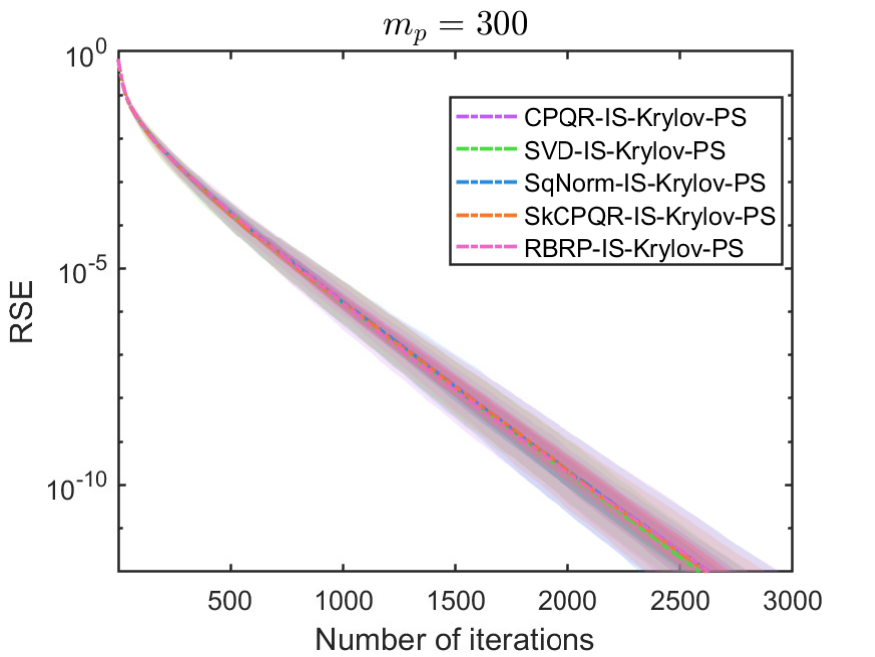}	\includegraphics[width=0.32\linewidth]{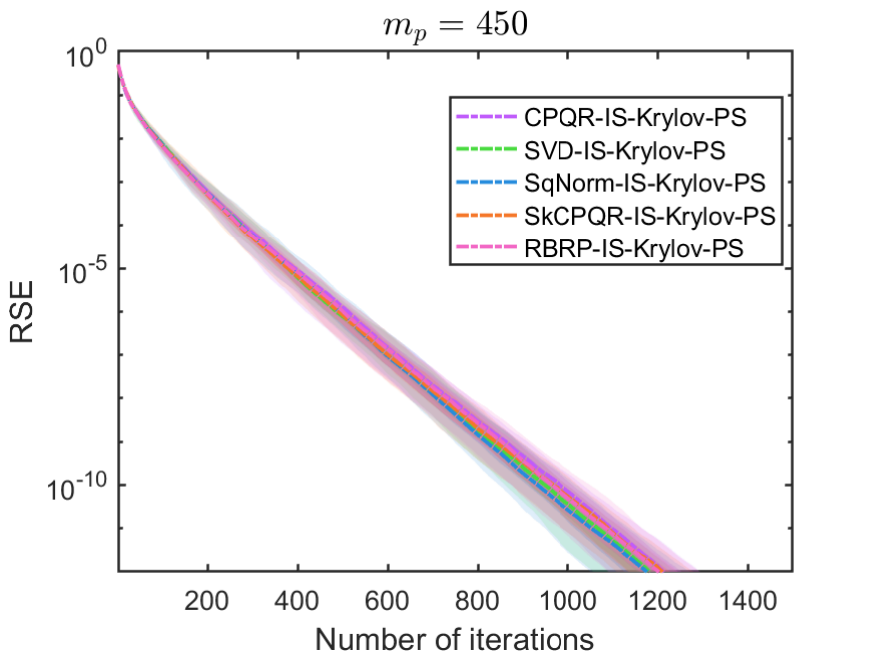}	\includegraphics[width=0.32\linewidth]{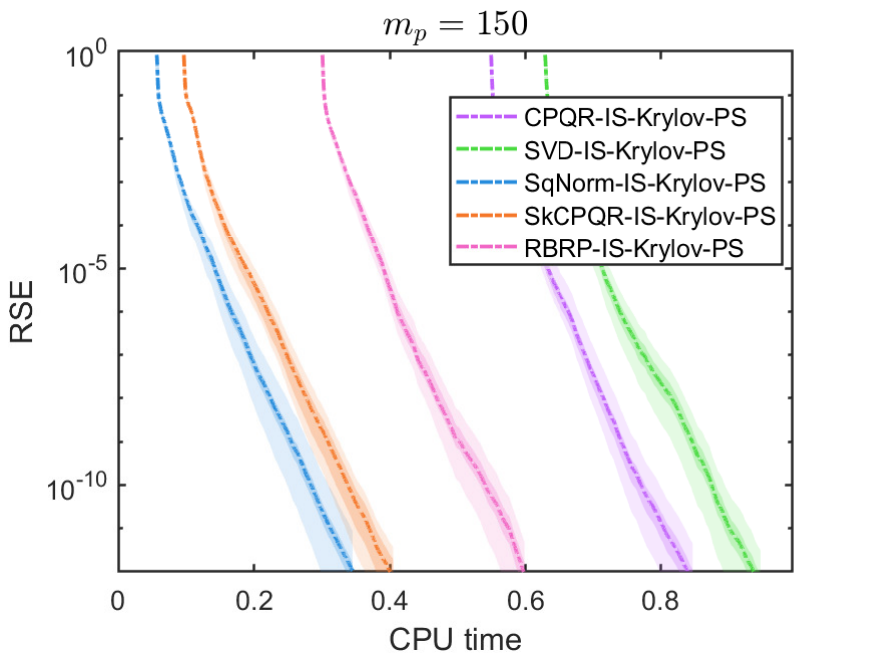}
	\includegraphics[width=0.32\linewidth]{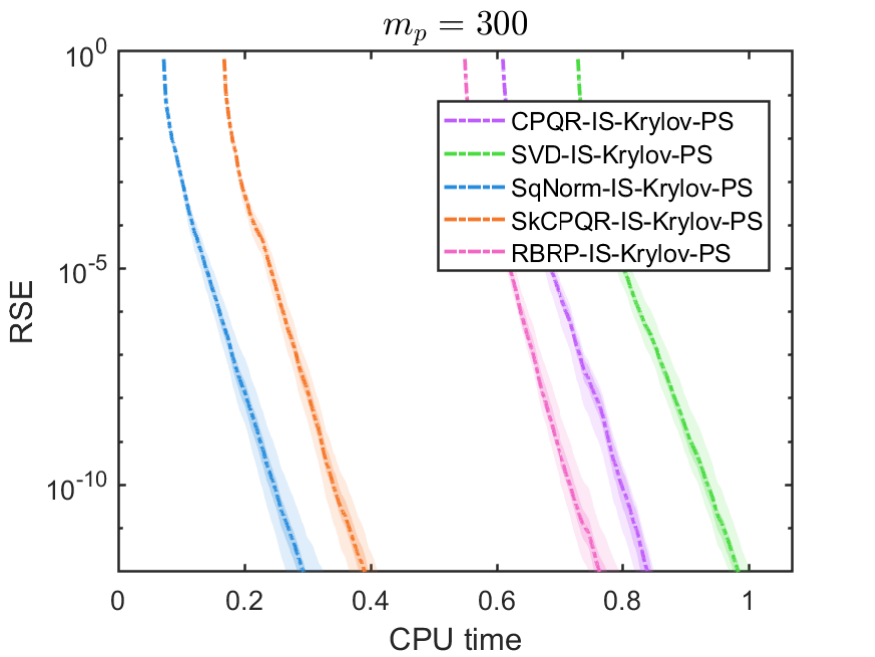}
	\includegraphics[width=0.32\linewidth]{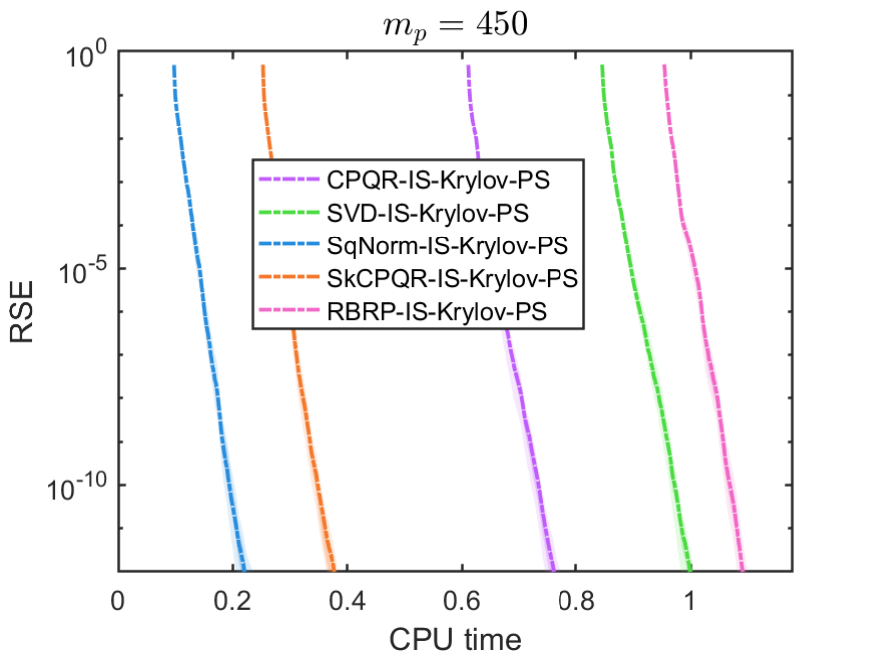}
	\caption{Figures depict the evolution of RSE with respect to the number of iterations (top) and the CPU time (bottom). The title of each subplot indicates the corresponding values of $m_p$.  
	The coefficient matrices are generated as \((5000,1000,900,300,300,2)\) matrices with
	\(R_L=[900,1000]\), \(R_M=[300,400]\), and \(R_S=[50,150]\). 	
	The other parameters are fixed as $q=50$ and $\ell=10$.}
	
	\label{ex_fig_2}
\end{figure}

\subsection{Choice of parameters \(q\), \(\ell\), and \(m_p\)}
In this subsection, we investigate the influence of the block size \(q\), 
the truncation parameter \(\ell\), and the number of selected constraint rows 
\(m_p=|\mathcal I_p|\) on the convergence behavior of SqNorm-IS-Krylov-PS. 
We note that the case \(\ell=1\) corresponds to Algorithm~\ref{Algo-1} with the SqNorm selection strategy and partition sampling, denoted by SqNorm-RIM-PS. 
The performance is measured by both the CPU time and the number of full iterations $k\cdot \frac{q}{m_r}$, where $m_r=m-m_p$. 

It can be seen from the first row of Figure~\ref{ex_fig_3} that, for fixed values of 
\(m_p\) and \(q\), increasing the truncation parameter \(\ell\) reduces the number of full iterations. 
In particular, the case \(\ell=1\), which corresponds to SqNorm-RIM-PS, requires more full iterations than the cases \(\ell>1\). 
Moreover, for fixed \(q\) and \(\ell\), increasing \(m_p\) generally leads to fewer full iterations, indicating that using more selected constraint rows can improve the convergence behavior.
In terms of CPU time, for a fixed \(m_p\), the SqNorm-RIM-PS method  first becomes faster and then slower as the block size \(q\) increases. 
For the SqNorm-IS-Krylov-PS method with \(\ell=2,5,10,\) and \(30\), increasing \(\ell\) generally reduces the CPU time, which shows the benefit of using more orthogonalized historical search directions. 
Moreover, for a fixed \(\ell\), increasing \(m_p\) also decreases the CPU time, which is consistent with the reduction in the number of full iterations. 
These observations indicate that SqNorm-IS-Krylov-PS can substantially improve the efficiency of the basic SqNorm-RIM-PS method. 
When both CPU time and the number of full iterations are taken into account, configurations with \(q=32\) or \(64\), \(\ell=10\) or \(30\), and \(m_p=16\) or \(64\) provide favorable performance, achieving a desirable balance between convergence speed, computational cost, and storage cost.
\begin{figure}
	\centering
	\includegraphics[width=0.32\linewidth]{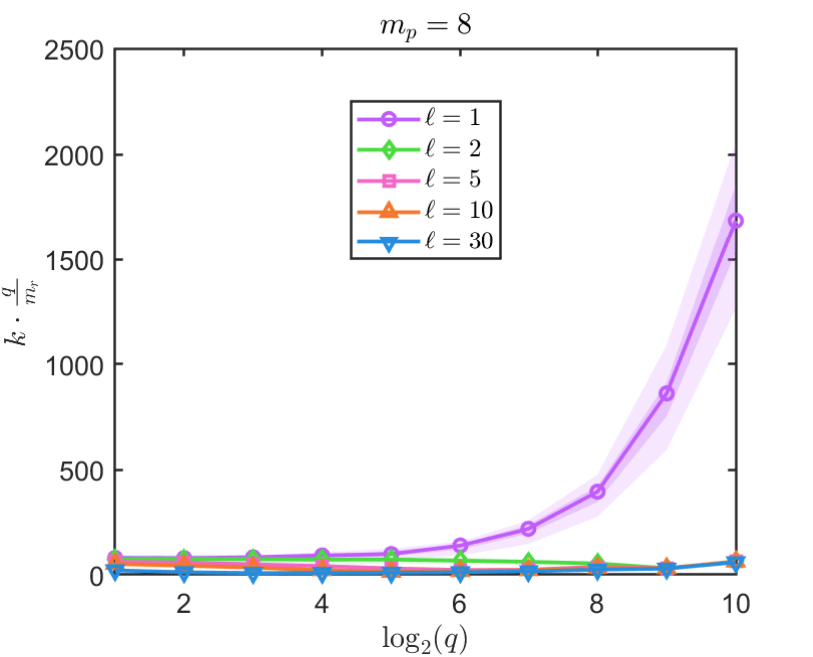}
	\includegraphics[width=0.32\linewidth]{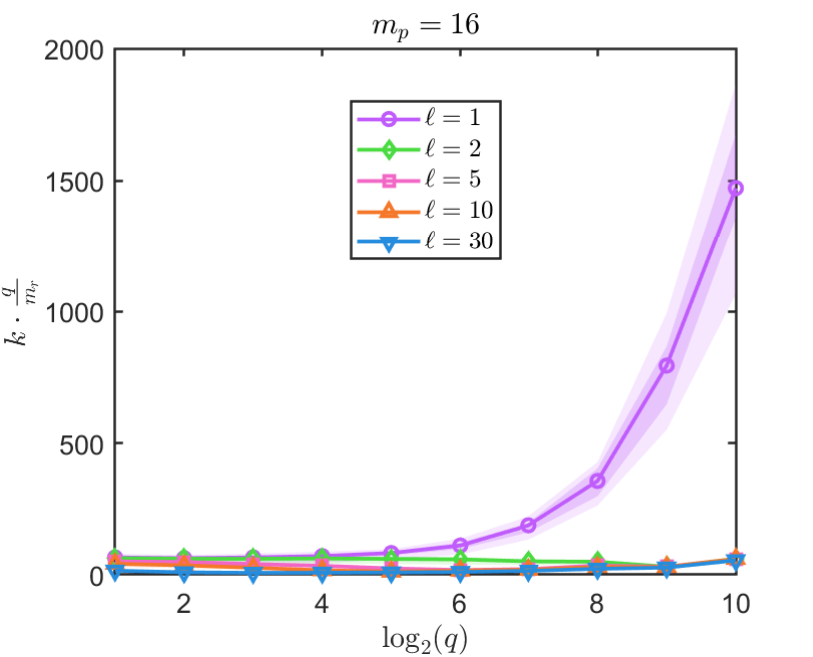}
	\includegraphics[width=0.32\linewidth]{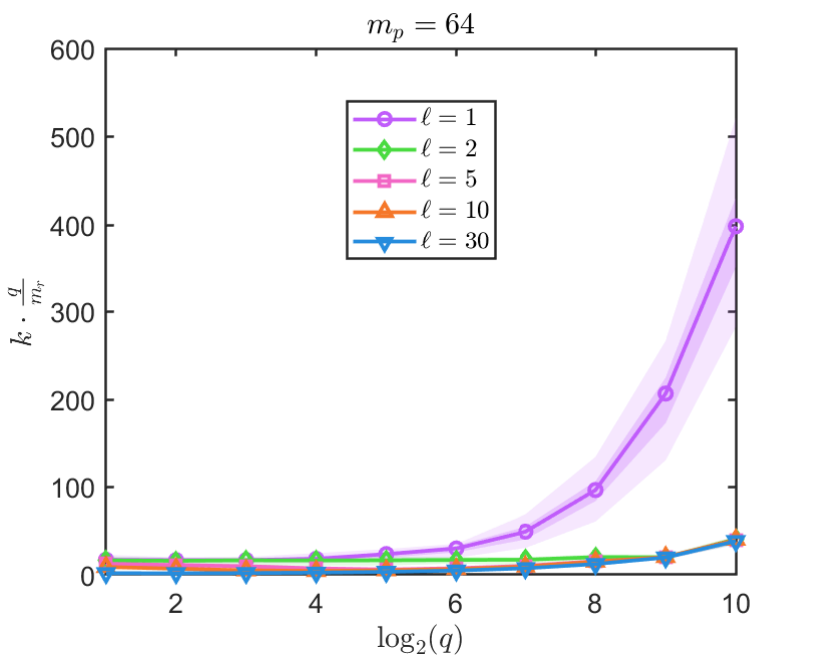}\\
	\includegraphics[width=0.32\linewidth]{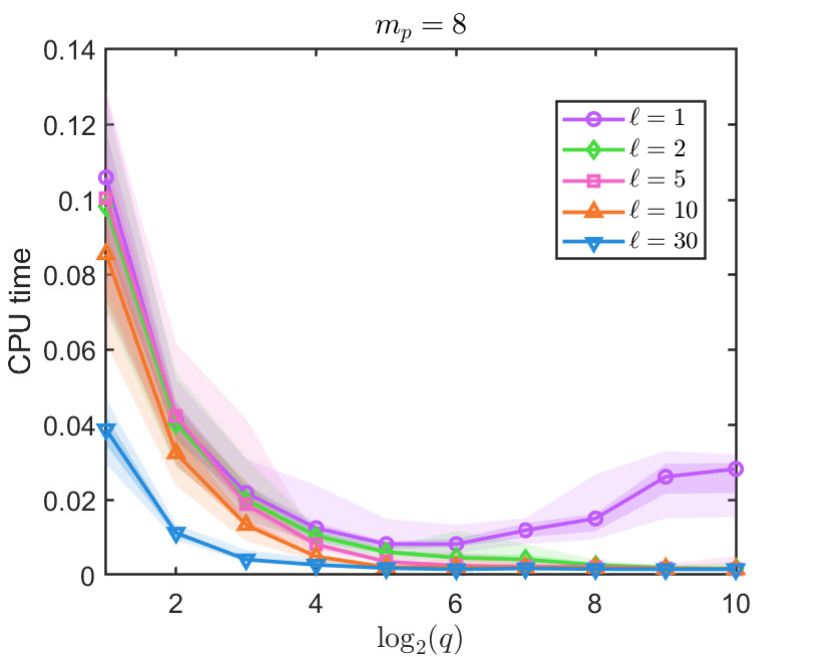}
	\includegraphics[width=0.32\linewidth]{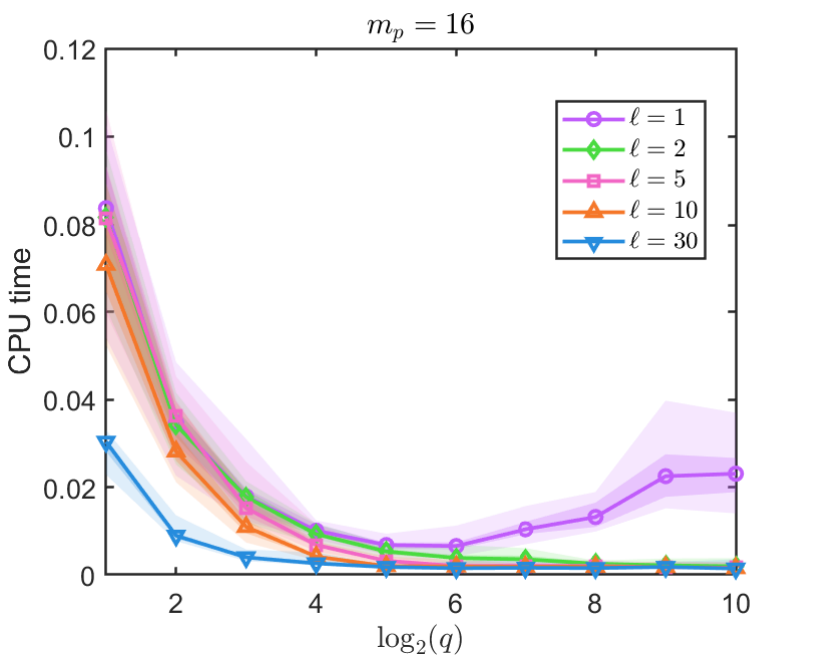}
	\includegraphics[width=0.32\linewidth]{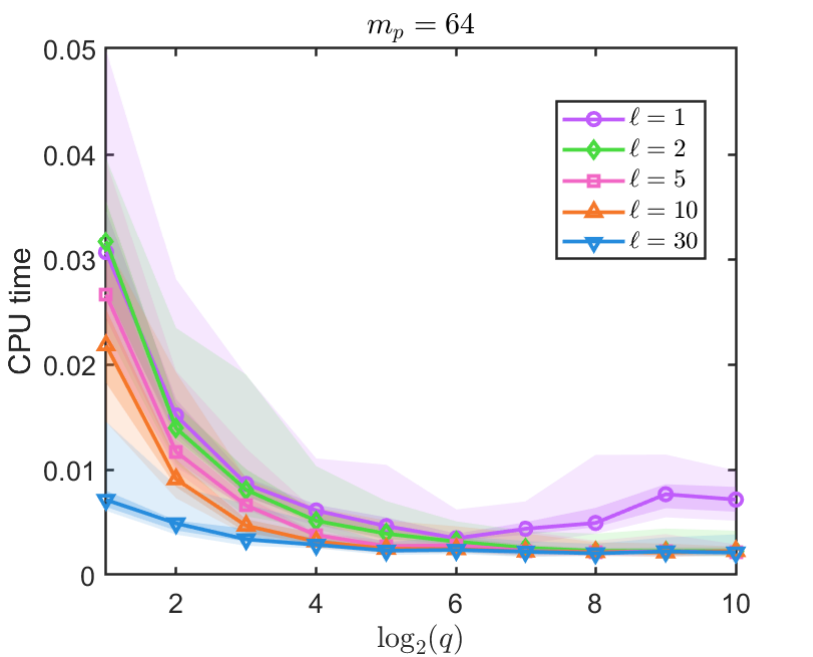}\\
	\caption{Performance of SqNorm-IS-Krylov-PS with different values of \(q\), \(\ell\), and \(m_p\). The top row shows the number of full iterations \(k\cdot \frac{q}{m_r}\), and the bottom row shows the CPU time. Each subplot title indicates the corresponding values of \(m_p\). 
	The coefficient matrices are generated as \((1024,128,128,16,16,2)\) matrices with
	\(R_L=[900,1000]\), \(R_M=[300,400]\), and \(R_S=[50,150]\).} 
	\label{ex_fig_3}
\end{figure}

\subsection{Comparison with IS-Krylov-PS}
We next compare SqNorm-IS-Krylov-PS with its non-preconditioned counterpart,
namely the IS-Krylov-PS method. 
Table~\ref{tab:1} and Figure~\ref{ex_fig_4} report the number of iterations and CPU time of the two
methods on real-world matrices from the SuiteSparse Matrix
Collection~\cite{kolodziej2019suitesparse} and LIBSVM~\cite{chang2011libsvm}.
The SuiteSparse test set consists of \texttt{bibd\_16\_8}, \texttt{crew1},
\texttt{WorldCities}, \texttt{lp\_ship04s}, \texttt{cr42}, \texttt{D\_6},
\texttt{D\_7}, \texttt{rel6}, \texttt{mk12-b2}, and \texttt{abtaha2}. 
The LIBSVM test set consists of \texttt{a9a}, \texttt{aloi},
and \texttt{protein}. These matrices cover full-rank and rank-deficient cases,
as well as overdetermined and underdetermined systems.

Table~\ref{tab:1} shows that SqNorm-IS-Krylov-PS requires fewer iterations than IS-Krylov-PS for all SuiteSparse test matrices. 
This indicates that the subspace-constrained preconditioning process effectively accelerates the convergence of IS-Krylov-PS.
In most cases, the reduced iteration count also leads to a smaller CPU time. 
The only exception is \texttt{rel6}, where the additional preconditioning and
projection costs offset the savings from fewer iterations. 
Figure~\ref{ex_fig_4} reports the corresponding results on the LIBSVM datasets.
For all three datasets, SqNorm-IS-Krylov-PS requires fewer iterations and less
CPU time than the IS-Krylov-PS method.
\begin{table}[h]
	\centering
	\caption{Average number of iterations and CPU time of IS-Krylov-PS and SqNorm-IS-Krylov-PS for matrices from the SuiteSparse Matrix Collection~\cite{kolodziej2019suitesparse}. All computations are terminated once RSE$<10^{-12}$.}
	\begingroup
\small
\setlength{\tabcolsep}{3pt}
		\begin{tabular}{ccccc cccc}
			\toprule
			\multirow{2}{*}{Matrix} & \multirow{2}{*}{$m$} & \multirow{2}{*}{$n$} & \multirow{2}{*}{rank} & \multirow{2}{*}{$m_p$}
			& \multicolumn{2}{c}{IS-Krylov-PS} & \multicolumn{2}{c}{SqNorm-IS-Krylov-PS} \\
			\cmidrule(r){6-7} \cmidrule(r){8-9}
			& & & & & Iter & CPU & Iter & CPU \\
			\midrule
			\texttt{bibd}\_16\_8 & 120   & 12870 & 120  & 50  & 148.35    & 0.0307 & 30.85    & \textbf{0.0250} \\
			\texttt{crew1}       & 135   & 6469  & 135  & 50  & 315.70    & 0.0347 & 90.45    & \textbf{0.0265} \\
			\texttt{WorldCities} & 315   & 100   & 100  & 20  & 551.20    & 0.0024 & 248.25   & \textbf{0.0020} \\
			\texttt{lp\_ship04s} & 402   & 1506  & 360  & 100 & 65783.45  & 1.8168 & 5654.05  & \textbf{0.2912} \\
			\texttt{cr42}        & 905   & 1513  & 905  & 200 & 105582.75 & 3.2831 & 22208.10 & \textbf{1.8309} \\
			\texttt{D\_6}        & 970   & 435   & 339  & 50  & 1167.05   & 0.0148 & 565.10   & \textbf{0.0134} \\
			\texttt{D\_7}        & 1270  & 971   & 629  & 100 & 41135.05  & 0.8451 & 20300.45 & \textbf{0.7593} \\
			\texttt{rel6}        & 2340  & 157   & 137  & 50  & 1154.25   & \textbf{0.0055} & 652.25 & 0.0061 \\
			\texttt{mk12-b2}     & 13860 & 1485  & 1420 & 50  & 1335.90   & 0.3377 & 1288.00  & \textbf{0.3354} \\
			\texttt{abtaha2}     & 37932 & 331   & 331  & 150 & 4350.35   & 0.1158 & 1204.75  & \textbf{0.1043} \\
			\bottomrule
		\end{tabular}
\endgroup
	\label{tab:1}
\end{table}
\begin{figure}
	\centering
	\includegraphics[width=0.32\linewidth]{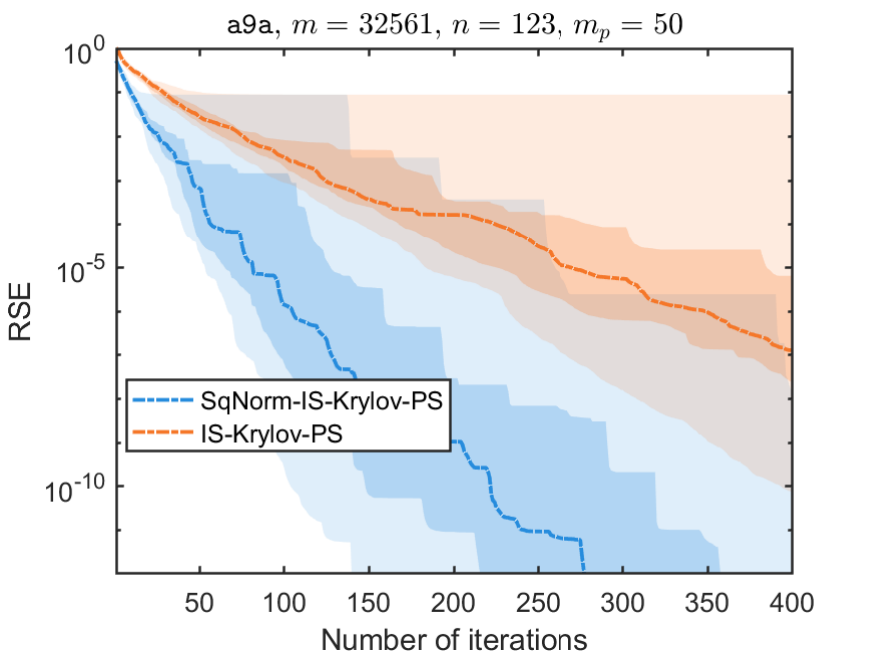}
	\includegraphics[width=0.32\linewidth]{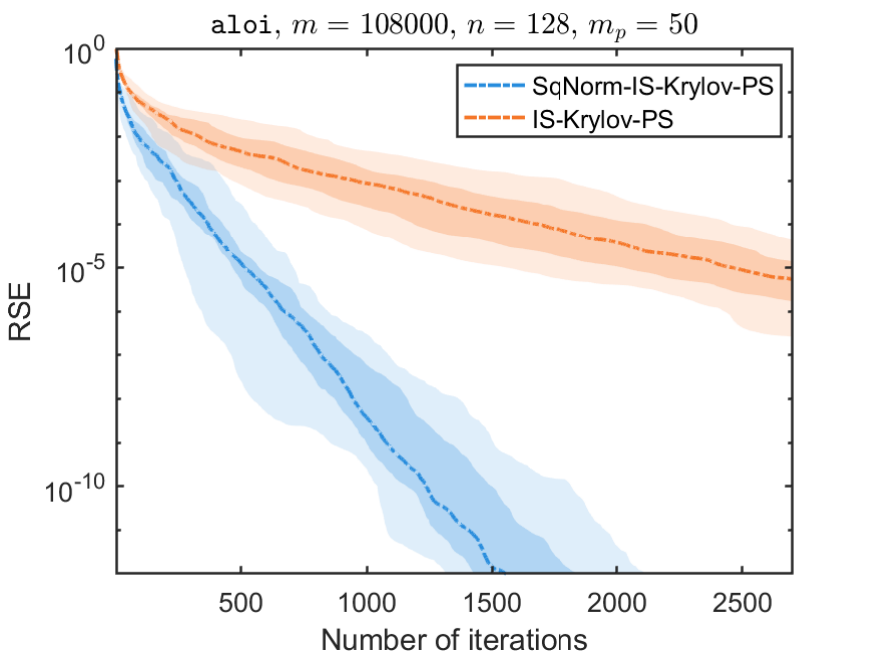}
	\includegraphics[width=0.32\linewidth]{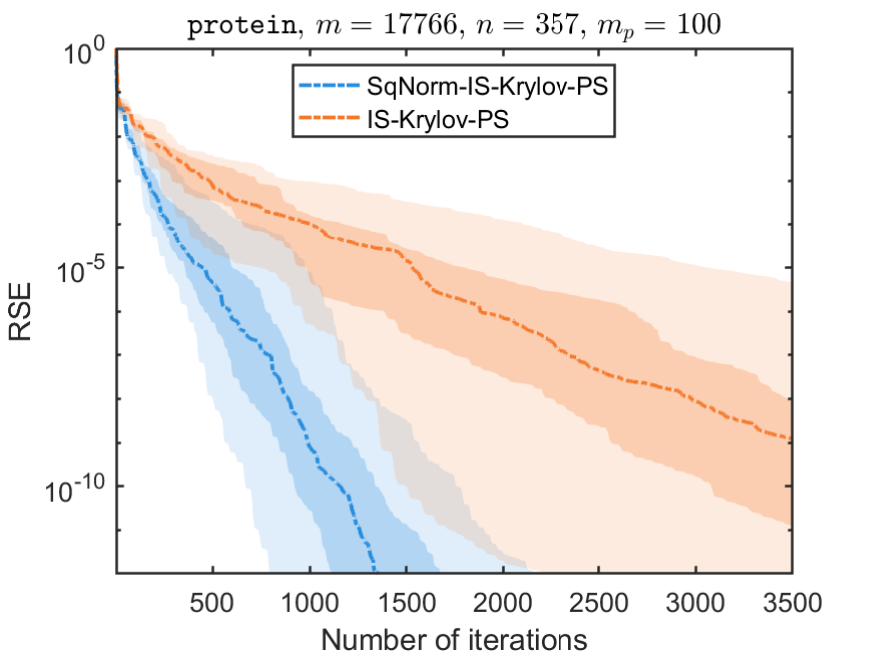}\\
	\includegraphics[width=0.32\linewidth]{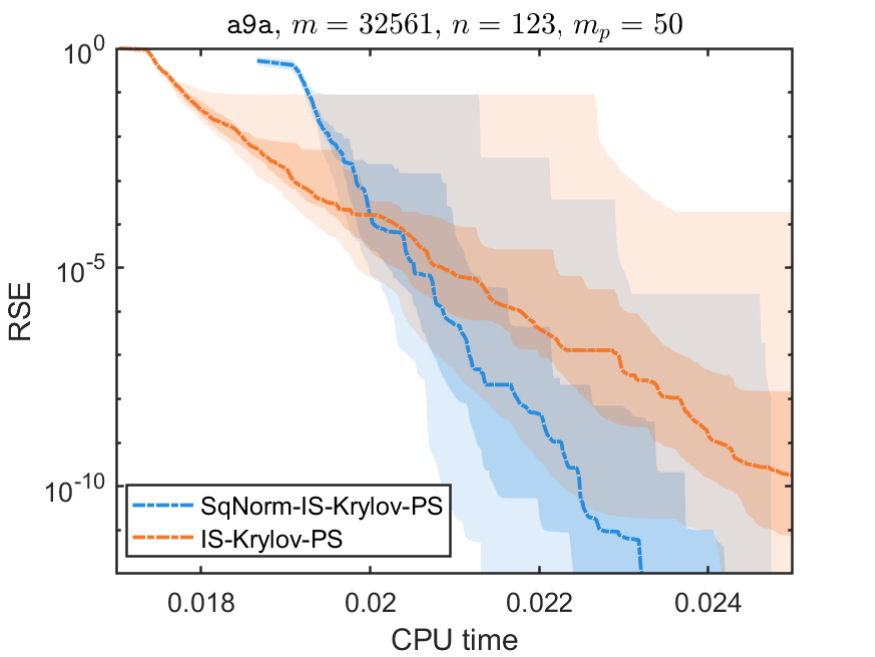}
	\includegraphics[width=0.32\linewidth]{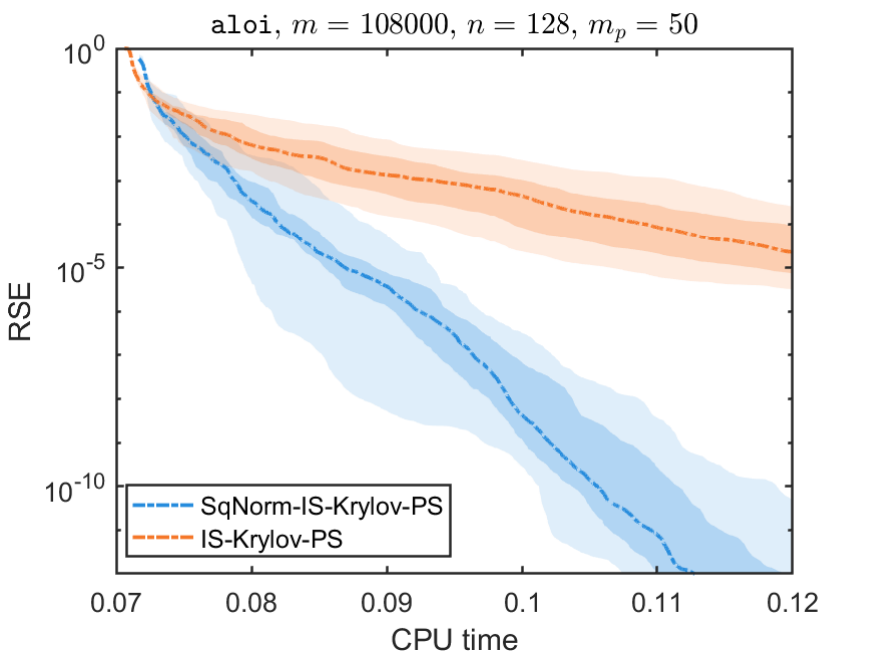}
	\includegraphics[width=0.32\linewidth]{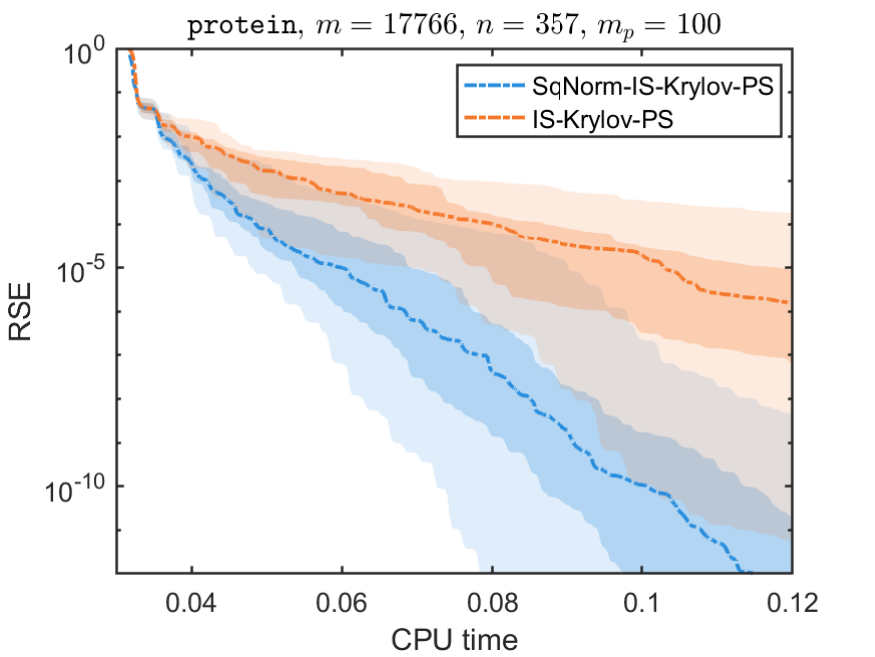}
	\caption{Performance of IS-Krylov-PS and SqNorm-IS-Krylov-PS for linear systems with coefficient matrices from LIBSVM~\cite{chang2011libsvm}. Figures depict the evolution of RSE with respect to the number of iterations and the CPU time. Each plot title indicates the dataset name and data dimensions. We set \(q=300\) and \(\ell=10\).}
	\label{ex_fig_4}
\end{figure}

\subsection{Comparison with \texttt{pinv} and \texttt{lsqminnorm}}
We compare IS-Krylov-PS and SqNorm-IS-Krylov-PS with the MATLAB built-in
functions \texttt{pinv} and \texttt{lsqminnorm}. We generate full-column-rank
matrices \(A\in\mathbb R^{m\times n}\) and vary the number of rows \(m\), while
keeping the number of columns fixed. The exact solution is generated as
\(x^*=\texttt{randn(n,1)}\), and the right-hand side is set to \(b=Ax^*\).
The iterative methods are terminated when their solution accuracy reaches the
accuracy level of the MATLAB built-in solvers.


Figure~\ref{ex_fig_5} compares the CPU time of \texttt{pinv},
\texttt{lsqminnorm}, IS-Krylov-PS, and SqNorm-IS-Krylov-PS as the number of rows
\(m\) increases. 
The results show that the CPU time of \texttt{pinv} increases rapidly with \(m\), especially in the larger-scale settings. 
The function \texttt{lsqminnorm} is competitive for small values of \(m\), but its CPU time increases more noticeably as \(m\) grows. 
For larger \(m\), SqNorm-IS-Krylov-PS generally requires less CPU time than \texttt{lsqminnorm}. 
Moreover, SqNorm-IS-Krylov-PS consistently improves upon IS-Krylov-PS in all three settings, which further demonstrates the benefit of the subspace-constrained preconditioning.

\begin{figure}
	\centering
	\includegraphics[width=0.32\linewidth]{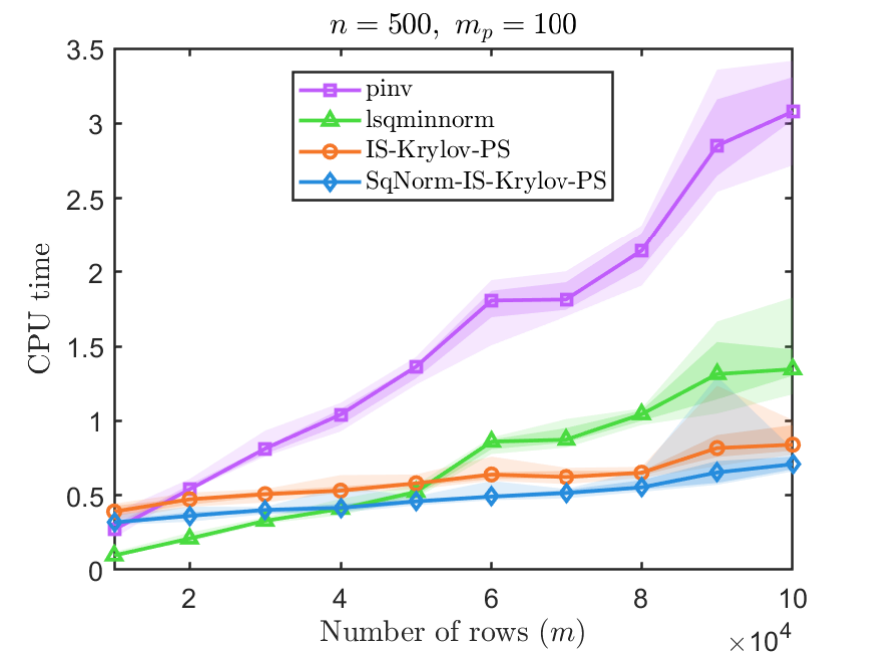}
	\includegraphics[width=0.32\linewidth]{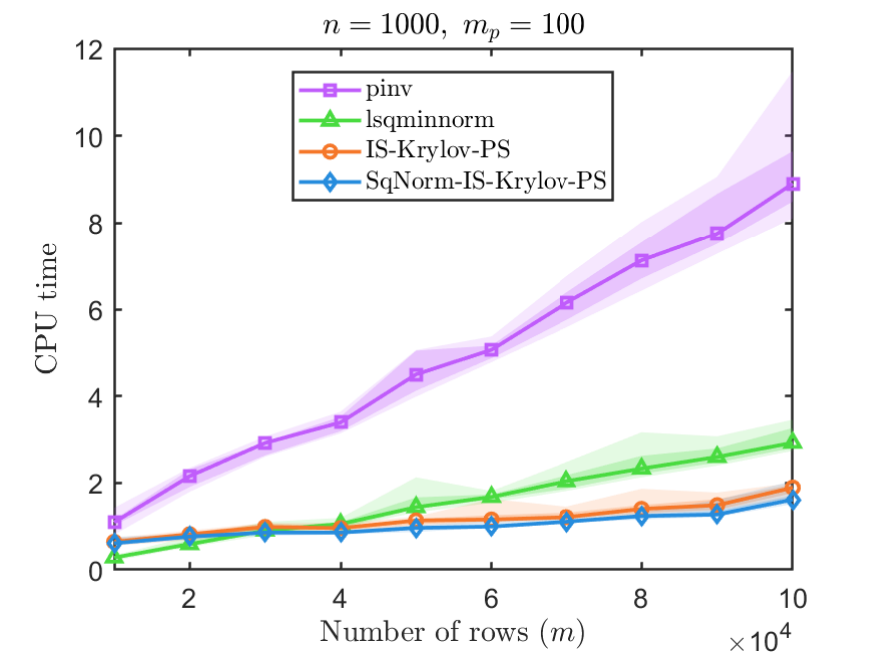}
	\includegraphics[width=0.32\linewidth]{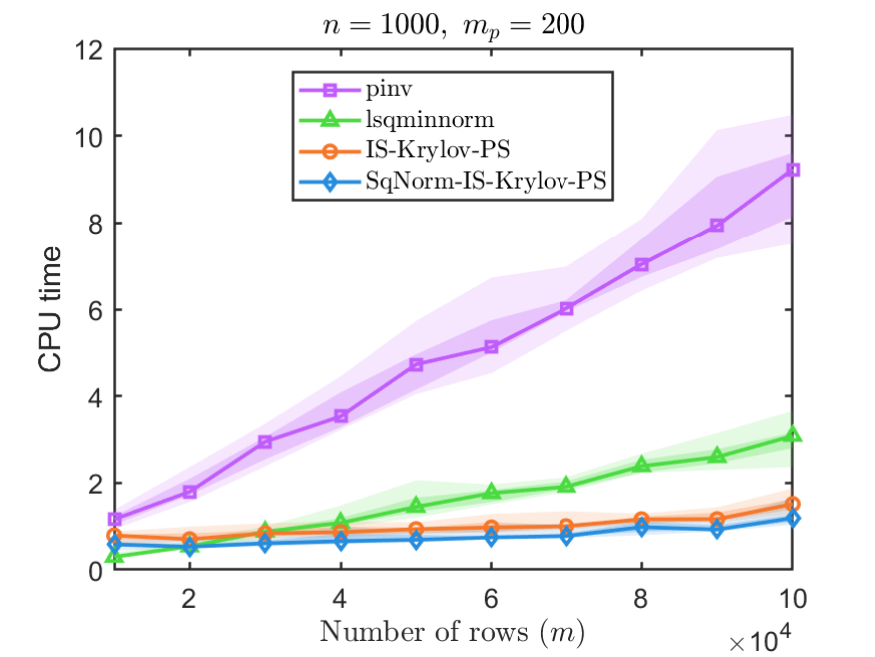}
\caption{The figures illustrate the evolution of CPU time with respect to the number of rows \(m\). The title of each subplot indicates the corresponding values of \(n\) and \(m_p\). 
	The coefficient matrices are generated as \((m,n,n,100,100,2)\) matrices with
	\(R_L=[900,1000]\), \(R_M=[300,400]\), and \(R_S=[50,150]\). The other parameters are fixed as \(\ell=50\) and \(q=30\).} 
	\label{ex_fig_5}
\end{figure}

\section{Concluding remarks}\label{sec-6}
 
This paper further investigated the subspace-constrained preconditioning technique based on a QR-like factorization and verified its applicability to general linear systems, including underdetermined, overdetermined, full-rank, and rank-deficient matrices. A projected randomized iterative method was subsequently adopted to solve the preconditioned linear systems, and it was confirmed that the computational cost could be effectively reduced when the initial point is chosen as a solution of the subsystem \(A_{\mathcal{I}_p} x = b_{\mathcal{I}_p}\). Furthermore, drawing on the principles of minimal error methods and the orthogonal direction method, the proposed subspace-constrained iterative framework was further accelerated, giving rise to a novel subspace-constrained iterative-sketching-based Krylov subspace (SC-IS-Krylov) method. Theoretical analysis showed that the proposed method achieves linear convergence in expectation. Numerical experiments validated our theoretical results and demonstrated the superior computational efficiency of the proposed method.

There are still many possible future avenues of research. Since linear systems arising in practical problems are often inconsistent due to noise \cite{zouzias2013randomized,zeng2025adaptive,zeng2026randomizedcg}, extending the subspace-constrained preconditioning technique to such cases is a valuable direction. The randomized sparse Kaczmarz method proposed in \cite{Schopfer2019Linear,zeng2026stochastic} has been recognized as an effective strategy for obtaining sparse solutions to linear systems. A potential avenue for future research could be the extension of the subspace-constrained preconditioning technique to address sparse recovery problems.

\bibliographystyle{abbrv}
\bibliography{ref}

\section{Appendix. Proof of the main results}
\label{sec-7}
\subsection{Proof of Lemma \ref{lem-interlacing}}
We first introduce the Poincar\'e separation theorem, which provides an interlacing inequality for eigenvalues.
\begin{theorem}[Theorem 1,\cite{scott1985separation}]\label{thm-poincare}
	Let $B\in\mathbb{R}^{n\times n}$ be a symmetric matrix, and let $Q\in\mathbb{R}^{n\times t}$ satisfy $Q^\top Q = I_t$. Then
	\begin{equation}\label{ineq-poincare}
		\lambda_{i+n-t}(B) \leq \lambda_i(Q^\top BQ) \leq \lambda_i(B), \quad i=1,\ldots,t, 
	\end{equation}
	where $\lambda_i(\cdot)$ denotes the $i$-th largest eigenvalue. 
\end{theorem}
\begin{proof}[Proof of Lemma \ref{lem-interlacing}]
	We first prove that $\operatorname{rank}(A_{\mathcal{I}_r}P)=\operatorname{rank}(A)- \operatorname{rank}(A_{\mathcal{I}_p})=r-r_p$. 
	From  $A=L\hat{A}$ in \eqref{eq-Atop-QR-like} with invertible $L$, it holds that
	$\operatorname{Range}(A^\top)=\operatorname{Range}(\hat{A}^\top) = \operatorname{Range}(A_{\mathcal{I}_p}^\top) + \operatorname{Range}(PA_{\mathcal{I}_r}^\top)$. 
	Hence, it suffices to show that $\operatorname{Range}(A_{\mathcal{I}_p}^\top) \cap \operatorname{Range}(PA_{\mathcal{I}_r}^\top) = \{0\}$. For any $x \in \operatorname{Range}(A_{\mathcal{I}_p}^\top) \cap \operatorname{Range}(PA_{\mathcal{I}_r}^\top)$, there exist $\xi \in \mathbb{R}^{m_p}$ and $\eta \in \mathbb{R}^{m_r}$ such that $x = A_{\mathcal{I}_p}^\top \xi = PA_{\mathcal{I}_r}^\top \eta$. 
	Consequently, it holds that
	$$
	\|x\|_2^2 = \langle A_{\mathcal{I}_p}^\top \xi, PA_{\mathcal{I}_r}^\top \eta \rangle = \langle \xi, A_{\mathcal{I}_p} P A_{\mathcal{I}_r}^\top \eta \rangle = 0,
	$$
	which implies that $x = 0$. Since \(x\) is arbitrary, we conclude that $\operatorname{Range}(A_{\mathcal{I}_p}^\top) \cap \operatorname{Range}(PA_{\mathcal{I}_r}^\top) = \{0\}$, which implies that $\operatorname{rank}(A)=\operatorname{rank}(A_{\mathcal{I}_r}P)+ \operatorname{rank}(A_{\mathcal{I}_p})$. 
	Hence, we obtain $\operatorname{rank}(A_{\mathcal{I}_r} P) = r - r_p$. 
	
	Since $A_{\mathcal{I}_p}P=0$, we have 
	\(
	AP=
	\begin{pmatrix}
		0\\
		A_{\mathcal{I}_r}P
	\end{pmatrix},
	\)
	which implies that $AP$ and $A_{\mathcal{I}_r}P$ have the same nonzero singular values. 
	Define the matrix $Q=\begin{pmatrix}
		q_1,\ldots,q_{n-r_p}
	\end{pmatrix}\in\mathbb{R}^{n\times(n-r_p)}$, where the columns form an orthonormal basis of $\operatorname{Null}(A_{\mathcal{I}_p})$.
	Since $P$ is the orthogonal projection operator onto the subspace $\operatorname{Null}(A_{\mathcal{I}_p})$, we have
	$P=QQ^\top$.
	Hence, it holds that
	$$AP(AP)^\top=AQQ^\top QQ^\top A^\top=AQQ^\top A^\top=AQ(AQ)^\top,$$
	which implies that the matrices $AP$ and $AQ$ have the same nonzero singular values. 
	Using Theorem \ref{thm-poincare} with $B=A^\top A$, it holds that
	\begin{equation}\label{ineq-proof1}
		\lambda_{i+r_p}(A^\top A) \leq \lambda_i(Q^\top A^\top AQ) \leq \lambda_i(A^\top A), \quad i=1,\ldots,n-r_p.
	\end{equation}
	Using the equations $\sigma_i(AQ)=\sqrt{\lambda_i(Q^\top A^\top AQ)}$ and $\sigma_i(AQ)=\sigma_i(AP)=\sigma_i(A_{\mathcal{I}_r}P)$, we have 
	\begin{equation*}
		\sigma_{i+r_p}(A) \leq \sigma_i(A_{\mathcal{I}_r}P) \leq \sigma_i(A), \quad i=1,\ldots,n-r_p.
	\end{equation*}
	This completes the proof of the lemma.
\end{proof}
\subsection{Proof of Theorem \ref{Thm-Convergence-SCRIM}}
To facilitate the convergence analysis, we introduce several auxiliary variables and notations associated with the \(k\)-th iteration.
At the \(k\)-th iteration, we consider the product probability space
\(
\left( \prod_{i=0}^k \Omega, \otimes_{i=0}^k \mathcal{F}, \tilde{\mathbf{P}} \right),
\)
where \(\otimes\) denotes the product of \(\sigma\)-algebras and \(\tilde{\mathbf{P}}\) is the corresponding product measure, as defined in \cite[Section 5]{athreya2006measure}.
Let \(\mathcal{B}_k:= \sigma(S_0, S_1, \ldots, S_{k-1})\) be the $\sigma$-algebra generated by the random variables $ (S_0, S_1, \ldots, S_{k-1})$, 
where $\mathcal{B}_0$ is the trivial $\sigma$-algebra (i.e., $\mathcal{B}_0 = \{\emptyset, \Omega\}$). 
We denote the conditional expectation with respect to $\mathcal{B}_k$ as 
\[
\mathbb{E}_{k}[\cdot] := \mathbb{E}[\cdot \mid \mathcal{B}_k].\]
Note that if \(X\) and \(XY\) are integrable random variables, and \(X\) is measurable with respect to  \(\mathcal{B}_k\), then the following identities hold \cite[Proposition 12.1.5 (ii)]{athreya2006measure}:
\begin{equation}\label{prop-Econd}
	\begin{aligned}
		\mathbb{E}[X | \mathcal{B}_k] = X \quad \text{and} \quad \mathbb{E}[XY| \mathcal{B}_k] = X \mathbb{E}[Y| \mathcal{B}_k].
	\end{aligned}
\end{equation}
Let \(\bar S_k\) be a sampling
matrix drawn from the probability space \((\Omega,\mathcal F,\mathbf{P})\). 
We define the auxiliary iterate 
$$\bar x^{k+1}:=x^k-\bar\alpha_k
P A_{\mathcal I_r}^{\top}
\bar S_k\bar S_k^{\top}
(A_{\mathcal I_r}x^k-b_{\mathcal I_r}).$$ 
The step-size $\bar\alpha_k$ is chosen as
\begin{equation}\label{eq-bar-alpha}
	\bar\alpha_k
	=
	\begin{cases}
		(2-\zeta)\frac{
			\|\bar S_k^\top(A_{\mathcal I_r}x^k-b_{\mathcal I_r})\|_2^2
		}{
			\|P A_{\mathcal I_r}^{\top}
			\bar S_k\bar S_k^\top
			(A_{\mathcal I_r}x^k-b_{\mathcal I_r})\|_2^2
		},
		& \text{if }\bar S_k\in\mathcal Q_k,\\[1mm]
		0,
		& \text{if }\bar S_k\in\mathcal Q_k^c,
	\end{cases}
\end{equation}
where $\mathcal{Q}_k$ is defined in \eqref{eq-Q_k}. 
For any measurable set \(\mathcal Q\subseteq\Omega\), we use
\(
\mathbf P(\bar S_k\in\mathcal Q\mid\mathcal B_k)
\)
to denote the conditional probability of the event
\(\{\bar S_k\in\mathcal Q\}\) given \(\mathcal B_k\). The complement of \(\mathcal{Q}\) is defined as \(\mathcal{Q}^c = \Omega \setminus \mathcal{Q}\).
The indicator function $\mathbbm{I}_{\mathcal{Q}}(\cdot)$ of the set $\mathcal{Q}$ is defined as 
\[
\mathbbm{I}_{\mathcal{Q}}(\bar{S}_k)
=
\begin{cases}
	1
	& \text{if }\bar S_k\in\mathcal Q;\\
	0
	& \text{otherwise}.
\end{cases}
\]

\begin{proof}[Proof of Theorem \ref{Thm-Convergence-SCRIM}]
	We first demonstrate that $Ax^k=b$ implies that $x^k=A^\dagger b$. 
	Suppose that $Ax^k=b$ and $Ax^i\neq b$ for $i=0,1,\ldots,k-1$. 
	Note that $x^0=A_{\mathcal{I}_p}^\dagger b_{\mathcal{I}_p}$ and all iterates $x^k$ of Algorithm \ref{Algo-1} belong to $x^0+\operatorname{Range}(PA_{\mathcal{I}_r}^\top)$, which ensures that 
	$$\left(I-(A_{\mathcal{I}_r}P)^\dagger A_{\mathcal{I}_r}P\right)x^k=\left(I-(A_{\mathcal{I}_r}P)^\dagger A_{\mathcal{I}_r}P\right)A_{\mathcal{I}_p}^\dagger b_{\mathcal{I}_p}=A_{\mathcal{I}_p}^\dagger b_{\mathcal{I}_p}.$$ 
	Using the assumption $Ax^k=b$, we obtain $A_{\mathcal{I}_r}Px^k=\hat{b}_{\mathcal{I}_r}$ and 
	\begin{equation*}
		\begin{aligned}
			x^k&=(A_{\mathcal{I}_r}P)^\dagger A_{\mathcal{I}_r}Px^k+A_{\mathcal{I}_p}^\dagger b_{\mathcal{I}_p}=(A_{\mathcal{I}_r}P)^\dagger \hat{b}_{\mathcal{I}_r}+A_{\mathcal{I}_p}^\dagger b_{\mathcal{I}_p}=A^\dagger b,
		\end{aligned}
	\end{equation*}
	where the last equality follows from Lemma \ref{lemma-same-solution-set}. 
	
	We next consider the case where $Ax^k\neq b$. 
	Since \(Ax^k\neq b\) and \(A_{\mathcal I_p}x^k=b_{\mathcal I_p}\), it follows that
	\(
	A_{\mathcal I_r}x^k\neq b_{\mathcal I_r}.
	\)
	We first show that
	\(
	\mathbf P(\bar S_k\in\mathcal Q_k\mid\mathcal B_k)>0.
	\) 
	Suppose, for contradiction, that
	\(\mathbf P(\bar S_k\in\mathcal Q_k\mid\mathcal B_k)=0\).
	Hence,
	\(
	\mathbf P(\bar S_k\in\mathcal Q_k^c\mid\mathcal B_k)=1.
	\)
	Together with the definition of \(\mathcal Q_k\) in \eqref{eq-Q_k}, this implies that
	\(
	\bar S_k^\top(A_{\mathcal I_r}x^k-b_{\mathcal I_r})=0
	\)
	almost surely given \(\mathcal B_k\). 
	Hence, 
	\begin{equation}\label{eq-0514-3}
		\mathbb{E}_{k}\left[\lVert \bar S_{k}^\top(A_{\mathcal{I}_r}x^k-b_{\mathcal{I}_r})\rVert_2^2\right]=0.
	\end{equation}
	However, 
	since \(A_{\mathcal{I}_r}x^k-b_{\mathcal{I}_r}\) is
	\(\mathcal B_k\)-measurable, we have
	\begin{equation}\label{eq-0514-4}
		\begin{aligned}
			\mathbb{E}_{k}\left[
			\left\|
			\bar S_{k}^\top(A_{\mathcal{I}_r}x^k-b_{\mathcal{I}_r})
			\right\|_2^2
			\right]
			&=
			(A_{\mathcal{I}_r}x^k-b_{\mathcal{I}_r})^\top
			\mathbb E_k[\bar S_k\bar S_k^\top]
			(A_{\mathcal{I}_r}x^k-b_{\mathcal{I}_r})\\
			&=(A_{\mathcal{I}_r}x^k-b_{\mathcal{I}_r})^\top
			\mathbb E[\bar S_k\bar S_k^\top]
			(A_{\mathcal{I}_r}x^k-b_{\mathcal{I}_r})>0,
		\end{aligned}
	\end{equation}
	where the second equality follows from the independence of \(\bar S_k\) and
	\(\mathcal B_k\), together with the fact that \(\bar S_k\) is sampled from
	\((\Omega,\mathcal F,\mathbf P)\). 
	The last inequality follows from
	the positive definiteness of
	\(\mathbb E[\bar S_k\bar S_k^\top]\) by Assumption~\ref{assumption1} and the fact that
	\(A_{\mathcal I_r}x^k-b_{\mathcal I_r}\neq0\).  
	The strict positivity in \eqref{eq-0514-4} contradicts \eqref{eq-0514-3}.
	Therefore, we obtain
	\(
	\mathbf P(\bar S_k\in\mathcal Q_k\mid\mathcal B_k)>0.
	\)
	
	Conditional on \(\mathcal B_k\), Step~2 of Algorithm~\ref{Algo-1} samples
	\(S_k\) according to the conditional distribution of \(\bar S_k\) given
	\(\bar S_k\in\mathcal Q_k\). Therefore, the random iterate \(x^{k+1}\) has
	the same conditional distribution as the auxiliary iterate \(\bar x^{k+1}\)
	conditioned on \(\bar S_k\in\mathcal Q_k\). 
	Hence,
	\begin{equation}\label{eq-0514-5}
		\mathbb E_k
		\left[
		\|x^{k+1}-A^\dagger b\|_2^2
		\right]
		=
		\frac{
			\mathbb E_k
			\left[
			\left\|\bar x^{k+1}-A^\dagger b\right\|_2^2
			\mathbbm{I}_{\mathcal Q_k}(\bar S_k)
			\right]
		}{
			\mathbf P(\bar S_k\in\mathcal Q_k\mid\mathcal B_k)
		}.
	\end{equation}
	If the sampling matrix $\bar{S}_k\in\mathcal{Q}_k$, we have
	\begin{equation}\label{eq-0514-1}
		\begin{aligned}
			\lVert \bar{x}^{k+1}-A^\dagger b\rVert_2^2=&\lVert x^k-A^\dagger b-\bar\alpha_kPA_{\mathcal{I}_r}^\top \bar S_{k}\bar S_{k}^\top(A_{\mathcal{I}_r}x^k-b_{\mathcal{I}_r})\rVert_2^2\\
			=&\lVert x^k-A^\dagger b\rVert_2^2-2\bar\alpha_k\langle x^k-A^\dagger b,PA_{\mathcal{I}_r}^\top \bar S_{k}\bar S_{k}^\top(A_{\mathcal{I}_r}x^k-b_{\mathcal{I}_r})\rangle\\
			&+\bar\alpha_k^2\lVert  PA_{\mathcal{I}_r}^\top \bar S_{k}\bar S_{k}^\top(A_{\mathcal{I}_r}x^k-b_{\mathcal{I}_r})\rVert_2^2.
		\end{aligned}
	\end{equation}
	Since $A_{\mathcal{I}_r}P(x^k-A^\dagger b)=A_{\mathcal{I}_r}Px^k-\hat{b}_{\mathcal{I}_r}=A_{\mathcal{I}_r}x^k-b_{\mathcal{I}_r}$ for $k\geq0$, we have 
	\begin{equation}\label{eq-0514-2}
		\langle x^k-A^\dagger b,PA_{\mathcal{I}_r}^\top \bar S_{k}\bar S_{k}^\top(A_{\mathcal{I}_r}x^k-b_{\mathcal{I}_r})\rangle=\lVert \bar S_{k}^\top(A_{\mathcal{I}_r}x^k-b_{\mathcal{I}_r})\rVert_2^2.
	\end{equation}
	Substituting \eqref{eq-0514-2} into \eqref{eq-0514-1} and using the definition of $\bar{\alpha}_k$ in \eqref{eq-bar-alpha}, we obtain
	\begin{equation*}
		\begin{aligned}
			\lVert \bar{x}^{k+1}-A^\dagger b\rVert_2^2
			=&\lVert x^k-A^\dagger b\rVert_2^2-2\bar\alpha_k\lVert \bar S_{k}^\top(A_{\mathcal{I}_r}x^k-b_{\mathcal{I}_r})\rVert_2^2+\bar\alpha_k(2-\zeta)\lVert \bar S_{k}^\top(A_{\mathcal{I}_r}x^k-b_{\mathcal{I}_r})\rVert_2^2\\
			=&\lVert x^k-A^\dagger b\rVert_2^2-\zeta(2-\zeta)\frac{
				\|\bar S_k^\top(A_{\mathcal I_r}x^k-b_{\mathcal I_r})\|_2^4
			}{
				\|P A_{\mathcal I_r}^{\top}
				\bar S_k\bar S_k^\top
				(A_{\mathcal I_r}x^k-b_{\mathcal I_r})\|_2^2
			}\\
			\leq&\lVert x^k-A^\dagger b\rVert_2^2-\zeta(2-\zeta)\frac{\lVert \bar S_{k}^\top(A_{\mathcal{I}_r}x^k-b_{\mathcal{I}_r})\rVert_2^2}{\|\bar S_k^\top A_{\mathcal I_r}P\|_2^2},
		\end{aligned}
	\end{equation*}
	where the last inequality follows from
	\[
	\frac{
		\|\bar S_k^\top(A_{\mathcal I_r}x^k-b_{\mathcal I_r})\|_2^4
	}{
		\|P A_{\mathcal I_r}^{\top}
		\bar S_k\bar S_k^\top
		(A_{\mathcal I_r}x^k-b_{\mathcal I_r})\|_2^2
	}\geq
	\frac{\|\bar S_k^\top(A_{\mathcal I_r}x^k-b_{\mathcal I_r})\|_2^2}{
		\|\bar S_k^\top A_{\mathcal I_r}P\|_2^2
	}.
	\] 
	Since $\mathbb{E}_{k}\left[\lVert x^k-A^\dagger b\rVert_2^2\mathbbm{I}_{\mathcal Q_k}(\bar S_k)\right]=
	\mathbf P(\bar S_k\in\mathcal Q_k\mid\mathcal B_k)\lVert x^k-A^\dagger b\rVert_2^2$, we have
	\begin{equation}\label{eq-0514-8}
		\begin{aligned}
			&\mathbb{E}_{k}\left[\lVert \bar{x}^{k+1}-A^\dagger b\rVert_2^2\mathbbm{I}_{\mathcal Q_k}(\bar S_k)\right]\\
			&\leq
			\mathbf P(\bar S_k\in\mathcal Q_k\mid\mathcal B_k)\lVert x^k-A^\dagger b\rVert_2^2-\zeta(2-\zeta)\mathbb{E}_{k}\left[\frac{\lVert \bar S_{k}^\top(A_{\mathcal{I}_r}x^k-b_{\mathcal{I}_r})\rVert_2^2}{\|\bar S_k^\top A_{\mathcal I_r}P \|_2^2}\mathbbm{I}_{\mathcal Q_k}(\bar S_k)\right].
		\end{aligned}
	\end{equation}
	Dividing both sides of \eqref{eq-0514-8} by
	\(\mathbf P(\bar S_k\in\mathcal Q_k\mid\mathcal B_k)\) and using
	\eqref{eq-0514-5}, we have
	\begin{equation}\label{eq-0514-9}
		\begin{aligned}
			\mathbb E_k
			\left[
			\|x^{k+1}-A^\dagger b\|_2^2
			\right]&\leq
			\lVert x^k-A^\dagger b\rVert_2^2-\frac{\zeta(2-\zeta)}{\mathbf P(\bar S_k\in\mathcal Q_k\mid\mathcal B_k)}\mathbb{E}_{k}\left[\frac{\lVert \bar S_{k}^\top(A_{\mathcal{I}_r}x^k-b_{\mathcal{I}_r})\rVert_2^2}{\|\bar S_k^\top A_{\mathcal I_r}P \|_2^2}\mathbbm{I}_{\mathcal Q_k}(\bar S_k)\right]\\
			&\leq
			\lVert x^k-A^\dagger b\rVert_2^2-\zeta(2-\zeta)\mathbb{E}_{k}\left[\frac{\lVert \bar S_{k}^\top(A_{\mathcal{I}_r}x^k-b_{\mathcal{I}_r})\rVert_2^2}{\|\bar S_k^\top A_{\mathcal I_r}P \|_2^2}\mathbbm{I}_{\mathcal Q_k}(\bar S_k)\right],
		\end{aligned}
	\end{equation}
	where the last inequality follows from $0<
	\mathbf P(\bar S_k\in\mathcal Q_k\mid\mathcal B_k)
	\leq1$. 
	Since
	\(
	\bar S_k^\top(A_{\mathcal I_r}x^k-b_{\mathcal I_r})=0
	\) for \(\bar S_k\in\mathcal Q_k^c\), 
	it holds that 
	\begin{equation}\label{eq-0514-10}
		\begin{aligned}
			&\mathbb E_k\left[
			\frac{
				\left\|\bar S_k^\top(A_{\mathcal I_r}x^k-b_{\mathcal I_r})\right\|_2^2
			}{
				\left\|\bar S_k^\top A_{\mathcal I_r}P\right\|_2^2
			}
			\mathbbm{I}_{\mathcal Q_k}(\bar S_k)
			\right]\\
			&=
			\mathbb E_k\left[
			\frac{
				\left\|\bar S_k^\top(A_{\mathcal I_r}x^k-b_{\mathcal I_r})\right\|_2^2
			}{
				\left\|\bar S_k^\top A_{\mathcal I_r}P\right\|_2^2
			}
			\mathbbm{I}_{\mathcal Q_k}(\bar S_k)
			\right]
			+
			\mathbb E_k\left[
			\frac{
				\left\|\bar S_k^\top(A_{\mathcal I_r}x^k-b_{\mathcal I_r})\right\|_2^2
			}{
				\left\|\bar S_k^\top A_{\mathcal I_r}P\right\|_2^2
			}
			\mathbbm{I}_{\mathcal Q_k^c}(\bar S_k)
			\right] \\
			&=\mathbb E_k\left[
			\frac{
				\left\|\bar S_k^\top(A_{\mathcal I_r}x^k-b_{\mathcal I_r})\right\|_2^2
			}{
				\left\|\bar S_k^\top A_{\mathcal I_r}P\right\|_2^2
			}		
			\right],
		\end{aligned}
	\end{equation}
	where the last equality follows from the fact that $\mathbbm{I}_{\mathcal Q_k}(\bar S_k)+\mathbbm{I}_{\mathcal Q_k^c}(\bar S_k)=1$ for any $\bar{S}_k\in\Omega$. 
	Since \(A_{\mathcal I_r}x^k-b_{\mathcal I_r}\) is \(\mathcal B_k\)-measurable, we have
	\[
	\begin{aligned}
		\mathbb E_k\left[
		\frac{
			\left\|
			\bar S_k^\top(A_{\mathcal I_r}x^k-b_{\mathcal I_r})
			\right\|_2^2
		}{
			\left\|
			\bar S_k^\top A_{\mathcal I_r}P
			\right\|_2^2
		}
		\right]
		&=
		(A_{\mathcal I_r}x^k-b_{\mathcal I_r})^\top
		\mathbb E_k\left[
		\frac{
			\bar S_k\bar S_k^\top
		}{
			\left\|
			\bar S_k^\top A_{\mathcal I_r}P
			\right\|_2^2
		}
		\right]
		(A_{\mathcal I_r}x^k-b_{\mathcal I_r}).
	\end{aligned}
	\]
	Moreover, since \(\bar S_k\) is independent of \(\mathcal B_k\) and samples from the probability space $(\Omega,\mathcal{F},\mathbf{P})$, it holds that
	\[
	\mathbb E_k\left[
	\frac{
		\bar S_k\bar S_k^\top
	}{
		\left\|
		\bar S_k^\top A_{\mathcal I_r}P
		\right\|_2^2
	}
	\right]
	=
	\mathbb E\left[
	\frac{
		\bar S_k\bar S_k^\top
	}{
		\left\|
		\bar S_k^\top A_{\mathcal I_r}P
		\right\|_2^2
	}
	\right]
	=
	H.
	\]
	Therefore, since $A_{\mathcal{I}_r}Px^k-\hat{b}_{\mathcal{I}_r}=A_{\mathcal{I}_r}x^k-b_{\mathcal{I}_r}$ for $k\geq0$, we have
	\begin{equation}\label{eq-0514-11}
		\begin{aligned}
			\mathbb E_k\left[
			\frac{
				\left\|
				\bar S_k^\top(A_{\mathcal I_r}x^k-b_{\mathcal I_r})
				\right\|_2^2
			}{
				\left\|
				\bar S_k^\top A_{\mathcal I_r}P
				\right\|_2^2
			}
			\right]
			=
			\|
			A_{\mathcal I_r}x^k-b_{\mathcal I_r}\|_H^2=\|
			A_{\mathcal{I}_r}Px^k-\hat{b}_{\mathcal{I}_r}
			\|_H^2.
		\end{aligned}
	\end{equation}
	Substituting \eqref{eq-0514-10} and \eqref{eq-0514-11} into \eqref{eq-0514-9}, we have 
	\begin{equation}\label{eq-0514-12}
		\begin{aligned}
			\mathbb E_k
			\left[
			\|x^{k+1}-A^\dagger b\|_2^2
			\right]
			&\leq
			\lVert x^k-A^\dagger b\rVert_2^2-\zeta(2-\zeta)\|
			A_{\mathcal{I}_r}Px^k-\hat{b}_{\mathcal{I}_r}
			\|_H^2\\
			&\leq\left(1-
			\zeta(2-\zeta)\sigma_{\min}^2(H^{\frac{1}{2}}A_{\mathcal{I}_r}P)\right)\| x^k-A^\dagger b\|_2^2,
		\end{aligned}
	\end{equation}
	where the inequality follows from  $H$ is positive definite and $x^k-A^\dagger b\in\operatorname{Range}(PA_{\mathcal{I}_r}^\top)$, which we proceed to verify below by induction. 
	When $k=0$, it holds that $x^0-A^\dagger b=A_{\mathcal{I}_p}^\dagger b_{\mathcal{I}_p}-A^\dagger b=-(A_{\mathcal{I}_r}P)^\dagger \hat{b}_{\mathcal{I}_r}\in\operatorname{Range}(PA_{\mathcal{I}_r}^\top)$. 
	Assume that $x^{k-1}-A^\dagger b \in \operatorname{Range}(PA_{\mathcal{I}_r}^\top)$. 
	It holds that $x^k-A^\dagger b=x^{k-1}-A^\dagger b-\alpha_{k-1}PA_{\mathcal{I}_r}^\top S_{k-1}S_{k-1}^\top(A_{\mathcal{I}_r}x^{k-1}-b_{\mathcal{I}_r})\in\operatorname{Range}(PA_{\mathcal{I}_r}^\top)$. 
	Hence, we conclude that $x^k-A^\dagger b\in \operatorname{Range}(PA_{\mathcal{I}_r}^\top)$ for $k\geq0$. 
	Taking the full expectation on both sides of \eqref{eq-0514-12} and applying
	the resulting inequality recursively, we obtain
	\begin{equation*}
		\begin{aligned}
			\mathbb{E}\left[\lVert x^{k+1}-A^\dagger b\rVert_2^2\right]
			&\leq\rho^{k+1}\lVert x^0-A^\dagger b\rVert_2^2,
		\end{aligned}
	\end{equation*}
	where $\rho$ is defined in \eqref{eq-rho}. 
	This completes the proof of this theorem.
\end{proof}
\subsection{Proof of Proposition \ref{thm-con-rate-factor} and Lemma \ref{col-surrogate}}
To prove the Proposition \ref{thm-con-rate-factor}, we first introduce the change-of variables theorem.
\begin{theorem}[Theorem 3.5.16,\cite{goswami2025measure}]
	\label{Thm-change-of-variable}
	Let \((\tilde{\Omega}, \tilde{\mathcal{F}}, \tilde{\mathbf{P}})\) be a measure space. 
	Suppose that \(T:(\tilde{\Omega},\tilde{\mathcal F})\to(\Omega,\mathcal F)\) is measurable, and define the measure \(\mathbf P\) on \((\Omega,\mathcal F)\) by
	\(\mathbf P(B)=\tilde{\mathbf P}(T^{-1}(B))\) for all \(B\in\mathcal F\).
	Then, for any extended-real-valued measurable function \(g\) on \(\Omega\), it holds that
	\[
	\int_{\Omega} g \, d\mathbf{P}
	=
	\int_{\tilde{\Omega}} g \circ T \, d\tilde{\mathbf{P}}.
	\]
\end{theorem}

\begin{proof}[Proof of Proposition \ref{thm-con-rate-factor}]
	To demonstrate that $\rho\leq\tilde{\rho}$, it is sufficient to show that \(
	\sigma_{\min}^2(\tilde{H}^{1/2} A) \leq \sigma_{\min}^2(H^{1/2} A_{\mathcal{I}_r} P)
	\). 
	Since \(\mathbb{E}[\tilde{S} \tilde{S}^\top]\) is positive definite and has principal submatrix $\mathbb{E}[SS^\top]$ , it follows that \(\mathbb{E}[S S^\top]\) is also symmetric positive definite. 
	As a result, both
	\[
	\tilde{H} = \mathbb{E} \left[ \frac{\tilde{S} \tilde{S}^\top}{\lVert \tilde{S}^\top A \rVert_2^2} \right]
	\quad \text{and} \quad
	H = \mathbb{E} \left[ \frac{S S^\top}{\lVert S^\top A_{\mathcal{I}_r} P \rVert_2^2} \right]
	\]
	are symmetric positive definite by Lemma~\ref{lem-pd}. 
	For any fixed \(x \in \mathbb{R}^n\), we define the function 
	\(
	g(S) := \frac{x^\top P A_{\mathcal{I}_r}^\top S S^\top A_{\mathcal{I}_r} P x}{\lVert S^\top A_{\mathcal{I}_r} P \rVert_2^2},S\in\Omega.
	\)
	From the block form of \(\tilde{S}\) in \eqref{eq-S}, we have the identity 
	$\Tilde{S}^\top AP=\begin{pmatrix}
		\tilde{S}_{\mathcal{I}_p}^\top&S^\top
	\end{pmatrix}\begin{pmatrix}
		0\\
		A_{\mathcal{I}_r} P
	\end{pmatrix}=S^\top A_{\mathcal{I}_r} P$, which leads to
	\(
	P A^\top \tilde{S} \tilde{S}^\top A P = P A_{\mathcal{I}_r}^\top S S^\top A_{\mathcal{I}_r} P.
	\)
	Hence, for any $\tilde{S}\in\tilde{\Omega}$, with \(S=T(\tilde S)\), we have
	\[
	g \circ T(\tilde{S})=g(S)=\frac{x^\top P A_{\mathcal{I}_r}^\top S S^\top A_{\mathcal{I}_r} P x}{\lVert S^\top A_{\mathcal{I}_r} P \rVert_2^2}= \frac{x^\top P A^\top \tilde{S} \tilde{S}^\top A P x}{\lVert \tilde{S}^\top A P \rVert_2^2}.
	\]
	Applying Theorem~\ref{Thm-change-of-variable}, we obtain
	\begin{equation*}
		\int_{\Omega} \frac{x^\top P A_{\mathcal{I}_r}^\top S S^\top A_{\mathcal{I}_r} P x}{\lVert S^\top A_{\mathcal{I}_r} P \rVert_2^2}  d\mathbf{P}
		=
		\int_{\tilde{\Omega}} \frac{x^\top P A^\top \tilde{S} \tilde{S}^\top A P x}{\lVert \tilde{S}^\top A P \rVert_2^2}  d\tilde{\mathbf{P}}.
	\end{equation*}
	Consequently, the minimum nonzero singular value of \(H^{\frac{1}{2}} A_{\mathcal{I}_r} P\) satisfies
	\begin{equation*}
		\begin{aligned}
			\sigma_{\min}^2(H^{\frac{1}{2}} A_{\mathcal{I}_r}P)
			&= \min_{\lVert x \rVert_2 = 1, x \in \operatorname{Range}\left((H^{\frac{1}{2}} A_{\mathcal{I}_r} P)^\top\right)} \int_{\Omega} \frac{x^\top P A_{\mathcal{I}_r}^\top S S^\top A_{\mathcal{I}_r} P x}{\lVert S^\top A_{\mathcal{I}_r} P \rVert_2^2}  d\mathbf{P} \\
			&= \min_{\lVert x \rVert_2 = 1, x \in \operatorname{Range}\left((H^{\frac{1}{2}} A_{\mathcal{I}_r}P)^\top\right)} \int_{\tilde{\Omega}} \frac{x^\top PA^\top \tilde{S} \tilde{S}^\top A Px}{\lVert \tilde{S}^\top A P \rVert_2^2}  d\tilde{\mathbf{P}} \\
			&= \min_{\lVert x \rVert_2 = 1, x \in \operatorname{Range}(PA_{\mathcal{I}_r}^\top)} \int_{\tilde{\Omega}} \frac{x^\top PA^\top \tilde{S} \tilde{S}^\top A Px}{\lVert \tilde{S}^\top A P \rVert_2^2}  d\tilde{\mathbf{P}} 
		\end{aligned}
	\end{equation*}
	where the last equation follows from $\operatorname{Range}\left((H^\frac{1}{2}A_{\mathcal{I}_r}P)^\top)\right)=\operatorname{Range}(PA_{\mathcal{I}_r}^\top)$ as $H$ is positive definite. 
	For any $x\in\operatorname{Range}(PA_{\mathcal{I}_r}^\top)$, it holds that $Px=x$ due to $P^2=P$. 
	Since $P$ is an orthogonal projection matrix, then we obtain $\lVert P\rVert_2=1$ and $\lVert \tilde{S}^\top AP\rVert_2^2\leq\lVert\tilde{S}^\top A\rVert_2^2\lVert P\rVert_2^2=\lVert \tilde{S}^\top A\rVert_2^2$ for any $\tilde{S}\in\tilde{\Omega}$.
	Then, it holds that 
	\begin{equation*}
		\begin{aligned}
			\sigma_{\min}^2(H^{\frac{1}{2}} A_{\mathcal{I}_r}P)
			&\geq \min_{\lVert x \rVert_2 = 1, x \in \operatorname{Range}(P A_{\mathcal{I}_r}^\top)} \int_{\tilde{\Omega}} \frac{x^\top A^\top \tilde{S} \tilde{S}^\top A x}{\lVert \tilde{S}^\top A \rVert_2^2}  d\tilde{\mathbf{P}}.
		\end{aligned}
	\end{equation*}
	We next show that $\operatorname{Range}(PA_{\mathcal{I}_r}^\top)\subseteq\operatorname{Range}(A^\top)$.
	For any $y\in\operatorname{Range}(PA_{\mathcal{I}_r}^\top)$, there exists $\eta\in\mathbb{R}^{m_r}$ such that $y=PA_{\mathcal{I}_r}^\top \eta=A_{\mathcal{I}_r}^\top \eta-A_{\mathcal{I}_p}^\dagger A_{\mathcal{I}_p}A_{\mathcal{I}_r}^\top \eta$
	From the equations $P=I-A_{\mathcal{I}_p}^\dagger A_{\mathcal{I}_p}$ and $(A_{\mathcal{I}_p}^\dagger A_{\mathcal{I}_p})^\top=A_{\mathcal{I}_p}^\dagger A_{\mathcal{I}_p}$, 
	we obtain
	\(y=A_{\mathcal{I}_r}^\top \eta-A_{\mathcal{I}_p}^\top(A_{\mathcal{I}_p}^\dagger)^\top A_{\mathcal{I}_r}^\top \eta\), which implies that $y\in\operatorname{Range}(A^\top)$.
	Since \(y\in\operatorname{Range}(PA_{\mathcal I_r}^\top)\) is arbitrary, it follows that $\operatorname{Range}(PA_{\mathcal{I}_r}^\top)\subseteq\operatorname{Range}(A^\top)$.
	Hence, we have
	\begin{equation*}
		\begin{aligned}
			\sigma_{\min}^2(H^{\frac{1}{2}} A_{\mathcal{I}_r}P)
			&\geq \min_{\lVert x \rVert_2 = 1, x \in \operatorname{Range}(A^\top)} \int_{\tilde{\Omega}} \frac{x^\top A^\top \tilde{S} \tilde{S}^\top A x}{\lVert \tilde{S}^\top A \rVert_2^2}  d\tilde{\mathbf{P}} \\
			&= \min_{\lVert x \rVert_2 = 1, x \in \operatorname{Range}\left((\tilde{H}^\frac{1}{2}A)^\top\right)} \int_{\tilde{\Omega}} \frac{x^\top A^\top \tilde{S} \tilde{S}^\top A x}{\lVert \tilde{S}^\top A \rVert_2^2}  d\tilde{\mathbf{P}}=\sigma_{\min}^2(\tilde{H}^{1/2} A),
		\end{aligned}
	\end{equation*}
	where the first equality follows from $\operatorname{Range}(A^\top)=\operatorname{Range}\left((\tilde{H}^\frac{1}{2}A)^\top\right)$ as $\tilde{H}$ is positive definite. 
	In conclusion, we have proved $\sigma_{\min}^2(\tilde{H}^{1/2} A) \leq \sigma_{\min}^2(H^{1/2} A_{\mathcal{I}_r} P)$, which implies that $\rho\leq\tilde{\rho}$. 
\end{proof}
\begin{proof}[Proof of Lemma \ref{col-surrogate}]
	We first show that $PA_{\mathcal{I}_r}^\top HA_{\mathcal{I}_r} P\succeq\frac{\lambda_{\min}(\bar{H})}{\lVert A_{\mathcal{I}_r} P\rVert_2^2}PA_{\mathcal{I}_r}^\top A_{\mathcal{I}_r} P$, where $A\succeq B$ means that $A-B$ is positive semidefinite.
	Since 
	\[
	H=\mathbb{E}\!\left[\frac{SS^\top}{\|S^\top A_{\mathcal{I}_r}P\|_2^2}\right]
	\;\succeq\;
	\frac{1}{\|A_{\mathcal{I}_r}P\|_2^2}\mathbb{E}\!\left[\frac{SS^\top}{\|S\|_2^2}\right]
	=\frac{1}{\|A_{\mathcal{I}_r}P\|_2^2}\bar{H}\succeq \frac{\lambda_{\min}(\bar{H})}{\|A_{\mathcal{I}_r}P\|_2^2}I\succeq \frac{\lambda_{\min}(\bar{H})}{\|A_{\mathcal{I}_r}P\|_F^2}I,
	\]
	where $H$ and $\bar{H}$ are defined in \eqref{eq-H}, it follows that $PA_{\mathcal{I}_r}^\top HA_{\mathcal{I}_r} P\succeq\frac{\lambda_{\min}(\bar{H})}{\lVert A_{\mathcal{I}_r} P\rVert_F^2}PA_{\mathcal{I}_r}^\top A_{\mathcal{I}_r} P$.
	Then, for any $x\in\mathbb{R}^n$, we have 
	$$x^\top PA_{\mathcal{I}_r}^\top HA_{\mathcal{I}_r} Px\geq\frac{\lambda_{\min}(\bar{H})}{\lVert A_{\mathcal{I}_r} P\rVert_F^2}x^\top PA_{\mathcal{I}_r}^\top A_{\mathcal{I}_r} Px.$$
	Hence, we obtain
	\begin{equation*}
		\begin{aligned}
			\sigma_{\min}^2(H^{\frac{1}{2}}A_{\mathcal{I}_r}P)
			&=\min_{\lVert x \rVert_2 = 1, x \in \operatorname{Range}(P A_{\mathcal{I}_r}^\top H^{1/2})}x^\top PA_{\mathcal{I}_r}^\top HA_{\mathcal{I}_r}Px\\
			&\geq \frac{\lambda_{\min}(\bar{H})}{\lVert A_{\mathcal{I}_r} P\rVert_F^2}\min_{\lVert x \rVert_2 = 1, x \in \operatorname{Range}(PA_{\mathcal{I}_r}^\top H^{1/2})}x^\top PA_{\mathcal{I}_r}^\top A_{\mathcal{I}_r} Px\\
			&= \frac{\lambda_{\min}(\bar{H})}{\lVert A_{\mathcal{I}_r} P\rVert_F^2}\min_{\lVert x \rVert_2 = 1, x \in \operatorname{Range}( PA_{\mathcal{I}_r}^\top)}x^\top PA_{\mathcal{I}_r}^\top A_{\mathcal{I}_r} Px\\
			&=\lambda_{\min}(\bar{H})\frac{\sigma_{\min}^2(A_{\mathcal{I}_r}P)}{\lVert A_{\mathcal{I}_r}P\rVert_F^2},
		\end{aligned}
	\end{equation*}
	where the third equation follows from  $\operatorname{Range}(PA_{\mathcal{I}_r}^\top H^{1/2})=\operatorname{Range}(PA_{\mathcal{I}_r}^\top)$ as the invertible of $H$. 
	This completes the proof of this lemma.
\end{proof}

\subsection{Proof of Lemma \ref{lem-pkneq0}}
We first introduce a useful lemma.
\begin{lemma}\label{lem-truncated-gs-orth}
	Let \(\{p^k\}_{k\geq0}\)
	be generated by \eqref{iter-pk} with $\ell\geq2$. 
	When $k\geq1$, 
	if \(p^i\neq0\) for
	\(i=0,1,\ldots,k\), then 
	\(
	\langle p^\mu,p^\nu\rangle=0\) for  
	\(j_{k,\ell}\leq \mu<\nu\leq k.
	\)
\end{lemma}
\begin{proof}[Proof of Lemma \ref{lem-truncated-gs-orth}]
	We prove by induction on \(k\geq1\) that
	\(\{p^\mu\}_{\mu=j_{k,\ell}}^k\) is orthogonal, i.e.,
	\(
	\langle p^\mu,p^\nu\rangle=0\)
	for \( j_{k,\ell}\le \mu<\nu\le k .
	\)
	For $k=1$, we have
	\[
	\langle p^0,p^1\rangle
	=
	\langle p^0,d^1\rangle
	-
	\frac{\langle p^0,d^1\rangle}{\|p^0\|_2^2}
	\|p^0\|_2^2
	=0.
	\]
	Assume that for some $k\geq1$, $\langle p^\mu,p^\nu\rangle=0$ for 
	$j_{k-1,\ell}\leq \mu<\nu\leq k-1$. 
	Since \(j_{k,\ell}\ge j_{k-1,\ell}\), the induction hypothesis
	implies that
	\[
	\langle p^\mu,p^\nu\rangle=0
	\quad
	\text{for all } j_{k,\ell}\le \mu<\nu\le k-1 .
	\]
	It remains to show that
	\(
	\langle p^\mu,p^k\rangle=0\)
	for all \(\mu=j_{k,\ell},\ldots,k-1 .
	\)
	From the definition of \(p^k\) in \eqref{iter-pk}, for any
	\(\mu=j_{k,\ell},\ldots,k-1\), we have
	\[
	\begin{aligned}
		\langle p^k,p^\mu\rangle
		=
		\langle d^k,p^\mu\rangle
		-
		\sum_{\nu=j_{k,\ell}}^{k-1}
		\frac{\langle d^k,p^\nu\rangle}{\|p^\nu\|_2^2}
		\langle p^\nu,p^\mu\rangle  
		&=
		\langle d^k,p^\mu\rangle
		-
		\frac{\langle d^k,p^\mu\rangle}{\|p^\mu\|_2^2}
		\|p^\mu\|_2^2
		=0,
	\end{aligned}
	\]
	where the second equality follows from the induction hypothesis. 
	This completes the proof of the orthogonal property. 
\end{proof}
\begin{proof}[Proof of Lemma \ref{lem-pkneq0}]
	We first show by induction that $x^k\in\mathcal{X}_p=\{x\in\mathbb R^n\mid A_{\mathcal I_p}x=b_{\mathcal I_p}\}$ for $k\geq0$. Since \(x^0=A_{\mathcal I_p}^\dagger b_{\mathcal I_p}\), we have
	\(x^0\in\mathcal X_p\). Suppose that \(x^k\in\mathcal X_p\) for some $k\geq0$. 
	Since $d^k\in \operatorname{Range}(PA_{\mathcal I_r}^\top)\subseteq
	\operatorname{Range}(P)=\operatorname{Null}(A_{\mathcal I_p})$, it follows from the definition of \(p^k\) in \eqref{iter-pk} that
	\(p^k\in\operatorname{Null}(A_{\mathcal I_p})\).
	According to the definition of $x^{k+1}$ in \eqref{iter-xk+1-cg}, we have $x^{k+1}\in x^k+\operatorname{span}\{p^k\}\subseteq\mathcal{X}_p$. 
	This completes the induction.
	
	For any iterate $x^k\in\mathcal{X}_p$, we have $Px^k=x^k-A_{\mathcal{I}_p}^\dagger A_{\mathcal{I}_p}x^k=x^k-A_{\mathcal{I}_p}^\dagger b_{\mathcal{I}_p}$, which yields  $A_{\mathcal{I}_r}Px^k-\hat{b}_{\mathcal{I}_r}=A_{\mathcal{I}_r}x^k-A_{\mathcal{I}_r}A_{\mathcal{I}_p}^\dagger b_{\mathcal{I}_p}-\hat{b}_{\mathcal{I}_r}=A_{\mathcal{I}_r}x^k-b_{\mathcal{I}_r}$. 
	Therefore, if $S_i^\top(A_{\mathcal{I}_r}x^i-b_{\mathcal{I}_r})\neq0$ for $i=0,1,\ldots,k$, we have 
	$S_i^\top(A_{\mathcal{I}_r}Px^i-\hat{b}_{\mathcal{I}_r})=S_i^\top(A_{\mathcal{I}_r}x^i-b_{\mathcal{I}_r})\neq0$. 
	By Lemma \ref{lemma-meanwhile}, it follows that 
	$d^i=-PA_{\mathcal{I}_r}^\top S_iS_i^\top(A_{\mathcal{I}_r}Px^i-\hat{b}_{\mathcal{I}_r})\neq0$ for $i=0,1,\ldots,k$. 
	If $\ell=1$, it follows from the definition of $p^k$ in \eqref{iter-pk} that $p^k=d^k\neq0$ and 
	\begin{equation*}
		\begin{aligned}
			\langle p^{k},x^k-A^\dagger b\rangle
			=
			\langle d^{k},x^k-A^\dagger b\rangle=-\lVert S_k^\top(A_{\mathcal{I}_r}Px^k-\hat{b}_{\mathcal{I}_r})\rVert_2^2=-\lVert S_k^\top(A_{\mathcal{I}_r}x^k-b_{\mathcal{I}_r})\rVert_2^2.
		\end{aligned}
	\end{equation*}
	
	We next demonstrate by induction that for $\ell>1$, $p^k\neq0$ and $\langle p^k,x^k-A^\dagger b\rangle=-\lVert S_k^\top(A_{\mathcal{I}_r}x^k-b_{\mathcal{I}_r})\rVert_2^2$. 
	For $k=0$, it holds that $p^0=d^0\neq0$ and 
	\begin{equation*}
		\begin{aligned}
			\langle p^{0},x^0-A^\dagger b\rangle
			=
			\langle d^{0},x^0-A^\dagger b\rangle=-\lVert S_0^\top(A_{\mathcal{I}_r}Px^0-\hat{b}_{\mathcal{I}_r})\rVert_2^2=-\lVert S_0^\top(A_{\mathcal{I}_r}x^0-b_{\mathcal{I}_r})\rVert_2^2.
		\end{aligned}
	\end{equation*}
	For $k=1$, it follows from the optimality condition of
	\eqref{iter-xk+1-cg} that
	\(
	\langle p^{0},x^1-A^\dagger b\rangle=0.
	\)
	Thus, 
	\begin{equation*}
		\begin{aligned}
			\langle p^{1},x^1-A^\dagger b\rangle&
			=\langle d^{1},x^1-A^\dagger b\rangle
			-\frac{\langle d^1,p^0\rangle}{\lVert p^0\rVert_2^2}\langle p^{0},x^1-A^\dagger b\rangle\\
			&=\langle d^{1},x^1-A^\dagger b\rangle
			=-\lVert S_1^\top(A_{\mathcal{I}_r}Px^1-\hat{b}_{\mathcal{I}_r})\rVert_2^2=-\lVert S_1^\top(A_{\mathcal{I}_r}x^1-b_{\mathcal{I}_r})\rVert_2^2\neq0,
		\end{aligned}
	\end{equation*}
	which ensures that $p^1\neq0$. 
	
	We next consider the case where $k\geq2$. 
	Assume that $p^i\neq0$ for $i=0,1,\ldots,k-1$. 
	By the definition of $p^k$ in \eqref{iter-pk}, we have
	\begin{equation}\label{eq-pk-inner-general}
		\begin{aligned}
			\langle p^{k},x^k-A^\dagger b\rangle
			&=
			\langle d^{k},x^k-A^\dagger b\rangle
			-
			\sum_{i=j_{k,\ell}}^{k-1}
			\frac{\langle d^k,p^i\rangle}{\lVert p^i\rVert_2^2}
			\langle p^{i},x^k-A^\dagger b\rangle.
		\end{aligned}
	\end{equation}
	We now show that $\langle p^{i},x^k-A^\dagger b\rangle=0$ for $i=j_{k,\ell},\ldots,k-1$. 
	For \(i=k-1\), it follows from the optimality condition of
	\eqref{iter-xk+1-cg} that
	\(
	\langle p^{k-1},x^k-A^\dagger b\rangle=0.
	\)
	For any \(i=j_{k,\ell},\ldots,k-2\), it follows that 
	\[
	x^k-A^\dagger b
	=
	\sum_{j=i+1}^{k-1}(x^{j+1}-x^j)
	+
	x^{i+1}-A^\dagger b.
	\]
	Taking the inner product with \(p^i\) on both sides, we have
	\begin{equation*}
		\begin{aligned}
			\langle p^{i},x^k-A^\dagger b\rangle
			&=
			\sum_{j=i+1}^{k-1}
			\langle p^{i},x^{j+1}-x^j\rangle
			+
			\langle p^{i},x^{i+1}-A^\dagger b\rangle=\sum_{j=i+1}^{k-1}
			\langle p^{i},x^{j+1}-x^j\rangle,
		\end{aligned}
	\end{equation*}
	where the last equation follows from the optimality condition of \eqref{iter-xk+1-cg}, which gives
	\(
	\langle p^{i},x^{i+1}-A^\dagger b\rangle=0.
	\)
	Since \(x^{j+1}-x^j\in\operatorname{span}\{p^j\}\) and 
	\(\langle p^i,p^j\rangle=0\) by Lemma \ref{lem-truncated-gs-orth}, we obtain $\langle p^{i},x^{j+1}-x^j\rangle=0$ for $j=i+1,\ldots,k-1$. 
	Thus, we have $\langle p^{i},x^k-A^\dagger b\rangle=0$ for $i=j_{k,\ell},\ldots,k-1$. 
	Substituting this equation into \eqref{eq-pk-inner-general}, we obtain 
	\begin{equation*}
		\begin{aligned}
			\langle p^{k},x^k-A^\dagger b\rangle
			=
			\langle d^{k},x^k-A^\dagger b\rangle=-\lVert S_k^\top(A_{\mathcal{I}_r}Px^k-\hat{b}_{\mathcal{I}_r})\rVert_2^2=-\lVert S_k^\top(A_{\mathcal{I}_r}x^k-b_{\mathcal{I}_r})\rVert_2^2\neq0,
		\end{aligned}
	\end{equation*}
	which ensures that $p^k\neq0$. 
	Hence, we obtain that $p^k\neq0$ for $k\geq0$. 
	By Lemma \ref{lem-truncated-gs-orth}, 
	if \(\ell\ge2\) and $k\geq1$, then
	\(
	\langle p^\mu,p^\nu\rangle=0\)
	for \( j_{k,\ell}\le \mu<\nu\le k .
	\)
	This completes the proof of the Lemma. 
\end{proof}
\subsection{Proof of Theorem \ref{Thm-Convergence-IS-Krylov-SC}}
Let \(\hat{S}_k\) be a sampling
matrix drawn from the probability space \((\Omega,\mathcal F,\mathbf{P})\). 
We define the auxiliary quantities
\begin{equation}\label{eq-hat-xk+1}
	\left\{
	\begin{aligned}
		\hat{d}^k
		&=-P A_{\mathcal I_r}^{\top}
		\hat{S}_k \hat{S}_k^\top
		(A_{\mathcal I_r}x^k-b_{\mathcal I_r}),\\
		\hat{p}^k
		&=\hat{d}^k-\sum_{i=j_{k,\ell}}^{k-1}
		\frac{\langle \hat{d}^k,p^i\rangle}{\lVert p^i\rVert_2^2}p^i,\\
		\hat{x}^{k+1}
		&=x^k+\hat{\delta}_k\hat{p}^k,
	\end{aligned}
	\right.
\end{equation}
where 
\begin{equation}\label{eq-hat-delta}
	\hat{\delta}_k
	=
	\begin{cases}
		\frac{
			\|\hat{S}_k^\top(A_{\mathcal I_r}x^k-b_{\mathcal I_r})\|_2^2
		}{
			\|\hat{p}^k\|_2^2
		},
		& \text{if }\hat{S}_k\in\mathcal Q_k,\\[1mm]
		0,
		& \text{if }\hat{S}_k\in\mathcal Q_k^c,
	\end{cases}
\end{equation}
and $\mathcal{Q}_k$ is defined in \eqref{eq-Q_k}. 
\begin{proof}[Proof of Theorem \ref{Thm-Convergence-IS-Krylov-SC}]
	We first demonstrate that $Ax^k=b$ implies that $x^k=A^\dagger b$. 
	Suppose that $Ax^k=b$ and $Ax^i\neq b$ for $i=0,1,\ldots,k-1$. 
	Since $d^k\in\operatorname{Range}(PA_{\mathcal{I}_r}^\top)$ for $k\geq0$, we obtain $p^k\in\operatorname{Range}(PA_{\mathcal{I}_r}^\top)$. 
	Hence, all iterates $x^k$ of Algorithm \ref{Algo-2} belong to $x^0+\operatorname{Range}(PA_{\mathcal{I}_r}^\top)$, which ensures that 
	$$\left(I-(A_{\mathcal{I}_r}P)^\dagger A_{\mathcal{I}_r}P\right)x^k=\left(I-(A_{\mathcal{I}_r}P)^\dagger A_{\mathcal{I}_r}P\right)A_{\mathcal{I}_p}^\dagger b_{\mathcal{I}_p}=A_{\mathcal{I}_p}^\dagger b_{\mathcal{I}_p},$$ 
	where the first equality follows from $x^0=A_{\mathcal{I}_p}^\dagger b_{\mathcal{I}_p}$. 
	Using the assumption $Ax^k=b$, we obtain $A_{\mathcal{I}_r}Px^k=\hat{b}_{\mathcal{I}_r}$ and 
	\begin{equation*}
		\begin{aligned}
			x^k&=(A_{\mathcal{I}_r}P)^\dagger A_{\mathcal{I}_r}Px^k+A_{\mathcal{I}_p}^\dagger b_{\mathcal{I}_p}=(A_{\mathcal{I}_r}P)^\dagger \hat{b}_{\mathcal{I}_r}+A_{\mathcal{I}_p}^\dagger b_{\mathcal{I}_p}=A^\dagger b,
		\end{aligned}
	\end{equation*}
	where the last equality follows from Lemma \ref{lemma-same-solution-set}. 
	
	We next consider the case where $Ax^k\neq b$. 
	From the proof of Theorem \ref{Thm-Convergence-SCRIM}, we obtain 
	\(
	\mathbf P(\hat{S}_k\in\mathcal Q_k\mid\mathcal B_k)>0.
	\)  
	Conditional on \(\mathcal B_k\), Step~6 of Algorithm~\ref{Algo-2} samples
	\(S_k\) according to the conditional distribution of \(\hat{S}_k\) given
	\(\hat{S}_k\in\mathcal Q_k\). Therefore, the random iterate \(x^{k+1}\) has
	the same conditional distribution as the auxiliary iterate \(\hat{x}^{k+1}\)
	conditioned on \(\hat{S}_k\in\mathcal Q_k\). 
	Hence,
	\begin{equation}\label{eq-0517-2}
		\mathbb E_k
		\left[
		\|x^{k+1}-A^\dagger b\|_2^2
		\right]
		=
		\frac{
			\mathbb E_k
			\left[
			\left\|\hat{x}^{k+1}-A^\dagger b\right\|_2^2
			\mathbbm{I}_{\mathcal Q_k}(\hat{S}_k)
			\right]
		}{
			\mathbf P(\hat{S}_k\in\mathcal Q_k\mid\mathcal B_k)
		}.
	\end{equation}
	If the sampling matrix $\hat{S}_k\in\mathcal{Q}_k$, we have 
	\begin{equation}\label{eq-0517-3}
		\begin{aligned}
			\lVert \hat{x}^{k+1}-A^\dagger b\rVert_2^2&=\lVert x^{k}+\hat{\delta}_k \hat{p}^k-A^\dagger b\rVert_2^2\\
			&=\lVert x^{k}-A^\dagger b\rVert_2^2+2\hat{\delta}_k\langle \hat{p}^k,x^k-A^\dagger b\rangle+\hat{\delta}_k^2\lVert \hat{p}^k\rVert_2^2\\
			&=\lVert x^{k}-A^\dagger b\rVert_2^2-2\frac{\lVert \hat{S}_{k}^\top(A_{\mathcal{I}_r}x^k-b_{\mathcal{I}_r})\rVert_2^4}{\lVert \hat{p}^k\rVert_2^2}+\frac{\lVert \hat{S}_{k}^\top(A_{\mathcal{I}_r}x^k-b_{\mathcal{I}_r})\rVert_2^4}{\lVert \hat{p}^k\rVert_2^2}\\
			&=\lVert x^{k}-A^\dagger b\rVert_2^2-\frac{\lVert \hat{S}_{k}^\top(A_{\mathcal{I}_r}x^k-b_{\mathcal{I}_r})\rVert_2^4}{\lVert \hat{p}^k\rVert_2^2},
		\end{aligned}
	\end{equation}
	where the third equality follows from Lemma \ref{lem-pkneq0} and the definition of $\hat{\delta}_k$ in (\ref{eq-hat-delta}). 
	It follows from the equation $\hat{p}^k=\hat{d}^k-\tilde{P}_k\tilde{P}_k^\dagger \hat{d}^k$ in \eqref{iter-pk-2} that $\lVert \hat{p}^k\rVert_2^2=\lVert \hat{d}^k\rVert_2^2-\lVert \tilde{P}_k\tilde{P}_k^\dagger \hat{d}^k\rVert_2^2$ and 
	\begin{equation}\label{eq-0517-4}
		\begin{aligned}
			\frac{\lVert \hat{S}_k^\top(A_{\mathcal{I}_r}x^k-b_{\mathcal{I}_r})\rVert_2^4}{\lVert \hat{p}^k\rVert_2^2}&=\frac{\lVert \hat{S}_k^\top(A_{\mathcal{I}_r}x^k-b_{\mathcal{I}_r})\rVert_2^4}{\lVert \hat{d}^k\rVert_2^2}\frac{\lVert \hat{d}^k\rVert_2^2}{\lVert \hat{d}^k\rVert_2^2-\lVert \tilde{P}_k\tilde{P}_k^\dagger \hat{d}^k\rVert_2^2}.
		\end{aligned}
	\end{equation}
	It follows from the definition of $\hat{d}^k$ that
	\begin{equation}\label{eq-0517-5}
		\begin{aligned}
			\frac{\lVert \hat{S}_k^\top(A_{\mathcal{I}_r}x^k-b_{\mathcal{I}_r})\rVert_2^4}{\lVert \hat{d}^k\rVert_2^2}&=\frac{\lVert \hat{S}_k^\top(A_{\mathcal{I}_r}x^k-b_{\mathcal{I}_r})\rVert_2^4}
			{\lVert PA_{\mathcal{I}_r}^\top \hat{S}_k\hat{S}_k^\top(A_{\mathcal{I}_r}x^k-b_{\mathcal{I}_r})\rVert_2^2}\geq\frac{\lVert \hat{S}_k^\top(A_{\mathcal{I}_r}x^k-b_{\mathcal{I}_r})\rVert_2^2}{\lVert \hat{S}_k^\top A_{\mathcal{I}_r}P\rVert_2^2}.
		\end{aligned}
	\end{equation}
	According to the definition of $q_k$ in (\ref{eq-q_k}), it holds that 
	\begin{equation}\label{eq-0517-6}
		\begin{aligned}
			\frac{\lVert \hat{d}^k\rVert_2^2}{\lVert \hat{d}^k\rVert_2^2-\lVert \tilde{P}_k\tilde{P}_k^\dagger \hat{d}^k\rVert_2^2}&=\left(1-\frac{\lVert \tilde{P}_k\tilde{P}_k^\dagger \hat{d}^k\rVert_2^2}{\lVert \hat{d}^k\rVert_2^2}\right)^{-1}\geq q_k.
		\end{aligned}
	\end{equation}
	Substituting \eqref{eq-0517-4}, \eqref{eq-0517-5} and \eqref{eq-0517-6} into \eqref{eq-0517-3}, we obtain
	\begin{equation}\label{eq-0517-7}
		\begin{aligned}
			\lVert \hat{x}^{k+1}-A^\dagger b\rVert_2^2&\leq\lVert x^{k}-A^\dagger b\rVert_2^2-q_k\frac{\lVert \hat{S}_{k}^\top(A_{\mathcal{I}_r}x^k-b_{\mathcal{I}_r})\rVert_2^2}{\lVert \hat{S}_k^\top A_{\mathcal{I}_r}P\rVert_2^2}.
		\end{aligned}
	\end{equation}
	Since $\mathbb{E}_{k}\left[\lVert x^k-A^\dagger b\rVert_2^2\mathbbm{I}_{\mathcal Q_k}(\hat{S}_k)\right]=
	\mathbf P(\hat{S}_k\in\mathcal Q_k\mid\mathcal B_k)\lVert x^k-A^\dagger b\rVert_2^2$, we have
	\begin{equation}\label{eq-0517-8}
		\begin{aligned}
			&\mathbb{E}_{k}\left[\lVert \hat{x}^{k+1}-A^\dagger b\rVert_2^2\mathbbm{I}_{\mathcal Q_k}(\hat{S}_k)\right]\\
			&\leq
			\mathbf P(\hat{S}_k\in\mathcal Q_k\mid\mathcal B_k)\lVert x^k-A^\dagger b\rVert_2^2-q_k\mathbb{E}_{k}\left[\frac{\lVert \hat{S}_k^\top(A_{\mathcal{I}_r}x^k-b_{\mathcal{I}_r})\rVert_2^2}{\|\hat{S}_k^\top A_{\mathcal I_r}P \|_2^2}\mathbbm{I}_{\mathcal Q_k}(\hat{S}_k)\right].\\
		\end{aligned}
	\end{equation}
	Dividing both sides of \eqref{eq-0517-8} by
	\(\mathbf P(\hat{S}_k\in\mathcal Q_k\mid\mathcal B_k)\) and using
	\eqref{eq-0517-2}, we have
	\begin{equation}\label{eq-0517-9}
		\begin{aligned}
			\mathbb E_k
			\left[
			\|x^{k+1}-A^\dagger b\|_2^2
			\right]&\leq
			\lVert x^k-A^\dagger b\rVert_2^2-\frac{q_k}{\mathbf P(\hat{S}_k\in\mathcal Q_k\mid\mathcal B_k)}\mathbb{E}_{k}\left[\frac{\lVert \hat{S}_k^\top(A_{\mathcal{I}_r}x^k-b_{\mathcal{I}_r})\rVert_2^2}{\|\hat{S}_k^\top A_{\mathcal I_r}P \|_2^2}\mathbbm{I}_{\mathcal Q_k}(\hat{S}_k)\right]\\
			&\leq
			\lVert x^k-A^\dagger b\rVert_2^2-q_k\mathbb{E}_{k}\left[\frac{\lVert \hat{S}_k^\top(A_{\mathcal{I}_r}x^k-b_{\mathcal{I}_r})\rVert_2^2}{\|\hat{S}_k^\top A_{\mathcal I_r}P \|_2^2}\mathbbm{I}_{\mathcal Q_k}(\hat{S}_k)\right]\\
			&=\lVert x^k-A^\dagger b\rVert_2^2-q_k\mathbb E_k\left[
			\frac{
				\|\hat{S}_k^\top(A_{\mathcal I_r}x^k-b_{\mathcal I_r})\|_2^2
			}{
				\|\hat{S}_k^\top A_{\mathcal I_r}P\|_2^2
			}		
			\right],
		\end{aligned}
	\end{equation}
	where the second inequality follows from $0<
	\mathbf P(\hat{S}_k\in\mathcal Q_k\mid\mathcal B_k)
	\leq1$. 
	The last equality follows by applying \eqref{eq-0514-10} with \(\bar S_k\) replaced by \(\hat S_k\), since \(\hat S_k\) is sampled from \((\Omega,\mathcal F,\mathbf P)\). 
	Substituting \eqref{eq-0514-11} into \eqref{eq-0517-9}, we have
	\begin{equation*}
		\begin{aligned}
			\mathbb{E}_k\left[\lVert x^{k+1}-A^\dagger b\rVert_2^2\right]
			&\leq\lVert x^k-A^\dagger b\rVert_2^2-
			q_k\lVert A_{\mathcal{I}_r}Px^k-\hat{b}_{\mathcal{I}_r}\rVert_H^2\\
			&\leq\left(1-
			q_k\sigma_{\min}^2(H^{\frac{1}{2}}A_{\mathcal{I}_r}P)\right)\lVert x^k-A^\dagger b\rVert_2^2,
		\end{aligned}
	\end{equation*}
	where the inequality follows from the positive definiteness of $H$ and $x^k-A^\dagger b\in\operatorname{Range}(PA_{\mathcal{I}_r}^\top)$, which we proceed to verify below by induction.
	When $k=0$, it holds that $x^0-A^\dagger b=A_{\mathcal{I}_p}^\dagger b_{\mathcal{I}_p}-A^\dagger b=-(A_{\mathcal{I}_r}P)^\dagger \hat{b}_{\mathcal{I}_r}\in\operatorname{Range}(PA_{\mathcal{I}_r}^\top)$.
	Assume that $x^{k-1}-A^\dagger b \in \operatorname{Range}(PA_{\mathcal{I}_r}^\top)$. 
	From the proof of Lemma \ref{lem-pkneq0}, we have $p^{k-1} \in \operatorname{Range}(PA_{\mathcal{I}_r}^\top)$. 
	Then it follows that $x^{k}-A^\dagger b=x^{k-1}-A^\dagger b+\delta_{k-1}p^{k-1}\in\operatorname{Range}(PA_{\mathcal{I}_r}^\top)$. 
	This completes the proof of $x^k-A^\dagger b\in\operatorname{Range}(PA_{\mathcal{I}_r}^\top)$ for $k\geq0$. 
	For any $S_k\in\mathcal{Q}_k$, it follows from Lemma \ref{lem-pkneq0} that $p^k=d^k-\tilde{P}_k\tilde{P}_k^\dagger d^k\neq0$. 
	Hence, we have
	$$
	0\leq\frac{\lVert \tilde{P}_k\tilde{P}_k^\dagger d^k\rVert_2^2}{\lVert d^k\rVert_2^2}<1,
	$$ 
	which implies that $q_k\geq 1$. 
	This completes the proof of the theorem. 
\end{proof}
\subsection{Proof of Lemma \ref{lemma-0423} and Theorem \ref{xie-equ-krylov}}
\begin{proof}[Proof of Lemma \ref{lemma-0423}]
	Let $\tilde{x}^{k+1}$ be the optimal solution of \eqref{iter-xk+1-cg2}. 
	If $\ell=1$, then $j_{k,\ell}=k$ for all $k\geq0$. 
	Then, the optimization problem \eqref{iter-xk+1-cg2} reduces to 
	\begin{equation*}
		\min_{x \in x^{k} + \operatorname{span}\{p^k\}} \|x - A^\dagger b\|_2^2,
	\end{equation*}
	which corresponds to \eqref{iter-xk+1-cg}. 
	This leads to $\tilde{x}^{k+1}=x^{k+1}$.
	
	We next demonstrate that $\tilde{x}^{k+1}=x^{k+1}$ for $\ell\geq2$.  
	Define
	\begin{equation*}
		\bar{P}_k := \left( p^{j_{k,\ell}},\, p^{j_{k,\ell}+1},\, \ldots,\, p^k \right) \in \mathbb{R}^{n \times (k - j_{k,\ell}+1)},
	\end{equation*}
	where the vectors $p^i$  are generated by Algorithm~\ref{Algo-2}. 
	Then, the minimization problem \eqref{iter-xk+1-cg2} reduces to solving
	\begin{equation*}
		\begin{aligned}
			s_{k}
			&=\argmin{s\in \mathbb{R}^{k-j_{k,\ell}+1}}\lVert x^{j_{k,\ell}}+\bar{P}_ks-A^\dagger b\rVert_2^2,
		\end{aligned}
	\end{equation*}
	which leads to the update $\tilde{x}^{k+1}=x^{j_{k,\ell}}+\bar{P}_ks_k$. 
	Then $s_k$ satisfies that 
	\begin{equation}\label{eq-opt}
		\begin{aligned}
			\bar{P}_k^\top \bar{P}_ks_k=-\bar{P}_k^\top(x^{j_{k,\ell}}-A^\dagger b).
		\end{aligned}
	\end{equation}
	Since the vectors $p^{j_{k,\ell}}, \ldots, p^k$ are mutually orthogonal, the matrix $\bar{P}_k^\top \bar{P}_k$ is diagonal with entries $\lVert p^i \rVert_2^2$ for $i = j_{k,\ell}, \ldots, k$. 
	Hence, (\ref{eq-opt}) reduces to
	\begin{equation*}
		\begin{aligned}
			\lVert p^i\rVert_2^2(s_k)_{i-j_{k,\ell}+1}=-\langle p^i,x^{j_{k,\ell}}-A^\dagger b\rangle, \ i=j_{k,\ell},\ldots,k.
		\end{aligned}
	\end{equation*}
	From the proof of Lemma \ref{lem-pkneq0}, we know that $\langle p^i,x^k-A^\dagger b\rangle=0,i=j_{k,\ell},\ldots,k-1$. 
	It further implies that 
	\begin{equation*}
		\langle p^i,x^{j_{k,\ell}}-A^\dagger b\rangle=\left\{
		\begin{array}{ll}
			\langle p^i,x^{j_{k,\ell}}-x^k\rangle, & \text{if }  i=j_{k,\ell},\ldots,k-1,\\
			\langle p^k,x^{j_{k,\ell}}-x^k\rangle+\langle p^k,x^k-A^\dagger b\rangle, & \text{if }  i=k.
		\end{array}
		\right.
	\end{equation*}
	For $k\geq0$, it holds that 
	\begin{equation}\label{eq-0517-1}
		x^k=x^{j_{k,\ell}}+x^k-x^{j_{k,\ell}}=x^{j_{k,\ell}}+\sum_{t=j_{k,\ell}}^{k-1}(x^{t+1}-x^t)=x^{j_{k,\ell}}+\sum_{t=j_{k,\ell}}^{k-1}\delta_tp^t.
	\end{equation}
	By the orthogonal property in Lemma \ref{lem-pkneq0}, we have
	\begin{equation*}
		\begin{aligned}
			\langle p^i,x^{j_{k,\ell}}-x^k\rangle=
			\left\{
			\begin{array}{ll}
				-\sum_{t=j_{k,\ell}}^{k-1}\delta_t\langle p^i,p^t\rangle=-\delta_i\lVert p^i\rVert_2^2, & \text{if }  i=j_{k,\ell},\ldots,k-1,\\
				-\sum_{t=j_{k,\ell}}^{k-1}\delta_t\langle p^k,p^t\rangle=0, & \text{if }  i=k.
			\end{array}
			\right.
		\end{aligned}
	\end{equation*}
	Thus, we have 
	\begin{equation*}
		\begin{aligned}
			(s_k)_{i-j_{k,\ell}+1}=
			\left\{
			\begin{array}{ll}
				-\frac{\langle p^i,x^{j_{k,\ell}}-A^\dagger b\rangle}{\lVert p^i\rVert_2^2}=\delta_i, & \text{if }  i=j_{k,\ell},\ldots,k-1,\\
				-\frac{\langle p^k,x^k-A^\dagger b\rangle}{\lVert p^k\rVert_2^2}=\frac{\lVert S_{k}^\top(A_{\mathcal{I}_r}x^k-b_{\mathcal{I}_r})\rVert_2^2}{\lVert p^k\rVert_2^2}=\delta_k, & \text{if }  i=k.
			\end{array}
			\right.
		\end{aligned}
	\end{equation*}
	Then, it follows from the equation \eqref{eq-0517-1} that $$\tilde{x}^{k+1}=x^{j_{k,\ell}}+\sum_{i=j_{k,\ell}}^k\delta_ip^i=x^{j_{k,\ell}}+\sum_{i=j_{k,\ell}}^{k-1}\delta_ip^i+\delta_kp^k=x^k+\delta_kp^k=x^{k+1}.$$ 
	This completes the proof of the lemma. 
\end{proof}
\begin{proof}[Proof of Theorem \ref{xie-equ-krylov}]
	Since $\ell=\infty$, we have $j_{k,\ell}=0$, and hence
	\[
	\Theta_k=x^0+\operatorname{span}\{p^0,p^1,\ldots,p^k\}.
	\]
	Since $\Omega=\{I\}$, the stochastic gradient direction reduces to
	$d^k=-P A_{\mathcal I_r}^{\top}(A_{\mathcal I_r}Px^k-\hat{b}_{\mathcal I_r})$. 
	We first show by induction that $
	\operatorname{span}\{p^0,p^1,\ldots,p^k\}
	=
	\mathcal K_{k+1}
	\left(
	P A_{\mathcal I_r}^{\top}A_{\mathcal I_r}P,\;
	P A_{\mathcal I_r}^{\top}\hat r_{\mathcal I_r}^0
	\right)
	$ for $k\geq0$. 
	When $k=0$, it holds that
	$p^0=d^0=-P A_{\mathcal I_r}^{\top}\hat r_{\mathcal I_r}^0$, and therefore
	\(
	\operatorname{span}\{p^0\}
	=
	\mathcal K_1
	\left(
	P A_{\mathcal I_r}^{\top}A_{\mathcal I_r}P,\;
	P A_{\mathcal I_r}^{\top}\hat r_{\mathcal I_r}^0
	\right).
	\)
	Assume that 
	\(
	\operatorname{span}\{p^0,p^1,\ldots,p^i\}
	=
	\mathcal K_{i+1}
	\left(
	P A_{\mathcal I_r}^{\top}A_{\mathcal I_r}P,\;
	P A_{\mathcal I_r}^{\top}\hat r_{\mathcal I_r}^0
	\right)
	\) for $i=0,1,\ldots,k-1$. 
	Since $\mathcal K_{i+1}
	\left(
	P A_{\mathcal I_r}^{\top}A_{\mathcal I_r}P,\;
	P A_{\mathcal I_r}^{\top}\hat r_{\mathcal I_r}^0
	\right)\subseteq\mathcal K_{k+1}
	\left(
	P A_{\mathcal I_r}^{\top}A_{\mathcal I_r}P,\;
	P A_{\mathcal I_r}^{\top}\hat r_{\mathcal I_r}^0
	\right)$ for $i=0,1,\ldots,k$, 
	we obtain that $p^i\in\mathcal K_{k}
	\left(
	P A_{\mathcal I_r}^{\top}A_{\mathcal I_r}P,\;
	P A_{\mathcal I_r}^{\top}\hat r_{\mathcal I_r}^0
	\right)$. 
	Since
	$x^k-x^0=
	\sum_{i=0}^{k-1}(x^{i+1}-x^i)
	=
	\sum_{i=0}^{k-1}\delta_i p^i
	$, we have
	\[
	x^k-x^0
	\in\operatorname{span}\{p^0,p^1,\ldots,p^{k-1}\}\subseteq
	\mathcal K_k
	\left(
	P A_{\mathcal I_r}^{\top}A_{\mathcal I_r}P,\;
	P A_{\mathcal I_r}^{\top}\hat r_{\mathcal I_r}^0
	\right).
	\]
	Using
	\(
	\hat r_{\mathcal I_r}^k
	=
	\hat r_{\mathcal I_r}^0
	+
	A_{\mathcal I_r}P(x^k-x^0),
	\)
	we obtain
	\[
	\begin{aligned}
		d^k
		&=
		-P A_{\mathcal I_r}^{\top}(A_{\mathcal I_r}Px^k-\hat{b}_{\mathcal I_r})\\  
		&=
		-P A_{\mathcal I_r}^{\top}(A_{\mathcal I_r}Px^0-\hat{b}_{\mathcal I_r})
		-
		P A_{\mathcal I_r}^{\top}A_{\mathcal I_r}P(x^k-x^0)
		\in
		\mathcal K_{k+1}
		\left(
		P A_{\mathcal I_r}^{\top}A_{\mathcal I_r}P,\;
		P A_{\mathcal I_r}^{\top}\hat r_{\mathcal I_r}^0
		\right).
	\end{aligned}
	\]
	From  \eqref{iter-pk}, we know that $p^k=d^k-\sum_{i=j_{k,\ell}}^{k-1}\frac{\langle d^k,p^i\rangle}{\lVert p^i\rVert_2^2}p^i$, therefore 
	\[
	p^k\in
	\mathcal K_{k+1}
	\left(
	P A_{\mathcal I_r}^{\top}A_{\mathcal I_r}P,\;
	P A_{\mathcal I_r}^{\top}\hat r_{\mathcal I_r}^0
	\right).
	\]
	Thus, we obtain
	\(
	\operatorname{span}\{p^0,p^1,\ldots,p^k\}
	\subseteq
	\mathcal K_{k+1}
	\left(
	P A_{\mathcal I_r}^{\top}A_{\mathcal I_r}P,\;
	P A_{\mathcal I_r}^{\top}\hat r_{\mathcal I_r}^0
	\right).
	\)
	Since \(p^0,p^1,\ldots,p^k\) are nonzero and mutually orthogonal by Lemma \ref{lem-pkneq0}, the subspace $\operatorname{span}\{p^0,p^1,\ldots,p^k\}$ has dimension \(k+1\).
	Since the dimension of $\mathcal K_{k+1}
	\left(
	P A_{\mathcal I_r}^{\top}A_{\mathcal I_r}P,\;
	P A_{\mathcal I_r}^{\top}\hat r_{\mathcal I_r}^0
	\right)$ is at most $k+1$, it holds that
	\[
	\operatorname{span}\{p^0,p^1,\ldots,p^k\}
	=
	\mathcal K_{k+1}
	\left(
	P A_{\mathcal I_r}^{\top}A_{\mathcal I_r}P,\;
	P A_{\mathcal I_r}^{\top}\hat r_{\mathcal I_r}^0
	\right),
	\]
	which gives
	\(
	x^0+\operatorname{span}\{p^0,p^1,\ldots,p^k\}
	=
	x^0+
	\mathcal K_{k+1}
	\left(
	P A_{\mathcal I_r}^{\top}A_{\mathcal I_r}P,\;
	P A_{\mathcal I_r}^{\top}\hat r_{\mathcal I_r}^0
	\right).
	\)
	
	We next show that 
	$\mathcal K_{k+1}
	\left(
	P A_{\mathcal I_r}^{\top}A_{\mathcal I_r}P,\;
	P A_{\mathcal I_r}^{\top}\hat r_{\mathcal I_r}^0
	\right)=\mathcal K_{k+1}(\hat A^\top\hat A,\hat A^\top\hat r^0)$. 
	Since $A_{\mathcal{I}_p}x^0=b_{\mathcal I_p}$, we have
	\[
	\begin{aligned}
		\hat A^\top\hat r^0
		&=
		\begin{pmatrix}
			A_{\mathcal I_p}^{\top} & P A_{\mathcal I_r}^{\top}
		\end{pmatrix}
		\begin{pmatrix}
			A_{\mathcal I_p}x^0-b_{\mathcal I_p} \\
			A_{\mathcal I_r}P x^0-\hat b_{\mathcal I_r}
		\end{pmatrix}
		=P A_{\mathcal I_r}^{\top}\hat r_{\mathcal I_r}^0 .
	\end{aligned}
	\]
	From the equation $\hat A^\top\hat A
	=
	A_{\mathcal I_p}^{\top}A_{\mathcal I_p}
	+
	P A_{\mathcal I_r}^{\top}A_{\mathcal I_r}P$ in Lemma \ref{lemma-same-solution-set} and \(A_{\mathcal I_p}P=0\), we obtain
	\[
	\hat A^\top\hat A\hat A^\top\hat r^0
	=
	P A_{\mathcal I_r}^{\top}A_{\mathcal I_r}P
	P A_{\mathcal I_r}^{\top}\hat r_{\mathcal I_r}^0 \text{ and } \hat A^\top\hat AP A_{\mathcal I_r}^{\top}A_{\mathcal I_r}P=\left(P A_{\mathcal I_r}^{\top}A_{\mathcal I_r}P\right)^2.
	\]
	For \(i=2,\ldots,k\), it holds that 
	\[
	\begin{aligned}
		(\hat A^\top\hat A)^{i}\hat A^\top\hat r^0
		&=
		\left(\hat A^\top\hat A\right)^{i-1}
		\hat A^\top\hat A\hat A^\top\hat r^0\\
		&=\left(\hat A^\top\hat A\right)^{i-1}P A_{\mathcal I_r}^{\top}A_{\mathcal I_r}P
		P A_{\mathcal I_r}^{\top}\hat r_{\mathcal I_r}^0=\left(P A_{\mathcal I_r}^{\top}A_{\mathcal I_r}P\right)^{i}
		P A_{\mathcal I_r}^{\top}\hat r_{\mathcal I_r}^0.
	\end{aligned}
	\]
	This implies
	\(
	\mathcal K_{k+1}(\hat A^\top\hat A,\hat A^\top\hat r^0)
	=
	\mathcal K_{k+1}
	\left(
	P A_{\mathcal I_r}^{\top}A_{\mathcal I_r}P,\;
	P A_{\mathcal I_r}^{\top}\hat r_{\mathcal I_r}^0
	\right).
	\)
	Combining the above identities, we obtain
	\[
	x^0+\operatorname{span}\{p^0,p^1,\ldots,p^k\}=x^0+
	\mathcal K_{k+1}
	\left(
	P A_{\mathcal I_r}^{\top}A_{\mathcal I_r}P,\;
	P A_{\mathcal I_r}^{\top}\hat r_{\mathcal I_r}^0
	\right)
	=
	x^0+\mathcal K_{k+1}(\hat A^\top\hat A,\hat A^\top\hat r^0).
	\]
	This completes the proof of the theorem.
\end{proof}
\end{document}